\documentclass[dvipsnames,11pt,reqno,twoside]{amsart}
\usepackage[
    style=alphabetic,
    date=year,
    maxnames=3,
    maxbibnames=99
]{biblatex}
\AtEveryBibitem{%
  \clearfield{url}%
  \clearfield{urldate}%
  \clearfield{urlday}%
  \clearfield{urlmonth}%
  \clearfield{urlyear}%
  \clearfield{doi}%
  \clearfield{isbn}%
  \clearfield{issn}%
}
\usepackage{xcolor}

\usepackage{amssymb,mathtools,amsthm}
\usepackage[a4paper,left=3cm,right=3cm,top=3cm,bottom=3cm]{geometry}
\usepackage[T1]{fontenc}
\usepackage[utf8]{inputenc}
\usepackage{microtype}
\usepackage{enumitem}
\setenumerate[1]{label={\arabic*.},ref={\arabic*}}
\setenumerate[2]{label={\alph*.},ref={\alph*}}
\usepackage[symbol]{footmisc}

\mathtoolsset{showonlyrefs}

\usepackage{xargs}                      
\usepackage[colorinlistoftodos,prependcaption,textsize=tiny]{todonotes}
\newcommandx\work[2][1=]{\todo[linecolor=RoyalBlue,backgroundcolor=RoyalBlue!25,bordercolor=RoyalBlue,#1]{\textsc{todo} #2}}
\newcommandx\comment[2][1=]{\todo[linecolor=OliveGreen,backgroundcolor=OliveGreen!25,bordercolor=OliveGreen,#1]{\textsc{comment} #2}}
\newcommandx\mistake[2][1=]{\todo[linecolor=red,backgroundcolor=red!25,bordercolor=red,#1]{\textsc{mistake} #2}}
\newcommandx\improve[2][1=]{\todo[linecolor=orange,backgroundcolor=orange!25,bordercolor=orange,#1]{\textsc{improve} #2}}
\newcommandx\change[2][1=]{\todo[linecolor=yellow,backgroundcolor=yellow!25,bordercolor=yellow,#1]{\textsc{change} #2}}
\newcommandx\mem[2][1=]{\todo[linecolor=orange,backgroundcolor=orange!25,bordercolor=orange,#1]{\textsc{mem} #2}}
\newcommandx\status[2][1=]{\todo[linecolor=Blue,backgroundcolor=Blue!25,bordercolor=Blue,#1]{\textsc{Status} #2}}

\newcounter{n}
\newcounter{bigthings}
\numberwithin{n}{section}

\theoremstyle{plain}

\newtheorem{biglemma}[bigthings]{Lemma}
\newtheorem{bigtheorem}[bigthings]{Theorem}
\newtheorem{bigobjective}[bigthings]{Objective}
\newtheorem{bigconjecture}[bigthings]{Conjecture}

\newtheorem{lemma}[n]{Lemma}
\newtheorem*{lemma*}{Lemma}
\newtheorem{proposition}[n]{Proposition}

\newtheorem{theorem}[n]{Theorem}
\newtheorem*{claim*}{Claim}
\newtheorem*{assertion*}{Assertion}
\newtheorem*{proposition*}{Proposition}

\theoremstyle{definition}
\newtheorem{definition}[n]{Definition}
\newtheorem*{definition*}{Definition}
\newtheorem{remark}[n]{Remark}

\newtheorem*{example*}{Example}

\usepackage{tikz}
\definecolor{myLightGreen}{RGB}{142,230,182}
\definecolor{myDarkGreen}{RGB}{34,151,87}
\usetikzlibrary{shapes.geometric, arrows}
\tikzstyle{model} = [rectangle, rounded corners,text centered, text width=3.2cm, draw=black, fill=Cyan!20]
\tikzstyle{widemodel} = [rectangle, rounded corners,text centered, text width=5cm, draw=black, fill=Cyan!20]
\tikzstyle{phaselight} = [rectangle, rounded corners,text centered,text width=3cm, draw=black,fill=myLightGreen]
\tikzstyle{phasedark} = [rectangle, rounded corners,text centered,text width=3cm, draw=black,fill=myDarkGreen]

\usepackage{subcaption}

\renewcommand\phi\varphi
\renewcommand\epsilon\varepsilon
\renewcommand\smallsetminus\setminus
\DeclareMathSymbol{\shortminus}{\mathbin}{AMSa}{"39}

\newcommand\cts{{\operatorname{cts}}}
\newcommand\dscrt{{\operatorname{discrete}}}

\newcommand\diffi{{\,\mathrm{d}}}
\newcommand\diff{{\mathrm{d}}}

\newcommand\Support{\operatorname{Support}}
\newcommand\calE{\mathcal{E}}

\newcommand\D{\mathbb D}
\newcommand\E{\mathbb E}

\newcommand\R{\mathbb R}
\newcommand\C{\mathbf{C}}
\renewcommand\P{\mathbb P}
\newcommand\Z{\mathbb Z}
\newcommand\G{\mathbb G}
\newcommand\V{\mathbb V}

\newcommand\calI{\mathcal I}
\newcommand\calA{\mathcal A}
\newcommand\calU{\mathcal U}
\newcommand\calL{\mathcal L}

\newcommand\calC{\mathcal C}

\newcommand\calT{\mathcal T}
\newcommand\calV{\mathcal V}
\newcommand\calN{\mathcal N}
\newcommand\calX{\mathcal X}
\newcommand\calY{\mathcal Y}
\newcommand\calR{\mathcal R}
\newcommand\calS{\mathcal S}
\newcommand\calZ{\mathcal Z}
\newcommand\calQ{\mathcal{Q}}
\newcommand\calW{\mathcal{W}}

\newcommand\connectin[4]{#1\xleftrightarrow{~\text{ $#2\cap#3$ }~}#4}

\newcommand\connect[3]{#1\xleftrightarrow{~\text{ $#2$ }~}#3}

\newcommand\Bound{{\mathbf{Bound}}}
\newcommand\BoundNonneg{{\mathbf{Bound}}_{\geq 0}}
\newcommand\TBound{{\mathbf{TrBound}}}
\newcommand\TBoundNonneg{{\mathbf{TrBound}}_{\geq 0}}

\newcommand\psiRSW[1]{\psi^{\mathrm{RSW}}}
\newcommand\cpush{c_\mathrm{push}}
\newcommand\crenorm{c_\mathrm{renorm}}
\newcommand\cdicho{c_\mathrm{dichot}}

\newcommand\PhiLoc{\operatorname{Loc}[\Phi]}
\newcommand\PhiDeloc{\operatorname{Deloc}[\Phi]}
\newcommand\PhiDecay{\operatorname{Decay}[\Phi]}
\newcommand\PhiPositive{\operatorname{Positive}[\Phi]}
\newcommand\PsiLoc{\operatorname{Loc}[\Psi]}
\newcommand\PsiDeloc{\operatorname{Deloc}[\Psi]}
\newcommand\PsiDecay{\operatorname{Decay}[\Psi]}
\newcommand\PsiPositive{\operatorname{Positive}[\Psi]}

\newcommand\betaeff{ T_\mathrm{eff}}
\newcommand\ceff{T_\mathrm{gap}}

\newcommand\CIRCUIT{\mathcal{C}}

\newcommand\blank{\,\cdot\,}

\newcommand\Var{\operatorname{Var}}
\newcommand\Cov{\operatorname{Cov}}

\newcommand\Sign[1]{\operatorname{Sign}(#1)}

\newcommand\ssubset{\Subset}

\usepackage{hyperref}
\definecolor{colorlinks}{RGB}{0, 24, 168}
\definecolor{colorcites}{RGB}{124, 10, 2}
\hypersetup{
    colorlinks=true,
    linkcolor=colorlinks,
    citecolor=colorcites,
    urlcolor=colorlinks,
    pdfborder={0 0 0}
}

\newcommand\m{\mathfrak{m}}

\begin{document}

\makeatletter
\@namedef{subjclassname@2020}{\textup{2020} Mathematics Subject Classification}
\makeatother

\title[A dichotomy theory for height functions]{A dichotomy theory for the height functions\\of the BKT transition}
\subjclass[2020]{Primary 82B20, 82B41; secondary 82B30}
 \author{Piet Lammers}
\keywords{%
    Height functions,
    renormalisation group,
    coarse-graining,
    phase diagram,
    sharpness,
    effective temperature gap,
    finite-size criterion,
    BKT transition,
    discrete Gaussian model,
    solid-on-solid model,
    lattice models,
    statistical mechanics}
\address{Institut des Hautes \'Etudes Scientifiques (2020--2023)}
\email{lammers@ihes.fr}

\address{CNRS, Sorbonne Université, LPSM (2023--present)}
\email{piet.lammers@cnrs.fr}

\vspace{-1em}

\begin{abstract}
This text considers the discrete height functions associated with the \emph{Berezinskii--Kosterlitz--Thouless transition} (BKT) at slope zero.
Our main results are as follows.
\begin{itemize}
    \item \textbf{Sharpness:} If the model is localised, then the two-point function (covariance) decays exponentially fast in the distance between the points.
    \item \textbf{Effective temperature gap:} If the model is delocalised, then the variance grows at least as $c\log n$, where $n$ is the distance to the boundary and $c>0$ a universal constant not depending on the temperature.  Thus, the effective temperature must jump from $0$ to at least $c$ at the transition point; values in the interval $(0,c)$ are forbidden.
    \item \textbf{Delocalisation at the transition point:} The delocalised phase includes the transition point, in the sense that it is a closed set in the phase diagram in the appropriate topology.
\end{itemize}
These results contribute to the understanding of the regime at and around the transition point which remained largely unexplored.
In a follow-up paper, the sharpness derived here is used to establish that the localisation-delocalisation transition
is equivalent to the BKT transition in the dual XY and Villain models.

\end{abstract}

\maketitle

\vspace{-2em}

\setcounter{tocdepth}{1}
\tableofcontents



\section{Introduction}

\label{section:intro}
\subsection{Preface}

\begin{figure}[b]
    \includegraphics{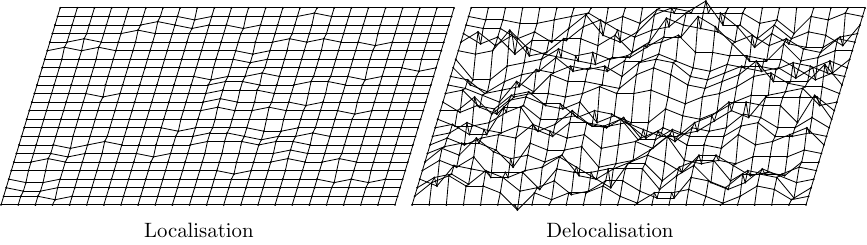}
    \caption{
        \textsc{Left}:
        A sample at low temperature from the discrete Gaussian model.
        The sample looks flat with a few local excitations.
        \textsc{Right}: A sample at high temperature from the same model.
        The surface looks rougher and the heights tend to move away from zero.
    }
    \label{fig:sim}
\end{figure}

Berezinskii~\cite{Berezinskii_1971_DestructionLongrangeOrder,Berezinskii_1972_DestructionLongrangeOrder}, and independently Kosterlitz and Thouless~\cite{KosterlitzThouless_1973_OrderingMetastabilityPhase},
predicted a new of phase transition
in certain two-dimensional spin systems with continuous spin symmetry in the early 1970s (more
precisely: the \emph{XY model} and the \emph{Villain model}).
Such systems were known not to undergo a magnetisation transition~\cite{MerminWagner_1966_AbsenceFerromagnetismAntiferromagnetism},
and the phase transition was predicted to occur on the level of \emph{topological excitations} instead.
The phase transition may be detected by considering the behaviour of the two-point correlation function:
it should decay \emph{exponentially fast} at high temperature, and \emph{polynomially fast} at low temperature.
The
importance of the \emph{Berezinskii--Kosterlitz--Thouless transition} (BKT) was recognised with the 2016 Nobel prize in physics.
The BKT transition is of great interest to researchers in mathematics today,
as is illustrated by a number of beautiful recent works~\cite{GarbanSepulveda_2023_QuantitativeBoundsVortex,GarbanSepulveda_2023_StatisticalReconstructionGFF,AizenmanHarelPeled_2022_DepinningIntegerrestrictedGaussian,EngelenburgLis_2023_ElementaryProofPhase,BauerschmidtParkRodriguez_2024_DiscreteGaussianModel,BauerschmidtParkRodriguez_2024_DiscreteGaussianModela}
discussed below in further detail.

Fröhlich and Spencer marked a major breakthrough in mathematical physics
by rigorously deriving the existence of the BKT transition in their 1981 article~\cite{FrohlichSpencer_1981_KosterlitzThoulessTransitionTwodimensional}.
In that work, the authors analyse the spin system in tandem with its dual object, the \emph{height function}.
This height function, which may be viewed as a Fourier transform of the original spin system,
appears naturally when expanding the partition function of the original model.
For example, the two models share the same partition function,
and the height function is believed to somehow encode the topological excitations of the spin model.
Fröhlich and Spencer prove the existence of the BKT transition by
establishing polynomial decay of the two-point function at low temperature
(it is easy to see that the decay is exponential at high temperature).
On the height function side, they establish \emph{delocalisation} at high temperature.

For the existence of a phase transition, it suffices to detect a qualitative difference
between the perturbative (very high and very low temperature) regimes.
The current paper is the first in a series of articles which aim to make some progress
on the natural follow-up question:
\begin{quote}
    ``Knowing that a phase transition occurs, what happens at and around the transition point?''
\end{quote}

The current article focusses exclusively on the height functions side.
\emph{Height functions} are random integer-valued functions on the vertices
of a lattice graph.
We restrict our attention to two dimensions, and focus mainly on the two-dimensional square lattice graph (Figure~\ref{fig:sim}).
We are interested in height functions with zero slope (or equivalently, zero boundary conditions)
because those boundary conditions naturally show up in the aforementioned duality with the spin models.
Height functions play a pivotal role in two-dimensional statistical mechanics because they are in direct correspondence with a vast range of
other two-dimensional models of varying nature:
not just the XY and Villain model mentioned above, but also, e.g., the 
six-vertex, percolation, Ising, Potts, loop, and dimer models.
An increasingly precise general theory for the analysis of height functions has
emerged over the past fifty years, and they are increasingly studied in their own right.
The current article investigates the nature of the height functions phase transition for a class of potentials, including
the height functions corresponding to the BKT transition.
A priori, we may subdivide our phase diagram into two: either the variance of the height at the origin
is bounded uniformly in the distance to the boundary of the domain: the
\emph{localised phase}, or this variance grows infinite as the domain is taken larger and
larger: the \emph{delocalised phase} (Figure~\ref{fig:sim}). This phase transition is generically
called the \emph{localisation-delocalisation transition}. 
We know that both phases occur because height functions are localised at low temperatures by
an application of the Peierls argument,
while Fröhlich and Spencer established that height functions are delocalised 
at high temperature in their 1981 work~\cite{FrohlichSpencer_1981_KosterlitzThoulessTransitionTwodimensional} on the BKT transition.
A number of precise questions arise.

\begin{bigobjective}\label{objectives}
{\color{white}~}
\begin{itemize}
    \item[Q1.] Is the localised phase one phase, or does it contain multiple subphases?
    \item[Q2.] Is the delocalised phase one phase, or does it contain multiple subphases?
    \item[Q3.] How does the height model behave at and around the transition point?
    \item[Q4.] Is the BKT transition indeed equivalent to the localisation-delocalisation transition?
\end{itemize}
\end{bigobjective}

This article makes partial progress on Questions~1--3.
Questions~1 and~2 are addressed in Theorems~\ref{thm:sharpness_NEW} and~\ref{thm:gap}.
Question~3 is addressed in Theorem~\ref{thm:deloc_at_crit} (and the previous two theorems).
A subsequent article~\cite{Lammers_2023_BijectingBKTTransition},
which relies on Theorem~\ref{thm:sharpness_NEW},
confirms Question~4.
Another follow-up article~\cite{DurandLammers_2026_BKTTransitionSurface} relates the mass (a statistical mechanical quantity separating the two phases)
to the \emph{surface tension} of the model.

At the heart of the current article is a variant of the \emph{renormalisation inequality}
introduced by Duminil-Copin, Sidoravicius, and Tassion in their work on the two-dimensional random-cluster model~\cite{Duminil-CopinSidoraviciusTassion_2017_ContinuityPhaseTransition}.
Our version of it is stated in Lemma~\ref{lemma:preface_second_coarse_graining}.

\subsection{Definitions and main results}

\subsubsection{Definition of height functions and the phase transition}

\begin{definition}[Square lattice graph]
    Let $(\Z^2,\E)$ denote the square lattice graph.  Write
    $\Lambda\ssubset\Z^2$ to say that $\Lambda$ is a finite subset of
    $\Z^2$, and write $\E(\Lambda)\subset\E$ for the set of edges incident to at
    least one vertex in $\Lambda$. Use the shorthand
    $\Lambda_n^\dscrt:=(-n,n)^2\cap\Z^2$ for any $n\in\Z_{\geq 1}$, and observe that
    $\Lambda_n^\dscrt\uparrow\Z^2$ as $n\to\infty$.
\end{definition}

\begin{definition}[Height functions]
    \label{def:height_functions}
    Let $\Omega^\dscrt:=\{h:\Z^2\to\Z\}$ denote the set of height functions, and write
    $\Omega^\dscrt_\Lambda:=\{h\in\Omega^\dscrt:\Support(h)\subset\Lambda\}$ for the height functions whose support is
    contained in $\Lambda\ssubset\Z^2$.  Define the probability measure
    $\mu^\dscrt_\Lambda$ on $\Omega^\dscrt$ to be the unique measure which has
    $\Omega^\dscrt_\Lambda$ as its support and which assigns a probability
    \[
        \mu^\dscrt_\Lambda[\{h\}]:=\frac1{Z_\Lambda}e^{-H_\Lambda(h)};
        \qquad
        H_\Lambda(h):=\sum_{xy\in\E(\Lambda)}V(h_y-h_x)
    \]
    to each height function $h\in\Omega^\dscrt_\Lambda$, where $V:\Z\to\R$ is
    an unbounded convex symmetric \emph{potential function}, and where
    $H_\Lambda$ and $Z_\Lambda$ denote the \emph{Hamiltonian} and
    \emph{partition function } in $\Lambda$ respectively.
    We think of $\mu^\dscrt_\Lambda$ as the law of the height function $h$ in $\Lambda$
    with zero boundary conditions.
\end{definition}

We are typically interested in the law of $h$ in the family
$(\mu^\dscrt_\Lambda)_{\Lambda\ssubset\Z^2}$ for a fixed potential function $V$.
In this article, we restrict to potential functions which are \emph{super-Gaussian}.

\begin{definition}[Super-Gaussian potentials]
    \label{def:pot}
    A \emph{potential function} is an unbounded convex symmetric function $V:\Z\to\R$.
    We say that a potential function $V$ is
    \emph{super-Gaussian} whenever its second derivative
    $V^{(2)}:\Z\to[0,\infty)$ defined by \[V^{(2)}(a):=V(a-1)-2V(a)+V(a+1)\]
    satisfies $V^{(2)}(a+1)\leq V^{(2)}(a)$ for all $a\geq 0$.  Write $\Phi$ for
    the set of super-Gaussian potential functions.  The class $\Phi$ is endowed
    with a natural topology $\calT$, namely the pull-back along the map
    \[
        V
        \mapsto
        \left(
            e^{-V+V(0)},
            \quad
            \lim_{a\to\infty}\tfrac{V(a)}{a}
        \right),
        \]
    where the first
    component of the codomain is endowed with the $\ell^\infty$ topology and the second
    component with the natural topology on the extended number line $[0,\infty]$
    (which makes the extended line connected).  Observe that this topology does not
    distinguish potentials which differ by a constant, which is natural because
    such potentials induce the same probability measures $(\mu^\dscrt_\Lambda)_\Lambda$.
\end{definition}

\begin{example*}
    The following potential functions belong to $\Phi$.
    \begin{itemize}
        \item The \emph{discrete Gaussian model} $V(a)=\beta a^2$ for fixed $\beta>0$,
        \item The \emph{solid-on-solid model} $V(a)=\beta |a|$ for fixed $\beta>0$,
        \item The potential $V(a):=\beta |a|^\gamma$ for fixed $\gamma\in(1,2]$ and $\beta>0$,
        \item The \emph{Poisson potential} $V(a):=-\log I_a(\beta)$ for sufficiently small $\beta>0$.
    \end{itemize}
    In the last example, $I_a$ is the \emph{modified Bessel function};
    we refer to~\cite{EngelenburgLis_2023_ElementaryProofPhase} for a proof.

    In each of these four cases, the potential is continuous in the parameter $\beta$,
    and the potential $V(a):=\beta |a|^\gamma$ is continuous in both $\beta$ and $\gamma$.
\end{example*}

For super-Gaussian potentials it is known that the map $\Lambda\mapsto\Var_{\mu^\dscrt_\Lambda}[h_x]$ is increasing
in $\Lambda$~\cite{LammersOtt_2024_DelocalisationAbsolutevalueFKGSolidonsolid}
(this will also be discussed later in more detail).
In particular, for any fixed potential $V$, the following limit is well-defined:
\[
    \lim_{\Lambda\uparrow\Z^2}\Var_{\mu^\dscrt_{\Lambda}}[h_x]\in[0,\infty].
\]

\begin{definition}[Localisation-delocalisation transition]
    \label{def:loc_deloc}
    We say that $V\in\Phi$ is:
    \begin{itemize}
        \item \emph{Localising} whenever $\lim_{\Lambda\uparrow\Z^2}\Var_{\mu^\dscrt_\Lambda}[h_x]<\infty$,
        \item \emph{Delocalising} whenever $\lim_{\Lambda\uparrow\Z^2}\Var_{\mu^\dscrt_\Lambda}[h_x]=\infty$.
    \end{itemize}
    Write $\PhiLoc$ and $\PhiDeloc$ for the sets of localising and delocalising
    potentials in $\Phi$.
\end{definition}

Either set is nonempty:
\begin{itemize}
    \item $\PhiLoc$ contains $\{V(1)-V(0)\gg 0\}$ by Peierls' argument~\cite{Peierls_1936_IsingsModelFerromagnetism,BrandenbergerWayne_1982_DecayCorrelationsSurface},
    \item $\PhiDeloc$ contains the
    potentials $\beta a^2$ and $\beta |a|$
    for small $\beta$, and $-\log I_a(\beta)$ for large $\beta$;
    see the work of Fröhlich and Spencer on the BKT transition~\cite{FrohlichSpencer_1981_KosterlitzThoulessTransitionTwodimensional}.
\end{itemize}

The paragraphs below, which describe the main results,
are illustrated by Figure~\ref{fig:phase_diagram}.

\begin{figure}
    \includegraphics{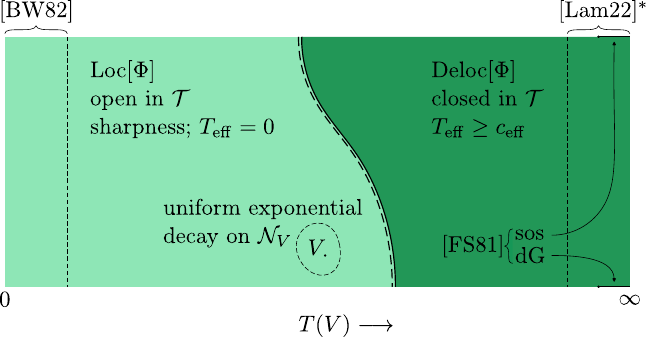}
    \caption{
        Schematic rendering of the main results in the phase diagram,
        realised as the topological space $(\Phi,\calT)$.
        The infinite-dimensional space is projected onto paper in such a way
        that the $x$-coordinate coincides precisely with $T(V):=(V(1)-V(0))^{-1}$.
        Unlike perhaps suggested, it is not proved that $\PhiLoc$ and $\PhiDeloc$
        are connected in this topology.
        Some existing localisation-delocalisation results are drawn:
         localisation
        at low temperature (the Peierls argument~\cite{BrandenbergerWayne_1982_DecayCorrelationsSurface}),
        and delocalisation at high temperature (\cite{FrohlichSpencer_1981_KosterlitzThoulessTransitionTwodimensional}
        for the solid-on-solid and discrete Gaussian models,
        and~\cite{Lammers_2022_HeightFunctionDelocalisation} for all super-Gaussian potentials on the hexagonal and octagonal lattices;
        the result is expected to generalise to the square lattice).
        The current article also applies to the hexagonal and octagonal lattices;
        see Subsection~\ref{subsec:into:sym_gen}.
    }
    \label{fig:phase_diagram}
\end{figure}

\subsubsection{Sharpness}

\emph{Sharpness} means that the localisation-delocalisation transition 
coincides with the onset of exponential decay of the two-point function (covariance)
of the height function.
This suggests that no phase transition occurs \emph{within} the set $\PhiLoc$,
giving a strong indication for the answer to Question~1 in Objective~\ref{objectives}.

\begin{bigtheorem}[Sharpness of the phase transition]
    \label{thm:sharpness_NEW}
    If $V\in\PhiLoc$, then the family $(\mu^\dscrt_\Lambda)_\Lambda$ converges
    to some limit measure $\mu^\dscrt$ on height functions as $\Lambda\uparrow\Z^2$
    in the local convergence topology.
    This limit is ergodic and extremal,
    and $\mu^\dscrt[\{|h_x|\geq \lambda\}]$ decays exponentially fast in $\lambda$.
    The covariance of the height function decays exponentially fast in
    this measure, in the sense that there exists a unique norm $\|\cdot\|_V$ on $\R^2$ such that
    \begin{align}
        \label{eq:norm1}
            &\Cov_{\mu^\dscrt}[h_x;h_y]=e^{-(1+o(1))\|y-x\|_V};
            \\
            \label{eq:norm2}
            &\Cov_{\mu^\dscrt}[\Sign{h_x};\Sign{h_y}]=e^{-(1+o(1))\|y-x\|_V};
    \end{align}
    as $\|y-x\|_2\to\infty$.
\end{bigtheorem}

Such a result was already known for the height function of the six-vertex model with parameters $a=b=1$
and $c\geq 1$ via a connection with the Bethe Ansatz and the random-cluster model (cf.~Subsection~\ref{subsec:intro:developments}).
To the best of the author's knowledge, however, Theorem~\ref{thm:sharpness_NEW} is the first time that
sharpness is proved directly at the level of the height function,
in a setting where explicit calculations are not available.

The norm $\|\cdot\|_V$ is related directly to two classical quantities,
namely the \emph{mass}
and its reciprocal, the \emph{correlation length},
defined respectively by
\begin{equation}
    \label{eq:mass_correlation_length}
    \m(V):=\|(1,0)\|_V;
    \qquad
    \xi(V):=1/\m(V).
\end{equation}
These quantities naturally extend to all of $\Phi$ by setting
\[
    \|\cdot\|_V:\equiv 0;
    \qquad
    \m(V):= 0
    ;
    \qquad
    \xi(V):=\infty,
    \qquad
    \forall V\in\PhiDeloc
\]
(cf.~Remark~\ref{rem:deloc_extension}).
It also means that \emph{sharpness} (Theorem~\ref{thm:sharpness_NEW}) may be phrased differently:
we may simply say that \emph{delocalisation} is equivalent to the \emph{vanishing of the mass}.

\subsubsection{Effective temperature gap}

Let us first state an important conjecture.

\begin{bigconjecture}[GFF convergence]
    Let $\Gamma$ denote the normalised Gaussian free field (GFF), which is a Gaussian process.
    For any $V\in\PhiDeloc$,
    there exists an \emph{effective temperature} $\betaeff(V)\in(0,\infty)$
    such that the gradient of the height function
    (in $\mu_\Lambda^\dscrt$ in the $\Lambda\uparrow\Z^2$ limit)
    has $\betaeff(V)\cdot\Gamma$
    as its scaling limit.

    Moreover, $\betaeff:\PhiDeloc\to(0,\infty)$ is an analytic function
    with respect to a natural analytic structure on $\PhiDeloc$.
\end{bigconjecture}

This conjecture suggests that no phase transition occurs \emph{within} the set $\PhiDeloc$
(Question~2 in Objective~\ref{objectives}).

Proving the above conjecture for \emph{all} delocalised potentials would constitute a major breakthrough
and seems out of reach with current methods
(GFF convergence was recently proved for the square potential at high temperature~\cite{BauerschmidtParkRodriguez_2024_DiscreteGaussianModel,BauerschmidtParkRodriguez_2024_DiscreteGaussianModela}).
However, we prove some interesting results which are consistent with it.
Moreover, we prove an interesting \emph{conditional} result:
if the above conjecture is true, then the effective temperature $\betaeff(V)$
is bounded away uniformly from zero.
This result is called the \emph{effective temperature gap}.

To state the result in its simplest form,
let $H_n$ denote the average of $h$ on $\Lambda_n^\dscrt\setminus\Lambda_{n-1}^\dscrt$;
we interpret this random variable as the discrete equivalent of the 
circle average in the analysis of the Gaussian free field,
and we also observe that $H_1=h_{(0,0)}$.

\begin{bigtheorem}[Effective temperature gap]
    \label{thm:gap}
    There exists a universal constant $\ceff>0$ with the following property.
    If $V\in\PhiDeloc$ is any delocalised potential, then
    \[
        \Var_{\mu^\dscrt_{\Lambda_n^\dscrt}}[H_m]\geq \ceff\cdot \log\tfrac{n}{m}\qquad\forall n\in\mathbb Z_{\geq 8000},\,\forall 1\leq m\leq n/8.
    \]
    In particular, if the model has the Gaussian free field
    $\betaeff(V)\cdot\Gamma$ with effective temperature $\betaeff(V)\in[0,\infty)$
    as its scaling limit,
    then $\betaeff(V)\geq \ceff$.
    Since $\betaeff(V)=0$ for $V\in\PhiLoc$,
    this means that there is a range $(0,\ceff)$ of forbidden values
    for the effective temperature, which we call the
    \emph{effective temperature gap}.
\end{bigtheorem}

We do not derive a general \emph{upper bound} on the variance,
but do remark that in certain cases (such as the square potential),
a logarithmic upper bound can be derived easily via Gaussian domination (cf.~\cite{AizenmanHarelPeled_2022_DepinningIntegerrestrictedGaussian}).

The lower bound is valid for the variance of any 
(weighted) average over the heights at vertices belonging to $\Lambda_m^\dscrt$,
not just circle averages.
The constant $\ceff$ is universal: it does not depend on the choice of the potential
$V$.
The constant seems to encode
a fundamental property of two-dimensional Euclidean space,
and arises naturally from the Russo-Seymour-Welsh theory applied to a percolation representation
of the height functions model.

The effective temperature gap is fundamentally linked to the
 \emph{integer-valued} nature of our height functions (for real-valued height functions, there is no such gap).
One way in which the discrete nature of the height function values plays a role is that high- and low-value regions of the height function are separated by
``cliff lines'' (in the dual graph) which have a height $\leq a$ on one side and a height
$\geq a+1$ on the other side.
For real-valued height functions, such ``cliff lines'' need not exist as they can be smoothed out into a larger
region with a small gradient.

If the height function converges to the GFF
and if the dual model converges to a conformal field theory,
then the value of $\betaeff(V)$ is directly linked to the conformal parameters of the dual theory.
For example, in the six-vertex model, the value of $\betaeff(V)$ is related to the value of 
$\kappa$, where the conformal loop ensemble $\operatorname{CLE}(\kappa)$ conjecturally describes the scaling limit of the corresponding random-cluster model in the 
Baxter--Kelland--Wu coupling.
Rigorous relations have already been established between $\betaeff(V)$ and
the one- and two-arm exponents in the two-dimensional random-cluster model (cf.~\cite{Duminil-CopinKajetanKozlowskiLammers_2026_GaussianFreeField} and references therein).
The value of $\betaeff(V)$ therefore has more meaning than that of a ``simple multiplicative constant'' in the scaling limit.

\subsubsection{Finite-size criterion}

Our results will follow from a coarse-graining inequality applied to an observable
in finite-volume measures $\mu^\dscrt_{\Lambda_n^\dscrt}$.
As we shall see, this implies the following:
the mass $\m(V)$ is strictly positive if and only if
the finite-size observable drops below a fixed treshold value
for some $n$.
Such a result is called a \emph{finite-size criterion}.
Since the finite-size observable is continuous in the choice of the potential
$V\in\Phi$, this leads to the following two results,
related to Question~3 in Objective~\ref{objectives}.

\begin{bigtheorem}[Height functions are delocalised at the phase transition]
    \label{thm:deloc_at_crit}
    The sets $\PhiLoc$ and $\PhiDeloc$ are respectively open and closed in $(\Phi,\calT)$.
\end{bigtheorem}

\begin{bigtheorem}[The mass is locally uniformly positive]
    \label{thm:local_exp}
    Each $W\in\PhiLoc$ admits a neighbourhood $\calN_W$ such that
    \( \inf_{V\in\calN_W}\|\cdot\|_V \) is a norm 
    (in particular, it is positive definite),
    where the norms are those introduced in Theorem~\ref{thm:sharpness_NEW}.
\end{bigtheorem}

In fact, we expect stronger results to hold true,
as is expressed in the following conjecture.

\begin{bigconjecture}[Properties of the mass]
    All of the following hold true:
    \begin{itemize}
        \item The map $V\mapsto \|\cdot\|_V$ is a continuous 
function on $(\Phi,\calT)$,
        \item Its restriction to $\PhiLoc$ is analytic with respect to a natural analytic structure.
    \end{itemize}
\end{bigconjecture}

\subsection{Generalisations}
\label{subsec:into:sym_gen}

The proofs rely in an essential way on the symmetries of the
square lattice.
More precisely, these symmetries are:
\begin{enumerate}
    \item Translational symmetry along a rank-two sublattice,
    \item Flip symmetry around some axis,
    \item Rotational symmetry by an angle of $\pi/2$.
\end{enumerate}
The results extend to models with the same symmetry group.
Let us now give a formal account of this more general setup.
Let $\G=(\V,\E)$ denote a planar graph, and let $V=(V_{xy})_{xy\in\E}\subset\Phi$
denote an assignment of potentials to the edges of the graph.
The \emph{automorphism group} $\operatorname{Aut}(\G,V)$ of the model
contains the set of automorphisms $\phi$ of $\G$ with the property
that $V_{\phi(xy)}=V_{xy}$ for all edges $xy\in\E$.

\subsubsection{To graphs with the same symmetry group}
All results generalise to
models $(\G,V)$ such that $\operatorname{Aut}(\G,V)$
contains a rank-two lattice, a flip symmetry around some axis,
and a rotational symmetry by an angle of $\pi/2$.
This is the precise symmetry group required
in~\cite{Duminil-CopinTassion_2019_RenormalizationCrossingProbabilities} and~\cite{Kohler-SchindlerTassion_2023_CrossingProbabilitiesPlanar}.
For example, this includes models on
the octagonal lattice (also known as the truncated square lattice).
This graph has degree three,
and delocalisation was established for all potentials $V\in\Phi$ with
\(
    T(V)^{-1}:=V(1)-V(0)\leq \log 2
\)~\cite{Lammers_2022_HeightFunctionDelocalisation} (cf.~Figure~\ref{fig:phase_diagram}).

\subsubsection{To graphs with other symmetry groups}
We claim (without giving a formal, complete justification)
that the theory also generalises to models $(\G,V)$ such that
$\operatorname{Aut}(\G,V)$ contains a rank-two sublattice,
a flip symmetry around some axis,
and a rotational symmetry by an angle $\pi/3$.
This includes models on the triangular and hexagonal lattices
(where the hexagonal lattice has degree three, so that~\cite{Lammers_2022_HeightFunctionDelocalisation} applies).

\subsubsection{The XY model and the Poisson potential}

The Poisson potential is of particular importance as it appears in the height function dual to the XY model.
We mentioned that the potential $V_\beta(a):=-\log I_a(\beta)$ satisfies $V_\beta\in\Phi$ for small $\beta>0$.
Although it is not known whether $V_\beta\in\Phi$ for all $\beta>0$, our results still apply
to all $\beta>0$ via a simple workaround, going back to~\cite{EngelenburgLis_2023_ElementaryProofPhase}.
Indeed, we have $e^{-V_{\beta}}\propto e^{-V_{\beta/2}}*e^{-V_{\beta/2}}$,
where $*$ denotes convolution over $\Z$.
Consider now the square lattice graph $\Z^2$,
and the graph $\G$ obtained from $\Z^2$
by replacing each edge $xy$ by two edges linked in series,
and such that the new vertex lies precisely halfway between $x$ and $y$ (see~Figure~\ref{fig:series}).
If $h$ is the random height function on $\G$ with potential $V_{\beta/2}$,
then $h|_{\Z^2}$ has the distribution of the height function on $\Z^2$ with potential $V_{\beta}$.
Thus, even if $V_{\beta}\not\in\Phi$, we may:
\begin{itemize}
    \item Replace the potential $V_\beta$ by $V_{\beta/n}\in\Phi$ (where $n$ is sufficiently large),
    \item Modify the graph by replacing each edge by $n$ edges linked in series,
    \item Run the entire argument, taking into account the above generalisations. 
\end{itemize}
Similar ideas lead to the following generalisation of our main results.

\begin{bigtheorem}
    \label{thm:extension}
    Theorems~\ref{thm:sharpness_NEW}, \ref{thm:gap}, \ref{thm:deloc_at_crit}, and \ref{thm:local_exp}
    extend to the class of potentials $\tilde\Phi$ of \emph{convoluted super-Gaussian potentials},
    that is, potential functions $V$ that may be written
    \[
        e^{-V} = e^{-V_1} *  \cdots * e^{-V_n}
    \]
    for $n\in\Z_{\geq 1}$ and $(V_i)_i\subset \Phi$.
    In particular, $\tilde\Phi$ contains the Poisson potential $V_\beta$ for any $\beta$.
\end{bigtheorem}

The case of the Poisson potential is particularly convenient,
because the edge distribution is infinitely divisible (one may consider the $n\to\infty$ limit),
and this makes the proof of the main results somewhat simpler than for general potentials (see Section~\ref{section:extension_poisson_XY}).

\begin{figure}
  \centering
  \includegraphics{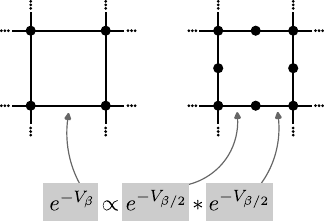}
  \caption{Theorem~\ref{thm:extension}. By replacing each edge by two (or more) edges linked in series,
  we may extend the results from the class $\Phi$ to the class of potentials obtained
  by ``convoluting'' potentials in $\Phi$.}
  \label{fig:series}
\end{figure}

\subsubsection{Other height function models}
The method exhibited here is expected to work for other integer-valued height function models on biperiodic planar graphs,
as soon as they satisfy some form of the absolute-value-FKG inequality.
The key lies in finding a good combinatorial representation of the model
that recovers a sufficient amount of planarity and symmetry to run the percolation arguments (cf.~\cite{GlazmanLammers_2025_DelocalisationContinuity2D,Karrila_2023_LogarithmicDelocalizationRandom},
which appeared after the initial version of this manuscript was posted online).

\subsection{Developments in the field}
\label{subsec:intro:developments}
\subsubsection{The discrete Gaussian model and the delocalised phase}
\label{subsubsec:intro:dG}

The discrete Gaussian model ($V(a)=\beta a^2$) is among
the earliest discrete height functions to appear in the mathematics literature. It is
the natural dual to the Villain model.
The Villain model is a close cousin of the XY model:
it is essentially an XY model with a modified potential to facilitate its analysis.
The model was proved to delocalise with
logarithmic variance growth in~\cite{FrohlichSpencer_1981_KosterlitzThoulessTransitionTwodimensional} (cf.~\cite{KharashPeled_2017_FrohlichSpencerProofBerezinskiiKosterlitzThouless}),
constituting the first height function delocalisation result. That work rigorously establishes
the BKT transition for the XY and Villain models, which was predicted independently by
Berezinskii~\cite{Berezinskii_1971_DestructionLongrangeOrder,Berezinskii_1972_DestructionLongrangeOrder} and later Kosterlitz and
Thouless~\cite{KosterlitzThouless_1973_OrderingMetastabilityPhase}. 
At low temperatures,
height function localisation is relatively easy to prove via the Peierls argument~\cite{BrandenbergerWayne_1982_DecayCorrelationsSurface,Peierls_1936_IsingsModelFerromagnetism}.

Several works on the BKT transition and on the models surrounding it have recently appeared.
Any Gibbs measure of the Villain model may be written as the independent product
of a massless Gaussian free field (the \emph{spin wave})
with a probability measure on vortices and antivortices.
In~\cite{GarbanSepulveda_2023_QuantitativeBoundsVortex}, Garban and Sepúlveda
compare the fluctuations of the model coming from these two components,
and derive that the fluctuations induced by the vortex-antivortex measure are at least of the same order 
of magnitude as those coming from the spin wave.
The vortex-antivortex measure is in direct correspondence with yet another model of interest
to physicists called the \emph{Coulomb gas};
refer to the extensive review of Lewin~\cite{Lewin_2022_CoulombRieszGases} for an overview of all (including recent) developments.
In another work~\cite{GarbanSepulveda_2023_StatisticalReconstructionGFF},
the same authors prove a quantitative lower bound on the delocalisation of the discrete Gaussian model 
when the height function is not $\Z$-valued,
but rather $(a(x)+\Z)$-valued where for each vertex
$x\in\Z^2$ the number $a(x)\in\mathbb R$ denotes an arbitrary constant.
This is remarkable because the setup essentially lacks any symmetry.

Independently of these developments, the link between spin models and height functions
was intensified in two articles:
Aizenman, Harel, Peled, and Shapiro proved that height function delocalisation 
implies polynomial decay for the two-point function in the Villain model~\cite{AizenmanHarelPeled_2022_DepinningIntegerrestrictedGaussian},
and simultaneously Van Engelenburg and Lis proved the equivalent result for the XY model~\cite{EngelenburgLis_2023_ElementaryProofPhase}.
Both articles use the delocalisation result in~\cite{Lammers_2022_HeightFunctionDelocalisation}
as an input;~\cite{AizenmanHarelPeled_2022_DepinningIntegerrestrictedGaussian} also extends this delocalisation proof
to the discrete Gaussian model on the square lattice.

In a breakthrough series of two papers, Bauerschmidt, Park, and
Rodriguez~\cite{BauerschmidtParkRodriguez_2024_DiscreteGaussianModel,BauerschmidtParkRodriguez_2024_DiscreteGaussianModela} prove
convergence to the Gaussian free field of the discrete Gaussian model through
the renormalisation group flow.
Their arguments apply at high temperature
so that the renormalisation group flow is initialised close to the limit point.

\subsubsection{Quantitative delocalisation results}
\label{subsubsec:intro:quant_overview}

Both the original Fröhlich-Spencer proof and the renormalisation strategy 
(exhibited in~\cite{Duminil-CopinHarelLaslier_2022_LogarithmicVarianceHeight} and in this article)
lead to logarithmic delocalisation,
but there are several other routes (depending on the model) that lead to the same result.

\begin{itemize}
    \item \emph{Dimer model.}
    Kenyon used integrable features
    to prove that the scaling limit of the dimer model is the Gaussian free field~\cite{Kenyon_2000_ConformalInvarianceDomino,Kenyon_2001_DominosGaussianFree},
    which was later extended to small perturbations of the dimer model~\cite{GiulianiMastropietroToninelli_2017_HeightFluctuationsInteracting}.
    This result is much stronger than logarithmic delocalisation.
    \item \emph{Square ice.}
    For the square ice model (six-vertex model with uniform parameters $a=b=c=1$),
    delocalisation of the height function
    was first observed by
    Chandgotia, Peled, Sheffield, and Tassy
    in~\cite{ChandgotiaPeledSheffield_2021_DelocalizationUniformGraph},
    which mentions~\cite{Sheffield_2005_RandomSurfaces} as already containing the more general statements
    that imply the result.
    Duminil-Copin, Harel, Laslier, Raoufi, and Ray
    independently
    implemented the dichotomy strategy to quantify the delocalisation as being
    logarithmic~\cite{Duminil-CopinHarelLaslier_2022_LogarithmicVarianceHeight}.

    \item \emph{Six-vertex model with $a=b=1$ and $c\geq 1$.}
    The full phase diagram of this parameter range is now understood.
For  $1\leq c\leq 2$, Duminil-Copin, Karrila, Manolescu, and Oulamara proved
logarithmic delocalisation~\cite{Duminil-CopinKarrilaManolescu_2024_DelocalizationHeightFunction};
they use the Bethe ansatz as an input to derive macroscopic crossing estimates,
then use Russo-Seymour-Welsh theory to turn these estimates into the desired delocalisation result.
The strategy is thus very different from the renormalisation strategy exhibited in this article.
For $c>2$, Glazman and Peled~\cite{GlazmanPeled_2023_TransitionDisorderedAntiferroelectric}
proved that there exist \emph{two} ergodic gradient Gibbs measures
(thus showing that the planar requirement is genuinely necessary for uniqueness).
Their result relies on the Baxter-Kelland-Wu coupling with the critical random-cluster model for $q>4$,
together with the discontinuity result~\cite{Duminil-CopinGagnebinHarel_2021_DiscontinuityPhaseTransition} mentioned above.
There also exist other delocalisation arguments covering part of the interval
$c\in[1,2]$, namely
the transition point $c=2$~\cite{Duminil-CopinGagnebinHarel_2021_DiscontinuityPhaseTransition,Duminil-CopinSidoraviciusTassion_2017_ContinuityPhaseTransition},
the free-fermion point $c=\sqrt2$~\cite{Kenyon_2000_ConformalInvarianceDomino,Kenyon_2001_DominosGaussianFree}
and a small neighbourhood~\cite{GiulianiMastropietroToninelli_2017_HeightFluctuationsInteracting},
and the range $[(2+2^{1/2})^{1/2},2]$~\cite{Lis_2021_DelocalizationSixvertexModel}.

\item \emph{Loop $\operatorname{O}(2)$ model.}
Duminil-Copin, Glazman, Peled, and
Spinka~\cite{Duminil-CopinGlazmanPeled_2020_MacroscopicLoopsLoop} proved the existence of large loops 
in the loop~$\operatorname{O}(n)$ model at the Nienhuis transition point
for $n\in[1,2]$.
This results may be phrased as a height function delocalisation result
(with logarithmic variance) at the point $n=2$;
the corresponding value for $x$ is $x_c=1/\sqrt2$.
The authors use the parafermionic observable to derive macroscopic crossing estimates,
then use a Russo-Seymour-Welsh theory to turn the loop segments so obtained into large loops.
By contrast, Glazman and Manolescu used planar percolation
to prove delocalisation 
for the uniformly random $1$-Lipschitz on the triangular lattice
(that is, the loop~$\operatorname{O}(2)$ model with the parameter $x=1$),
together with the renormalisation strategy to quantify the delocalisation~\cite{GlazmanManolescu_2021_UniformLipschitzFunctions}.
\end{itemize}

\subsubsection{Localisation-delocalisation in higher dimension}

Height functions are expected to localise in dimension $d\geq 3$ in a rather general setting.
This has been proved rigorously for the discrete Gaussian model~\cite{FrohlichSimonSpencer_1976_InfraredBoundsPhase,FrohlichIsraelLieb_1978_PhaseTransitionsReflection}
and the solid-on-solid model~\cite{BricmontFontaineLebowitz_1982_SurfaceTensionPercolation} in dimension three and higher,
as well as for the uniformly random Lipschitz function in sufficiently high dimension~\cite{Peled_2017_HighdimensionalLipschitzFunctions}.

It is a general phenomenon that the critical dimension of a lattice model
changes after introducing a random disorder.
For example, it was proved recently that the two-dimensional Ising model
exhibits exponential decay at all temperatures after the introduction of a 
random disorder~\cite{DingXia_2021_ExponentialDecayCorrelations,AizenmanHarelPeled_2020_ExponentialDecayCorrelations}.
In the domain of height functions, Dario, Harel, and Peled showed that a wide class of 
\emph{real-valued} height function delocalises in dimension $d\leq 4$
after introducing a random disorder~\cite{DarioHarelPeled_2023_RandomfieldRandomSurfaces}.
For the integer-valued discrete Gaussian model they prove however 
that the height function is already localised in dimension $d=3$ when the disorder is weak;
the localisation-delocalisation question is left open for a strong disorder.

\subsubsection{Recent developments}
After a first version of this work was posted on the arXiv, a few other relevant articles appeared online.
Delocalisation of tilted height functions was obtained in~\cite{OttSchweiger_2025_QuantitativeDelocalizationSolidonsolid}.
Karrila~\cite{Karrila_2023_LogarithmicDelocalizationRandom} implemented a dichotomy strategy
for uniformly random Lipschitz functions, although the delocalisation input on the square lattice is still missing.
A qualitative (percolation-oriented) strategy for delocalisation in the six-vertex and loop~$\operatorname{O}(2)$ models with $c\in[1,2]$ and $x\in[x_c,1]$ respectively
was derived in~\cite{GlazmanLammers_2025_DelocalisationContinuity2D}.
The recent work of Van Engelenburg and Lis~\cite{EngelenburgLis_2025_DualityHeightFunctions} intensified the relationship
between the height function and the dual spin model via a duality relation (an exact identity).
The scaling limit of the six-vertex model with $c\in[\sqrt3,2]$ was determined in~\cite{Duminil-CopinKajetanKozlowskiLammers_2026_GaussianFreeField},
which led to the calculation of several arm exponents in the dual two-dimensional random-cluster model~\cite{alpha1,alpha2,qcloseto4,qequal4}.

\subsection{A new version of a renormalisation inequality}

At the heart of our analysis is an \emph{interface coarse-graining inequality}
which relies exclusively on probabilistic (non-integrable) techniques,
and is stated in terms of a percolation representation of the model.
This is necessary because our height function are expected not to be integrable
(unlike, for example, the height function of the six-vertex model).

In 2017, Duminil-Copin, Sidoravicius, and Tassion~\cite{Duminil-CopinSidoraviciusTassion_2017_ContinuityPhaseTransition} 
found a new inequality for the two-dimensional random-cluster model.
They introduce an observable
$a_n(q)\in[0,1]$ which measures
(for a given cluster
weight $q\geq 1$) the likelihood 
    of seeing a macroscopic interface between
the free and the wired clusters at the scale $n$.
They use percolation theory to prove that there exists a
universal constant $\crenorm>0$ such that for any $q$,
\begin{equation}
    \label{eq:original_renormalisation_inequality}
    a_{7n}(q)\leq a_n(q)^2/\crenorm \qquad \forall n.
\end{equation}
The beauty of this equation, called a \emph{renormalisation inequality}, is that
it holds true in all regimes simultaneously (more precisely, the inequality
holds true at and around the transition point $q=4$).
Yet at the same time the
inequality implies a dichotomy;
iterating~\eqref{eq:original_renormalisation_inequality} implies that one of the following two must hold true:
\begin{itemize}
    \item Either $a_n(q)\geq\crenorm$ for all $n$, or
    \item $a_n(q)<\crenorm$ for some $n$, in which case $a_{n}(q)\to0$ stretch-exponentially fast in $n$.
\end{itemize}
The side of the dichotomy depends on $q$
(the former case occurs for $q\leq 4$; the latter for $q>4$).
The asymptotic behaviour of the observable encodes the phase
of the model.
The equation is called a \emph{renormalisation inequality} because it is inspired by
the renormalisation group picture for percolation models, which is a well-developed
theory in physics but which is (in general) hard to make rigorous mathematically.

This renormalisation strategy was employed twice in the context of
height functions: by Glazman and Manolescu~\cite{GlazmanManolescu_2021_UniformLipschitzFunctions} in the
context of the loop~$\operatorname{O}(2)$ model
and by Duminil-Copin, Harel, Laslier, Raoufi,
and Ray for the square ice (the six-vertex model with uniform weights)~\cite{Duminil-CopinHarelLaslier_2022_LogarithmicVarianceHeight}.
In both cases the dichotomy was used at a single point in the phase diagram to better understand the delocalised
height function (and not to study the localisation-delocalisation transition in general).
Thus, in our case, we use the same inequality to derive very different results.

Our variation of the renormalisation inequality is stated below (Equation~\eqref{eq:original_second_coarse_graining}; Lemma~\ref{lemma:preface_second_coarse_graining}).
We call it a \emph{coarse-graining inequality}
because it leads directly to exponential decay of the observable
(rather than to stretch-exponential decay, and without iterating the inequality).
This is due to a new, small but useful optimisation of the percolation argument,
which also applies to the random-cluster model (and the other models mentioned above).
We shall not yet give a formal definition of the observable $p_n(V)$.
Figure~\ref{fig:simplified_observable} gives a simplified definition,
and we should think of $p_n(V)$ as being the natural analogue of $a_n(q)$
in~\cite{Duminil-CopinSidoraviciusTassion_2017_ContinuityPhaseTransition}.

\begin{figure}
    \centering
    \begin{minipage}{0.3\textwidth}
        \centering
        \includegraphics{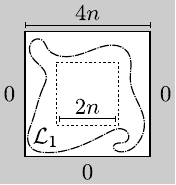}
    \end{minipage}%
    \begin{minipage}{0.45\textwidth}
        \small
        \textsc{Legend.}
        Percolation figures in this article use the following pattern:
        \begin{itemize}[leftmargin=*]
            \item \textbf{White area:} the domain $\Lambda$,
            \item \textbf{Fat lines:} height zero level lines,
            \item \textbf{Fat dashed lines:} height one level lines.
        \end{itemize}
    \end{minipage}
    \caption{
        The (simplified version of the) observable $p_n(V)$
        is defined as the probability of seeing
        a circuit at height $\geq 1$ in the annulus at scale $n$,
        when a boundary height of $0$ is imposed on the outer boundary
        of the annulus.
        The observable measures the ability of the model
        to transition from one height to another at the macroscopic scale.
        The formal definition is more convoluted due to technical complications:
        see Section~\ref{section:main_ingredients444} and Figure~\ref{fig:observable_real}.
    }
    \label{fig:simplified_observable}
\end{figure}

\newcommand\secondcoarsegraininglemmacontents[1]{
    There is a universal constant $\cdicho>0$ with the following property.  For
    any potential $V\in\Phi$, the observables $(p_n(V))_n\subset[0,1]$ which are defined at
    each scale $n\in\Z_{\geq 1}$ satisfy, for each $n\in\Z_{\geq 1000}$, the
    equation
    \begin{equation}
        \label{#1}
        p_{20kn}(V)\leq (p_n(V)/\cdicho)^k\qquad\forall k\in\Z_{\geq 1}.
    \end{equation}
    In particular, for each potential $V\in\Phi$, either $p_n(V)\geq \cdicho$ for all
    $n\in\Z_{\geq 1000}$, or $(p_{kn}(V))_{k\geq 1}$ decays exponentially fast in $k$ for
    some fixed $n\in\Z_{\geq 1}$.
}

\begin{biglemma}[Interface coarse-graining inequality]
    \label{lemma:preface_second_coarse_graining}
    \secondcoarsegraininglemmacontents{eq:original_second_coarse_graining}
\end{biglemma}

This lemma gives rise to an alternative notion of a correlation length:
\begin{equation}
    \label{eq:alternative_correlation_length}
        \xi'(V):=\inf\{n\in\Z_{\geq 1000}: p_n(V)<\cdicho / 2\}\in\Z_{\geq 1000}\cup\{\infty\}.
\end{equation}
Definition~\ref{def:loc_deloc} partitioned the class of potentials $\Phi$ into two sets:
$\PhiLoc$ and $\PhiDeloc$.
The alternative notion of correlation length suggests another bipartition:
\[
    \PhiDecay := \{\xi'<\infty\};
    \qquad
    \PhiPositive :=
    \{\xi'=\infty\} = \{
    V\in\Phi: \inf_{n\geq 1000} p_n(V)\geq \cdicho
    \}.
\]
In our proof, we shall see that the bipartitions coincide:
\begin{equation}
    \label{eq:loc_deloc_equals_decay_positive}
    \PhiLoc = \PhiDecay
    \qquad\text{and}\qquad
    \PhiDeloc = \PhiPositive.
\end{equation}

\subsection{Percolation theory and the height functions of the BKT transition}

The analysis of the height functions of the BKT transition is complicated by two obstacles.
\begin{itemize}
    \item \textbf{Lack of combinatorial structure.}
    Many two-dimensional models in statistical mechanics
    (for example, the random-cluster model, the six-vertex model, and the loop~$\operatorname{O}(2)$ model)
    benefit from a neat combinatorial structure, which allows us to consider different representations of the model,
    duality arguments, and sometimes exact integrability.
    The height functions of the BKT transition do not seem to enjoy such a structure,
    which makes their analysis (using techniques from percolation theory) more challenging.
    \item \textbf{Failure of the Lipschitz constraint.}
    To make things worse, the gradient of our height functions is not even uniformly bounded.
    We shall later consider ``level sets at height $a\in\Z$''.
    Failure of the Lipschitz constraint implies that two level sets at different heights $a\neq a'$ can cross each other
    (violating planarity of the plane) no matter our choice of $a$ and $a'$.
    This is catastrophic from the perspective of two-dimensional percolation theory,
    where planarity is a key ingredient.
    Indeed, even if such violations of planarity are unlikely on the microscopic level,
    their rare occurence may still have an effect on the macroscopic level due to the critical nature of the model.
\end{itemize}
We get around these problems by introducing several percolations which are related in a subtle fashion.
Understanding these relations is the key to deriving the main results.
To illustrate the complexity of the situation:
the ``level lines at height one'' in the definition of the observable $p_n$ (Figure~\ref{fig:simplified_observable})
form a percolation, but this percolation is not even stochastically increasing in the domain
when boundary conditions are imposed at height zero.
Problems of this sort must be circumvented in several places.

The author considers that one contribution of this paper lies in the fact
that these obstacles are overcome, potentially
opening the door to further progress on the BKT transition using tools
from two-dimensional percolation theory.

\subsection{Note on the interface coarse-graining inequality (Lemma~\ref{lemma:preface_second_coarse_graining})}

\begin{figure}[b]
    \begin{tikzpicture}[node distance=1.5cm,->]
    \node (allphases) [model] {all phases};
    \node (subcritical) [model, above of=allphases, xshift=2.5cm, yshift=-0.7cm] {subcritical phase};
    \node (remaining) [model, right of=allphases, xshift=3.5cm] {remaining phases};
    \path[thick] (allphases) edge [bend right=17.6]  (subcritical);
    \path[thick] (allphases) edge node [below] {CG.1.P} (remaining);

    \node (supercritical) [model, above of=remaining, xshift=2.5cm, yshift=-0.7cm] {supercritical phase};
    \node (critical) [model, right of=remaining, xshift=3.5cm] {critical phases};
    \path[thick] (remaining) edge[bend right=17.6]   (supercritical);
    \path[thick] (remaining) edge node [below] {CG.1.D} (critical);

    \node (ccrit) [widemodel, below of=critical, xshift=-0.5cm, yshift=-1.2cm] {continuous critical phase};
    \path[thick] (critical) edge[bend left =5] (ccrit);

    \node (dcrit) [widemodel, below of=critical, xshift=-0.5cm, yshift=-.3cm] {discontinuous critical phase};
    \path[thick] (critical) edge [bend left=8] node [right,xshift=.1cm] {CG.2} (dcrit);

    \node (loc) [phaselight, left of=dcrit, xshift=-5cm] {localised phase};
    \node (deloc) [phasedark, left of=ccrit, xshift=-5cm] {delocalised phase};
    \path[thick] (dcrit) edge node[above]{$\cong$} (loc);
    \path[thick] (loc) edge (dcrit);
    \path[thick] (ccrit) edge node[above]{$\cong$}(deloc);
    \path[thick] (deloc) edge (ccrit);

\end{tikzpicture}
        \caption{
            The coarse-graining inequalities separate the four phases of the random-cluster model 
            through three dichotomies.
            The first coarse-graining inequality has a primal and a dual version.
            In our context, the off-critical phases can be ruled out
            via arguments specific to height functions.
            The second (interface) coarse-graining inequality
            therefore separates the two critical phases
            and describes the localisation-delocalisation
            transition for height functions.
        }
        \label{fig:diagram}
\end{figure}

It appears that the height functions in this text have two phases:
localised and delocalised.
This can be compared to the two-dimensional random-cluster model
with parameters $q\geq 1$ and $p\in[0,1]$.
That model has four phases:
subcritical, supercritical, continuous critical, and discontinuous critical.
Each phase describes different behaviour of the percolations.
In the last phase (discontinuous critical), the behaviour is highly dependant on boundary conditions.

This \emph{quadrichotomy} can be obtained by establishing three dichotomies, see~\cite{Duminil-CopinTassion_2019_RenormalizationCrossingProbabilities}
and Figure~\ref{fig:diagram}.
This article follows a similar pattern, \emph{except that the subcritical and supercritical phases can be ruled
out using ideas specific to height functions} (see Section~\ref{sec:favourable}; in particular Lemmas~\ref{lemma:cg2_new}--\ref{lemma:duality_new}).
This parallel structure suggests that we can relate the two phases of the height functions to the two critical phases of the random-cluster model.
\begin{itemize}
    \item The localised phase is analogous to the discontinuous critical phase in the random-cluster model.
    Indeed, $h$ concentrates around the height at which boundary conditions are imposed;
    different boundary heights lead to different distributions.
    \item The delocalised phase is analogous to the continuous critical phase in the random-cluster model.
    Indeed, imposing boundary conditions at height $0$ leads to similar behaviour as boundary conditions
    at height $1$, and our renormalisation inequality suggests that the mixing occurs at a polynomial rate.  
\end{itemize}

\subsection{Proof organisation}
\label{subsec:proof_organisation}

All arguments are phrased in terms of a percolation structures
related to an \emph{interpolated} version of the height function $h$.
This interpolation is much easier to handle for potentials in the smaller class
of quadratic potentials
\[
    \Psi:= \{V_\beta(a)=\beta a^2 : \beta\in (0,\infty)\}\subset \Phi.
\]
Define the sets $\PsiLoc$, $\PsiDeloc$, $\PsiDecay$, and $\PsiPositive$ by intersecting the corresponding set for $\Phi$
with $\Psi$.
We first execute all our proofs for the class $\Psi$,
then extend to the full class $\Phi$.

\begin{itemize}
    \item \textbf{Part~\ref{part:proof_main_results}} contains the proofs of Theorems~\ref{thm:sharpness_NEW},  \ref{thm:gap}, \ref{thm:deloc_at_crit}, \ref{thm:local_exp} for the class $\Psi$.
    \begin{itemize}
    \item[\emph{Section~\ref{section:lupu}}] describes the interpolation,
    and introduces the percolation structures related to the interpolated height function.
    It also rigorously states the required Fortuin--Kasteleyn--Ginibre (FKG) inequalities
    that we import from~\cite{LammersOtt_2024_DelocalisationAbsolutevalueFKGSolidonsolid}.
    \item[\emph{Section~\ref{section:main_ingredients444}}] provides a formal statement
    of the interface coarse-graining inequality (Lemma~\ref{lemma:preface_second_coarse_graining})
    and of a Russo-Seymour-Welsh (RSW) result. The latter is imported from~\cite{Kohler-SchindlerTassion_2023_CrossingProbabilitiesPlanar}.
    The proof of Lemma~\ref{lemma:preface_second_coarse_graining} is deferred to Part~\ref{part:lemma_proof} (Sections~\ref{sec:statement_of_ingredients}--\ref{sec:cg2proof}).
    \item[\emph{Section~\ref{sec:loc}}] proves that for any $V\in\PsiDecay$,
    the conclusions of Theorem~\ref{thm:sharpness_NEW} hold true.
    This implies already that $\PsiDecay\subset\PsiLoc$,
    but does not yet establish Theorem~\ref{thm:sharpness_NEW} for $\PsiLoc$.
    \item[\emph{Section~\ref{sec:deloc}}] proves that for any $V\in\PsiPositive$,
    the conclusions of Theorem~\ref{thm:gap} hold true.
    This automatically implies that $\PsiPositive\subset\PsiDeloc$.
    Together with the previous item, we obtain $\PsiDecay=\PsiLoc$
    and $\PsiPositive=\PsiDeloc$.
    This formally establishes Theorems~\ref{thm:sharpness_NEW} and~\ref{thm:gap}
    over $\Psi$.
    \item[\emph{Section~\ref{sec:continuity444}}] proves that the finite-volume
    observable $p_n(V)$ appearing in Lemma~\ref{lemma:preface_second_coarse_graining}
    is continuous in $V\in\Psi$.
    This formally establishes Theorems~\ref{thm:deloc_at_crit}--\ref{thm:local_exp}
    for the class $\Psi$.
    \end{itemize}
    \item \textbf{Part~\ref{part:lemma_proof}} contains the proof of Lemma~\ref{lemma:preface_second_coarse_graining} for the class $\Psi$.
    \begin{itemize}
        \item[\emph{Section~\ref{sec:statement_of_ingredients}}] contains the statement of the three main ingredients of the proof of Lemma~\ref{lemma:preface_second_coarse_graining}.
        \item[\emph{Sections~\ref{sec:ingred_I_symmetry}--\ref{sec:push}}] contains the proofs of those three main ingredients.
        \item[\emph{Section~\ref{sec:cg2proof}}] contains the proof of Lemma~\ref{lemma:preface_second_coarse_graining} using the main ingredients.
    \end{itemize}
    \item\textbf{Part~\ref{part:extensions}} extends the results to the whole class $\Phi$.
    \begin{itemize}
        \item[\emph{Section~\ref{section:extension_poisson_XY}}] extends our results to the dual height function of the XY model (the Poisson potential).
        In fact, the Poisson potential has the same infinite divisibility property as the quadratic potential,
        which means that all the proofs extend immediately.
        \item[\emph{Section~\ref{section:extension_general}}] explains how the results extend to all potentials in $\Phi$ by constructing an alternative for the Brownian interpolation used in Parts~\ref{part:proof_main_results}--\ref{part:lemma_proof}.
    \end{itemize}
\end{itemize}

\part{Proof of the main results}
\label{part:proof_main_results}
\section{Percolation representation of height functions}
\label{section:lupu}

We first execute the entire proof for the Gaussian potentials $V(a)=\beta a^2\in\Psi$.
The parameter $\beta$ is fixed except when clearly indicated.
Recall from Definition~\ref{def:height_functions} that
the probability of each configuration $h\in\Omega^\dscrt_\Lambda$ is proportional to
\begin{equation}
    \label{eq:hamiltonian_lupu}
\textstyle
\mu_\Lambda^\dscrt[h]\propto 
    e^{-H_\Lambda(h)}=\prod_{xy\in\E(\Lambda)} e^{-\beta (h_x-h_y)^2}.
\end{equation}

We now informally introduce the idea of \emph{Brownian interpolations};
it is worked out rigorously below.
Each factor $e^{-\beta (h_x-h_y)^2}$ is precisely the heat kernel in $\R\supset\Z$ with diffusion time $t=1/2\beta$.
Such a factor can be interpreted as the partition function of a Brownian bridge of length $t=1/2\beta$ running from $h_x$ to $h_y$.
This enables a very natural construction,
where we enrich the probability measure $\mu_\Lambda^\dscrt$ by sampling,
for each edge $xy\in\E(\Lambda)$,
an independent Brownian bridge of length $t=1/2\beta$ from $h_x$ to $h_y$.
To the best knowledge of the author,
such a construction can be traced back to the work of Berezinskii~\cite[Page~494]{Berezinskii_1971_DestructionLongrangeOrder},
and was later rediscovered in the mathematics literature by Lupu~\cite{Lupu_2016_LoopClustersRandom}.

\begin{remark}
    While this section is written in the context of the square lattice graph,
    all ideas extend readily to general graphs.
    Even the planarity of the graph does not play a role, except in the definition of dual percolations.
\end{remark}

\subsection{Brownian interpolations of height functions}

The Brownian interpolation of a height function has a \emph{continuum domain}
as its natural domain.
We define those first.
Once they are established, we give a formal definition of the sample space,
before formally defining the Brownian interpolation measure.

\begin{definition}[Continuum domains]
    \label{def:continuum_domains}
    Define $\C:=(\Z\times\R)\cup(\R\times\Z)\subset\R^2$.
    It is the \emph{cable graph} associated with the square lattice graph: the union of $\Z^2$ with all line segments connecting nearest neighbour vertices.
    A \emph{continuum domain} 
    (or simply a \emph{domain})
    is a bounded open set $\Lambda\subset\C$
    with finitely many connected components.
    We let $\partial\Lambda$ and $\bar\Lambda$ denote the topological boundary and closure of $\Lambda$ in $\C$ respectively.
    A \emph{boundary height function} is a function $\xi:\partial\Lambda\to\Z$.
    Write $\Bound$ for the set of boundary conditions $(\Lambda,\xi)$,
    and define $\BoundNonneg:=\{(\Lambda,\xi)\in\Bound:\xi\geq 0\}$.
\end{definition}

\begin{definition}[Sample space]
    For any $(\Omega,\xi)\in\Bound$, the natural sample space is
    \[
        \Omega_\Lambda
        :=
        \{h\in\R^{\bar\Lambda}:\text{$h$ is continuous and satisfies $h_x\in\Z$ for $x\in\partial\Lambda\cup(\Lambda\cap\Z^2)$}\}.
    \]
\end{definition}

\begin{definition}[Observation sets]
    We shall define a probability measure on $\Omega_\Lambda$
    by specifying its finite-dimensional marginals.
    For a convenient definition, we want that the marginals are always evaluated
    on $\partial\Lambda\cup(\Lambda\cap\Z^2)$.
    For a given domain $\Lambda$, we therefore introduce the notion of an \emph{observation set},
    which is any finite set $O\subset\bar\Lambda$ that contains the set $\partial\Lambda\cup(\Lambda\cap\Z^2)$.
    
    For a given domain $\Lambda$ and observation set $O$,
    we let $S_O$ denote the partition of $\Lambda\setminus O$ into connected components.
    Notice that all elements $s\in S_O$ are in fact open line segments of Euclidean length
    $\ell(s)\leq 1$.
    The two endpoints are denoted $s_-$ and $s_+$.
\end{definition}

\begin{definition}[Interpolated height function measure]
    Fix a boundary condition $(\Lambda,\xi)\in\Bound$.
    We define $\mu_{\Lambda}^\xi$ as the unique
    probability measure on $\Omega_\Lambda$ such that for any observation set $O$,
    the density of the random variable $h|_O$ on
    \[
        \Z^{\partial\Lambda\cup(\Lambda\cap\Z^2)}\times\R^{(O\setminus(\partial\Lambda\cup(\Lambda\cap\Z^2)))}
    \]
    with respect to the natural Haar measure (or Lebesgue measure) is proportional to
    \begin{equation}
        \label{eq:interpolation_law}
        1[h|_{\partial\Lambda}=\xi]\cdot
        \prod_{s\in S_O} e^{-\frac{\beta}{\ell(s)} (h_{s_+}-h_{s_-})^2}.
    \end{equation}
    We leave it to the reader to verify that the consistency conditions
    for Kolmogorov's extension theorem are satisfied,
    and that the interpolation is indeed Brownian on each line segment
    so that the probability measure may be realised on a sample space of continuous functions.

    We drop the superscript $\xi$ when $\xi\equiv 0$.
\end{definition}

\begin{proposition}[Consistency of the interpolated height function]
    Consider a finite subset $\Lambda\subset\Z^2$,
    and let $\Lambda'\subset\C$ denote the union of $\Lambda$ with the open unit line segments incident to $\Lambda$.
    Then
    \[
        \text{the law of $h|_{\Lambda}$ under $\mu^\dscrt_{\Lambda}$}
        \quad=\quad
        \text{the law of $h|_{\Lambda}$ under $\mu_{\Lambda'}$}
        .
    \]
\end{proposition}

\begin{proof}
    Consider the observation set $O=\Lambda\cup\partial\Lambda'$,
    and observe that Equation~\eqref{eq:interpolation_law}
    equals the law given in Definition~\ref{def:height_functions}.
\end{proof}

From now on, we shall always work in the more general setting of interpolated height functions.
We sometimes wish to study the law conditioned on events of the form $\{h|_T=\zeta\}$,
which may have zero probability.
This is done in the usual way:
via regular conditional probability distributions given by the density in Equation~\eqref{eq:interpolation_law}.
This is equivalent to conditioning on events of the form $\{|h|_T-\zeta|<\epsilon\}$,
and then letting $\epsilon\to 0$.

The rich structure of the interpolation yields some useful results for free.

\begin{theorem}
    \label{theorem:interpolation_properties}
    Let $(\Lambda,\xi)\in\Bound$.
    Fix $T\subset\Lambda$ finite and let $\zeta\in\Z^T$.
    \begin{itemize}
        \item \textbf{Intermediate value theorem.}
        Fix $a\in\R$. 
        The random sets $\{h<a\}$, $\{h>a\}$, and $\{h=a\}$,
        are open, open, and closed subsets of $\bar\Lambda$ respectively,
        and the topological boundary of $\{h<a\}$ and $\{h>a\}$ is contained in $\{h=a\}$.
        \item \textbf{Consistency.}
        We have
        \(
            \mu_{\Lambda}^\xi[\blank|\{h|_T=\zeta\}]=\mu_{\Lambda\setminus T}^{\xi\zeta}
        \), where
     $(\Lambda\setminus T,\xi\zeta)\in\Bound$ is the natural extended boundary condition.
        \item \textbf{Markov property.}
        Let $(\Lambda_i)_i$ denote the connected component decomposition of $\Lambda$,
        and let $\xi_i:=\xi|_{\partial\Lambda_i}$.
        then
        \(
            \mu_{\Lambda}^\xi=\prod_i \mu_{\Lambda_i}^{\xi_i}
        \).
        \item \textbf{Flip symmetry.}
        If $\xi\equiv a$ for $a\in\Z$, then
        \(
            h\) and \( 2a - h\)
        have the same law in $\mu_{\Lambda}^\xi$.
    \end{itemize}
\end{theorem}

\begin{proof}
    Item~1 follows from continuity of $h$;
    Items~2--4 from Equation~\eqref{eq:interpolation_law}.
\end{proof}

\subsection{Level lines, excursions, and explorations}


\begin{definition}[Level lines and excursions]
    Consider a boundary condition $(\Lambda,\xi)\in\Bound$ and a sample (an interpolated height function) $h\in\Omega_\Lambda$.
    Let $a\in\Z$.
    \begin{itemize}
        \item The \emph{level lines} are defined as the set
        \[
            \calL_a^\cts:=\{h=a\}\subset\bar\Lambda.
        \]
        We shall also write
        \[
            \calL_{\leq a}^\cts:=\{h\leq a\};
            \qquad
            \calL_{\geq a}^\cts:=\{h\geq a\},        \]
            for the \emph{lower} and \emph{upper level lines} respectively.
        \item The \emph{excursions} $\calE^\cts_\diamond$ are defined as the complements $\bar\Lambda\setminus\calL_\diamond^\cts$ of the level lines:
        \[
            \calE^\cts_a:=\{h\neq a\};
            \qquad
            \calE^\cts_{\leq a}:=\{h>a\};
            \qquad
            \calE^\cts_{\geq a}:=\{h<a\}.
        \]
        \item The \emph{regional excursions} $\calR^\cts_\diamond(S)$ are defined as follows:
        \[
            \calR^\cts_\diamond(S):=\{x\in \calE^\cts_{\diamond}:\text{the $\calE^\cts_{\diamond}$-connected component of $x$ intersects $S$}\},
        \]
        where $S\subset\partial\Lambda\cup(\Lambda\cap\Z^2)$ is any kind of \emph{starting set}
        (this restriction guarantees that $h$ takes integer values on $S$).
        The dependency on $S$ is sometimes omitted when $S=\partial\Lambda$,
        and in those cases $\calR^\cts_\diamond$ is also called the \emph{boundary excursion}.
        \item Finally, the \emph{restricted regional excursion} $\calR^\cts_\diamond(S,E)$ is defined as:
        \[
            \calR^\cts_\diamond(S,E):=\{x\in \calE^\cts_{\diamond}:\text{the $(\calE^\cts_{\diamond}\cap E)$-conn.\ comp.\ of $x$ intersects $S$}\},
        \]
        where $S\subset\partial\Lambda\cup(\Lambda\cap\Z^2)$ is any kind of \emph{starting set},
        and $E$ is any kind of \emph{restriction set},
        which must be the intersection of $\bar\Lambda$ with a union of closed unit line segments in $\C$ with integer endpoints.
    \end{itemize}
    The superscript ``$\cts$'' signals that these objects are continuous objects (subsets of $\bar\Lambda$).
\end{definition}

These objects are a graphical representation of the height function $h$.
They can be used to geometrically understand interactions between different regions
in the domain and with the boundary.
In practice we often use the following lemma,
describing an \emph{exploration process}.
An important example of such an exploration process is given below.

\begin{lemma}[Exploration process]
    \label{lemma:exploration}
    Consider the following setup:
    \begin{itemize}
        \item Choose a boundary condition $(\Lambda,\xi)\in\Bound$, and study $\mu_{\Lambda}^\xi$,
        \item Choose a starting set $S\subset\partial\Lambda\cup(\Lambda\cap\Z^2)$ to begin the exploration from,
        \item Choose a target $\diamond\in\{\mathord{=}a,\mathord{\leq}a,\mathord{\geq}a\}$ to stop the exploration,
        \item Choose a restriction set $E\subset\bar\Lambda$ to restrict the exploration (often, $E=\bar\Lambda$).
    \end{itemize}

    Let $(\Lambda',\xi')\in\Bound$ denote the boundary condition induced by $\calR_\diamond^\cts(S,E)$,
    defined via
    \[
        \Lambda':=(\Lambda\setminus S)\setminus\bar\calR_\diamond^\cts(S,E),
        \qquad
        \xi':\partial\Lambda'\to\Z,\,x\mapsto h_x.
    \]
    Then all of the following hold true almost surely:
    \begin{itemize}
        \item Conditional on $(\Lambda',\xi')$, the law of $h|_{\bar\Lambda'}$
        is given by $\mu_{\Lambda'}^{\xi'}$,
        \item On $\partial\Lambda'$, the boundary height function $\xi'$ satisfies
        the following properties:
                \begin{equation}
        \begin{cases}
          \xi_x'  = a &  \text{if $x\in\partial\Lambda'\setminus(S\cup\partial\Lambda  \cup\partial E )$,}\\
          \xi_x' \diamond &\text{if $x\in S\setminus \partial E$,}\\
          \text{$\xi'_x=\xi_x$, which can be any integer} & \text{if $x\in\partial\Lambda$,}\\
            \text{$\xi'_x=h_x$, which can be any integer} &\text{if $x\in\partial E\setminus\partial\Lambda$.}
        \end{cases}
    \end{equation}
    \end{itemize}
\end{lemma}

\begin{proof}
    The event $\{(\Lambda',\xi')=(\Delta,\zeta)\}$ may be written
    as the intersection of
    the event 
    \(
        \{\text{$h$ equals $\zeta$ on $\partial\Delta\setminus\partial\Lambda$}\}
    \)
    with an event measurable with respect to $h|_{\bar\Lambda\setminus\bar\Delta}$.
    The first item then follows from the consistency property and the Markov property
    (Theorem~\ref{theorem:interpolation_properties}).

    The second item is not a random statement;
    it can be derived by analysing how $(\Lambda',\xi')$ is (deterministically)
    constructed from the height function and the regional exploration.
\end{proof}

\begin{figure}
  \centering
  \includegraphics{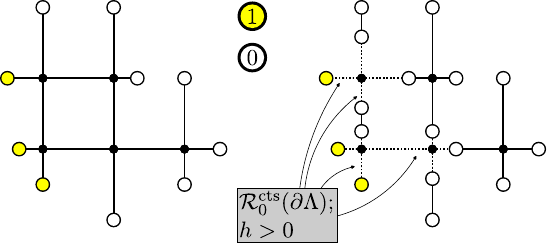}
  \caption{\textsc{Left}: A domain $\Lambda$ with boundary condition $0\leq \xi\leq 1$.
  \textsc{Right}: The exploration of $\calR_0^\cts(\partial\Lambda)$ induces a random domain $\Lambda'$
  with boundary conditions $0$.}
  \label{fig:explorationdef}
\end{figure}

For example, consider some measure $\mu_\Lambda^\xi$. Suppose that $\xi$ equals
$0$ on most of $\partial\Lambda$,
except on a small portion where it equals $1$.
Then $\calR_0^\cts(\partial\Lambda)$ is a random subset of $\bar\Lambda$
emanating from $\{\xi=1\}\subset\partial\Lambda$
(by considering this regional exploration, we implicitly set $E:=\bar\Lambda$).
Define $\Lambda':=\Lambda\setminus\bar\calR_0^\cts(\partial\Lambda)$.
Then the law of $h|_{\bar\Lambda'}$ conditional on $\calR_0^\cts(\partial\Lambda)$ is precisely equal to
$\mu_{\Lambda'}$.
One can informally interpret this as follows:
the random set $\calR_0^\cts(\partial\Lambda)$ ``feels'' $\{\xi=1\}$
 (for example, $h$ is strictly positive on this set),
while the domain $\Lambda'$ ``forgets'' or ``is shielded away''
from $\{\xi=1\}$ (for example, $h$ is flip-symmetric in $\Lambda'$).
See Figure~\ref{fig:explorationdef}.

\subsection{Discrete dual percolations associated with level lines}

Recall that $(\Z^2,\E)$ is the discrete square lattice graph.
We write $(\Z^2)^*:=(\Z+\frac12)^2$ for its dual graph's vertex set,
and let $\E^*$ denote the set of dual edges.
The map $\E\to\E^*,\,xy\mapsto xy^{*}$ is bijective.
The \emph{trace} of an edge $xy\in\E$ is the closed unit line segment in $\C$ connecting $x$ and $y$.

\begin{definition}[Discrete percolations]
    For any continuum domain $\Lambda\subset\C$,
    let $\E(\Lambda)$ denote the set of square lattice edges
    whose trace intersects $\Lambda$.
    Let $\E^*(\Lambda):=\{xy^*:xy\in\E(\Lambda)\}$.
    Introduce the following notations for $\diamond\in\{a,\mathord\leq a,\mathord\geq a\}$:
    \begin{gather}
        \calL_\diamond:=\{xy^*\in\E^*(\Lambda):\text{the trace of the dual edge $xy\in\E(\Lambda)$ intersects $\calL_\diamond^\cts$}\};
        \\
        \calE_\diamond:=\{xy\in\E(\Lambda):\text{the trace of $xy$ does not intersect $\calL_\diamond^\cts$}\};
        \\
        \calR_\diamond(S):=\{xy\in\calE_\diamond:\text{the trace of $xy$ intersects $\calR_\diamond^\cts(S)\setminus\Z^2$}\};
        \\
        \calR_\diamond(S,E):=\{xy\in\calE_\diamond:\text{the trace of $xy$ intersects $\calR_\diamond^\cts(S,E)\setminus\Z^2$}\}.
    \end{gather}
    Notice that $(\calL_\diamond,\calE_\diamond)$
    is a duality pair
    to which standard planar duality arguments in percolation theory apply.
    The percolation $\calR_\diamond(S)$ contains the connected components
    of $\calE_\diamond$ which intersect the starting set $S$.
    The percolation $\calR_\diamond(S,E)$ contains the connected components
    of $\calE_\diamond\cap E$ which intersect the starting set $S$.
\end{definition}

These objects are just discrete versions of the continuous objects defined earlier.

\subsection{Correlation inequalities}

When applying percolation theory arguments,
we shall frequently use a correlation inequality called the 
\emph{Fortuin--Kasteleyn--Ginibre inequality} (FKG)~\cite{FortuinKasteleynGinibre_1971_CorrelationInequalitiesPartially}.
Such an FKG inequality is typically associated with a partial ordering of the state
space.

\begin{definition}[Monotone functions]
    A function is called \emph{increasing} if it preserves the partial ordering,
    and \emph{decreasing} if it flips the partial ordering.
    The partial ordering is sometimes implicit:
    on sets of sets it is $\subset$; on sets of functions it is $\leq$.
\end{definition}

We will use FKG inequalities associated
with \emph{two distinct partial orderings} of our state space:
we consider the height function $h$ with its natural partial ordering,
as well as the \emph{absolute} height function $|h|$ with its natural partial ordering.
One particular random function may be written as an increasing function of \emph{both} representations.

\begin{lemma}[Twice measurability of $\calR^\cts_0(\partial\Lambda)$]
    Let $(\Lambda,\xi)\in\Bound_{\geq 0}$ and consider $\mu_{\Lambda}^\xi$.
    Since all boundary heights are positive, we get
    \(
        \calR^\cts_0(\partial\Lambda)=\calR^\cts_{\leq 0}(\partial\Lambda)
    \).
    In particular, those sets are increasing functions
    of both $h$ and $|h|$,
    and decreasing functions of both
    $\calL_0^\cts$ and $\calL_{\leq 0}^\cts$.
    The same holds true for the discrete regional excursions
    \(
        \calR_0(\partial\Lambda)=\calR_{\leq 0}(\partial\Lambda)
    \).
\end{lemma}

The first FKG inequality concerns $h$.
If $\Lambda$ is any continuum domain and $O$
any observation set, then the density in Equation~\eqref{eq:interpolation_law}
satisfies the FKG lattice condition.
The following lemma is an immediate corollary;
refer to~\cite{FortuinKasteleynGinibre_1971_CorrelationInequalitiesPartially,Sheffield_2005_RandomSurfaces,LammersTassy_2024_MacroscopicBehaviorLipschitz} for details.

\begin{lemma}[Monotonicity for heights]
    \label{lemma:monotonicity}
    Fix a domain $\Lambda$.
    \begin{itemize}
        \item \textbf{FKG inequality.}
        For any two bounded increasing functions $X,Y:\Omega_\Lambda\to\R$ and any boundary height function $\xi$, we have
        \[
            \mu_{\Lambda}^\xi[X Y]\geq \mu_{\Lambda}^\xi[X]\cdot\mu_{\Lambda}^\xi[Y].
        \]
        \item \textbf{Monotonicity in boundary height.} For any
        bounded increasing function $X:\Omega_\Lambda\to\R$
        and any two boundary height functions $\xi,\xi'$, we have
        \[
        \xi\leq\xi'\qquad\implies\qquad
            \mu_{\Lambda}^\xi[X] \leq \mu_{\Lambda}^{\xi'}[X].
        \]
    \end{itemize}
\end{lemma}

The following lemma is a consequence of the same FKG lattice condition.

\begin{lemma}[Log-concavity]
    \label{lemma:log_concavity}
    Let $(\Lambda,\xi)\in\Bound$ and $x\in\Lambda\cap\Z^2$.
    Then the density of the random variable $h_x$ under $\mu_{\Lambda}^\xi$
    with respect to the counting measure on $\Z$
    is log-concave.
\end{lemma}

\begin{proof}
    By the FKG lattice condition, we find that
    \[
        \mu_{\Lambda}^\xi[\{h_x=a\}]\mu_{\Lambda}^{\xi+1}[\{h_x=a+1\}]
        \geq
        \mu_{\Lambda}^\xi[\{h_x=a+1\}]\mu_{\Lambda}^{\xi+1}[\{h_x=a\}].
    \]
    Now shifting $\xi$ by a constant shifts the law of $h$ by the same constant.
    This leads to the desired inequality.
\end{proof}

The second FKG inequality concerns $|h|$,
and was discovered more recently in~\cite{LammersOtt_2024_DelocalisationAbsolutevalueFKGSolidonsolid}.
In fact, it is proved in~\cite{LammersOtt_2024_DelocalisationAbsolutevalueFKGSolidonsolid} that the density of $|h||_O$
satisfies the FKG lattice condition for any observation set $O$.
The following lemma is an immediate corollary.

\begin{lemma}[Monotonicity for absolute heights]
    \label{lemma:monotonicity_abs}
    Fix a domain $\Lambda$.
    \begin{itemize}
        \item \textbf{FKG inequality.}
        For any two bounded functions $X,Y:\Omega_\Lambda\to\R$ which are measurable and increasing in $|h|$ and any boundary height function $\xi\geq 0$, we have
        \[
            \mu_{\Lambda}^\xi[X Y]\geq \mu_{\Lambda}^\xi[X]\cdot\mu_{\Lambda}^\xi[Y].
        \]
        \item \textbf{Monotonicity in boundary height.} For any bounded function $X:\Omega_\Lambda\to\R$ which is measurable and increasing in $|h|$ 
        and any two boundary height functions $\xi,\xi'\geq 0$,
        we have
        \[
            \xi\leq\xi'\qquad\implies\qquad
            \mu_{\Lambda}^\xi[X]\leq \mu_{\Lambda}^{\xi'}[X].
        \]
    \end{itemize}
\end{lemma}

\begin{definition}[Partial ordering on boundary conditions]
    Let $(\Lambda,\xi),(\Lambda',\xi')\in\Bound_{\geq 0}$.
    Then we write
    $(\Lambda,\xi)\preceq (\Lambda',\xi')$ whenever all of the following hold true:
    \[
        \Lambda\subset\Lambda',
        \qquad
        \text{$\xi\equiv 0$ on $\partial\Lambda\setminus\partial\Lambda'$},
        \qquad
        \text{$\xi\leq \xi'$ on $\partial\Lambda\cap\partial\Lambda'$}.
    \]
\end{definition}

\begin{lemma}[Monotonicity in domains]
    \label{lemma:monotonicity_abs_new}
    Fix a domain $\Lambda$ and let $X$ denote a bounded increasing function of $|h||_{\Lambda}$.
    Then for any $(\Lambda,\xi),(\Lambda',\xi')\in\Bound_{\geq 0}$,
    we have
    \[
        (\Lambda,\xi)\preceq (\Lambda',\xi')
        \qquad
        \implies
        \qquad
        \mu_{\Lambda}^\xi[X] \leq \mu_{\Lambda'}^{\xi'}[X].
    \]
\end{lemma}

\begin{proof}
    Define the boundary height function
    \[\zeta:\partial\Lambda'\to\Z,\,x\mapsto 1[x\in\partial\Lambda]\cdot \xi_x.\]
    We claim that
    \[
    \mu_{\Lambda'}^{\xi'}[X]\geq
    \mu_{\Lambda'}^\zeta[X]\geq
    \mu_{\Lambda'}^\zeta[X|\{h|_{\partial\Lambda\setminus\partial\Lambda'}=0\}]
    =
        \mu_{\Lambda}^\xi[X]  .
    \]
    The inequalities are Lemma~\ref{lemma:monotonicity_abs}
    (monotonicity in boundary heights and the FKG inequality respectively).
    The equality is the consistency property and the Markov property.
\end{proof}

The FKG lattice condition can also be used to prove the FKG inequality in measures
conditioned on events that are also distributive sublattices.
This leads to the following useful lemma.

\begin{lemma}
    \label{lemma:monotonicity_abs_conditioned}
    Let $(\Lambda,\xi)\in\Bound_{\geq 0}$ and consider
    a smaller domain $\Lambda'\subset\Lambda$.
    Let $A$ denote any event measurable with respect to the restriction of $|h||_{\bar\Lambda\setminus\bar\Lambda'}$.
    Then, for any bounded increasing function $X$ of $|h||_{\Lambda'}$,
    we have
    \[
        \mu_{\Lambda'}[X]
        \leq
        \mu_{\Lambda}^\xi[X|A]
        .
    \]
\end{lemma}

\begin{proof}
    Without loss of generality, the event $A$ is of the form
    \[
        A=\{|h||_{\bar\Lambda\setminus\bar\Lambda'}=\zeta\}.
    \] 
    By the FKG lattice condition,
    we get
    \[
        \mu_{\Lambda'}[X]
        =
        \mu_{\Lambda}\big[X\big|\{|h||_{\bar\Lambda\setminus\bar\Lambda'}\equiv 0\}\big]
        \leq 
        \mu_{\Lambda}^\xi\big[X\big|\{|h||_{\bar\Lambda\setminus\bar\Lambda'}=\zeta\}\big].
    \]
    This is the desired inequality.
\end{proof}

\section{Formal statements of interface coarse-graining and RSW}
\label{section:main_ingredients444}

The main results in this work are driven by two intermediate results:
an \emph{interface coarse-graining inequality},
and a \emph{Russo--Seymour--Welsh theory}.
This section contains the formal statements.
The interface coarse-graining inequality is proved in  Part~\ref{part:lemma_proof} (Sections~\ref{sec:statement_of_ingredients}--\ref{sec:cg2proof});
the Russo--Seymour--Welsh theory that we rely on was developed in~\cite{Kohler-SchindlerTassion_2023_CrossingProbabilitiesPlanar}. 

\subsection{The interface coarse-graining inequality}
We first introduce the observable in the interface coarse-graining inequality. 
From now on, use the shorthand notations
$\Lambda_n := ((-n,n)\times (-n,n))\cap\C$ and
$\Lambda_{w,h}:=((-w,w)\times(-h,h))\cap\C$.

\begin{definition}[Circuit observable, cf.~Figure~\ref{fig:observable_real}]
    \label{def:circuit_observable}
        Let $\CIRCUIT(\calX,n)$ denote the \emph{circuit event},
        that is, the event that $[-2n,2n]^2\setminus [-n,n] ^2$
    contains a non-contractible $\calX$-circuit.
    The \emph{circuit observable} at scale $n$ is defined via
    \[
        p_n:=\sup_{\Lambda_n\subset \Lambda\subset \Lambda_{2n}}\mu_{\Lambda}
        [\CIRCUIT(\calL_1,n)].
    \]

    The circuit event is realised by a circuit at height $1$,
    while boundary conditions are imposed at height $0$.
    The event in the observable is thus \emph{not} a monotone function of
    $|h|$, and therefore we cannot use monotonicity in domains (Lemma~\ref{lemma:monotonicity_abs_new})
    to easily find a domain $\Lambda$ which realises the supremum.
\end{definition}

\begin{figure}
        \begin{minipage}{0.4\textwidth}
        \centering
    \includegraphics{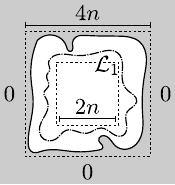}
    \end{minipage}%
    \begin{minipage}{0.45\textwidth}
        \small
        \textsc{Legend.}
        \begin{itemize}[leftmargin=*]
            \item \textbf{White area:} the domain $\Lambda$,
            \item \textbf{Fat lines:} height zero level lines,
            \item \textbf{Fat dashed lines:} height one level lines.
        \end{itemize}
    \end{minipage}
    \caption{
    Illustration of Definition~\ref{def:circuit_observable}.
    The formal definition of the observable is consistent with Figure~\ref{fig:simplified_observable},
    except that we take a \emph{supremum} over domains $\Lambda$.}
    \label{fig:observable_real}
\end{figure}

\begin{lemma*}[Interface coarse-graining inequality]
    \secondcoarsegraininglemmacontents{}
\end{lemma*}

Let us reflect for a moment on the statement in the lemma.
Consider $p_n$, and suppose that the event $\CIRCUIT(\calL_1,n)$
occurs in the probability measure $\mu_{\Lambda}$.
To go from a boundary condition at height $0$ to a level line at height $1$,
the Brownian interpolations must ``cross the gap''.
This leads to an \emph{interface}:
the $\calL_1$-circuit contributing to $\CIRCUIT(\calL_1,n)$
must in fact be surrounded by an $\calI_{01}:=\calL_0\cap\calL_1$-circuit.
If $\Lambda\subset\Lambda_{2n}$, then the event $\calC(\calI_{01},n)$ occurs as well.
This phenomenon is specific to integer-valued height functions,
since the integer conditioning at the vertices forces the gap between $0$ and $1$ to be crossed within single
edges.

Intuitively, one expects that the likelihood of seeing the interface $\calI_{01}$ is:
\begin{itemize}
    \item Exponentially small in the annulus size, in the localised phase,
    \item A function of the annulus conformal radius, in the delocalised phase.
\end{itemize}
The above lemma is consistent with this intuitive understanding of the two phases.
\begin{itemize}
    \item In the localised phase, the interface required to realise the observable 
    $p_{20kn}$ is much longer than $k$ times the interface required to realise
    the observable $p_n$.
    Thus, naturally this probability is smaller (up to the constants $\cdicho$).
    \item In the delocalised phase, the conformal radius of the annulus
    is constant in $n$.
    We naturally expect that $p_n$ converges as $n\to\infty$.
    This is consistent with the lemma whenever $p_n\geq \cdicho$ for all $n$.
\end{itemize}

\subsection{Russo--Seymour--Welsh theory}
\label{subsec:RSW}
Russo--Seymour--Welsh (RSW) theory is concerned with the regularity properties
of planar percolation.
Traditionally, the core objective of the theory is to prove that if rectangles can be crossed in the
``easy'' direction (connecting the long sides with a short path) with a sufficiently high probability,
then the same holds true in the ``difficult'' direction (connecting the short sides with a long path).
Köhler-Schindler and Tassion realised such a theory in a general setting in~\cite{Kohler-SchindlerTassion_2023_CrossingProbabilitiesPlanar}.
The purpose of the current section is to explain how that article (in particular
\cite[Theorem~2]{Kohler-SchindlerTassion_2023_CrossingProbabilitiesPlanar})
applies to our setting.

In this article, we do not work with ``difficult'' and ``easy'' rectangle crossings,
but rather with \emph{circuits} and \emph{arms} in annuli respectively.
Circuits were already introduced in Definition~\ref{def:circuit_observable};
we now introduce arms.

\newcommand\ARM{\mathcal{A}}

\begin{definition}[Arms]
    \label{def:arms}
    Let $\ARM(\calX,n)$ denote the \emph{arm event},
    that is, the event that $[-2n,2n]^2\setminus [-n,n] ^2$
    contains an $\calX$-path connecting the inner boundary and the outer boundary of the annulus.
\end{definition}

Arm should be thought of as ``easy rectangle crossings'',
and circuits should be thought of as ``difficult rectangle crossings''.
Indeed, the arm event is contained in the union of four easy rectangle crossings,
and the circuit events contains the intersection of four difficult rectangle crossings.
We state the main result of~\cite{Kohler-SchindlerTassion_2023_CrossingProbabilitiesPlanar} directly in
the setting of arms and circuits, as this is more natural in our context.

\begin{theorem}[{\cite[Theorem~2]{Kohler-SchindlerTassion_2023_CrossingProbabilitiesPlanar}}]
    \label{thm:generic_RSW}
    There exists a universal constant $N\in\Z_{\geq 1}$ with the following properties.
    Let
    \[
        \psiRSW{N}:[0,1]\to[0,1],\,x\mapsto (1-\sqrt[N]{1-x})^N.
    \]
    Let $\calL,\calE\subset\R^2$ denote any percolation structures defined in $(\mu_\Lambda)_\Lambda$ satisfying the following properties:
    \begin{itemize}
        \item\textbf{Symmetry:} If $\Sigma$ is any symmetry of $\C$, then following laws are the same:
        \[  
        \text{the laws $\Sigma\calL \sim \mu_{\Lambda}$ and $\calL \sim \mu_{\Sigma\Lambda}$,}
        \qquad\text{and}\qquad
        \text{the laws $\Sigma\calE \sim \mu_{\Lambda}$ and $\calE \sim \mu_{\Sigma\Lambda}$.}
        \]
        \item\textbf{Monotonicity:}
        The law of $\calL$ in $\mu_{\Lambda}$ is stochastically decreasing in $\Lambda$,
        and the law of $\calE$ in $\mu_{\Lambda}$ is stochastically increasing in $\Lambda$.
    \end{itemize}
    Then all of the following hold true for any $n,m\in\Z_{\geq 0}$, $k\in\Z_{\geq 1}$, and $s\in\{1,\ldots,k-1\}$,
    \begin{enumerate}
        \item $\mu_{\Lambda_{4n}}[\calC(\calL,n)]
        \geq
        \psiRSW{M}\big(\mu_{\Lambda_{40n}}[\calA(\calL,20n)]\big)$,
        \item $\mu_{\Lambda_{2^{k+2}m}}[
            \calC(\calE,2^{k-s-1}m)
        ]\geq\psiRSW{}(\mu_{\Lambda}[
            \calA(\calE,2^{k-s-1}m)
        ])$ for any $\Lambda\subset\Lambda_{2^km}$,
        \item $\mu_{\Lambda_{8n}}[\calC(\calL_0,n)]\geq \psiRSW{}(\mu_{\Lambda_{20n}}[\calA(\calL_0,n)] )$.
    \end{enumerate}
\end{theorem}

\begin{remark}
    \begin{itemize}
        \item The proof of \cite[Theorem~2]{Kohler-SchindlerTassion_2023_CrossingProbabilitiesPlanar} relies on many applications of the
        FKG inequality and the \emph{square root trick}.
        The square root trick asserts that if $A_1,\dots,A_n$ are increasing events,
        and $\P$ satisfies the FKG inequality, then
        \[
            \max\{\P(A_1),\dots,\P(A_n)\}
            \geq
            1-\sqrt[n]{1-\P(A_1\cup\dots\cup A_n)}.
        \]
        This is why $\psiRSW{N}$ takes the form specified in Theorem~\ref{thm:generic_RSW}.
        \item We choose for circuits and arms
        rather than difficult and easy rectangle crossings (this is more natural in our context).
        Arms lead to easy rectangle crossings via the square root trick (the arm event is contained in the union of a few easy rectangle crossings).
        Difficult rectangle crossings lead to circuits 
        via the FKG inequality (the circuit event contains the intersection of a few difficult rectangle crossing events).
    \end{itemize}
\end{remark}

\begin{proof}[Proof of Theorem~\ref{thm:generic_RSW}]
    Each of the three statements in the theorem is proved in the same way.
    Each statement has two measures, one in a larger domain and one in a smaller domain.
    Since the percolations $\calL$ and $\calE$ are monotone in the domain and behave covariantly under
    symmetries, it is easy to see that the measures satisfy the stochastic domination relations of
    \cite[Theorem~2 (Subsection~6.1)]{Kohler-SchindlerTassion_2023_CrossingProbabilitiesPlanar}.
    The ideas in that paper then apply \emph{mutatis mutandis} to yield the desired inequalities
    (whose form is tailored to our very specific setup).
\end{proof}
\section{Sharpness in the localised regime (Theorem~\ref{thm:sharpness_NEW})}
\label{sec:loc}

In this section, we assume that Lemma~\ref{lemma:preface_second_coarse_graining} holds.
Recall from Subsection~\ref{subsec:proof_organisation} that
our objective is to prove the properties in Theorem~\ref{thm:sharpness_NEW} for all potentials $V\in\PsiDecay$.
Together with the next section, this implies Equation~\eqref{eq:loc_deloc_equals_decay_positive}
(for $\Phi$ replaced with $\Psi$) and establishes
Theorems~\ref{thm:sharpness_NEW} and~\ref{thm:gap} over the class $\Psi$.

We first state an auxiliary lemma.

\begin{lemma}
    \label{lemma:cov}
    Let $\Lambda\subset\C$ be a domain
    and $\alpha$ and $\alpha'$ two 
    probability measures on $\Lambda\cap\Z^2$,
    so that $\alpha(h)$ and $\alpha'(h)$ are random variables.
    Then
    \[
        \Cov_{\mu_{\Lambda}}[\alpha(h);\alpha'(h)]
        =
        \int
        \mu_{\Lambda}
        \big[
            1\{
                \connect{x}{\calE_0}{y}
            \}
            \cdot
            |h_x|
            \cdot
            |h_y|
        \big]
        \diff(\alpha\times\alpha')(x,y).
    \]
    Similarly we have, for any vertices $x,y\in\Lambda\cap\Z^2$,
    \[
        \Cov_{\mu_{\Lambda}}[\Sign{h_x};\Sign{h_y}]
        =
        \mu_{\Lambda}
        \big[\{
                \connect{x}{\calE_0}{y}
            \}
        \big]
    \]
    Since the law of $(|h|,\calE_0)$
    in $\mu_{\Lambda}$
    is increasing in $\Lambda$,
    so are those covariances.
    In particular, they converge (to some number in $[0,\infty]$)
    as $\Lambda\uparrow\C$.
\end{lemma}

\begin{proof}
    The identities follow from flip symmetry (Theorem~\ref{theorem:interpolation_properties}).
    The law of $(|h|,\calE_0)$ is increasing in $\Lambda$ by monotonicity in domains
    (Lemma~\ref{lemma:monotonicity_abs_new}).
\end{proof}

\begin{figure}
  \centering
  \includegraphics{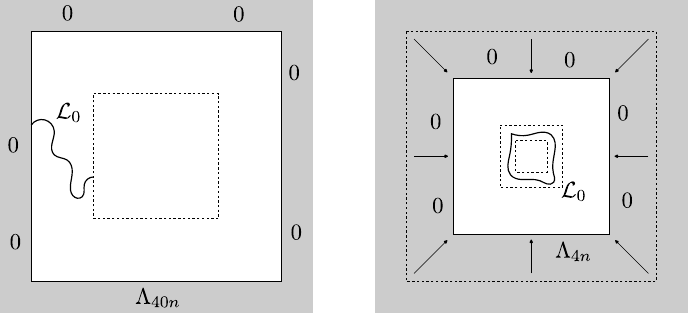}
  \caption{In the key step for sharpness (Lemma~\ref{lemma:sharpness_key_step}),
  we turn the upper bound on $\calL_{\geq 1}$-circuits into a lower bound on $\calL_0$-circuits.
  It uses the RSW theory; we may turn $\calL_0$-arms into $\calL_0$-circuits by
  shrinking the domain and using monotonicity in domains.}
  \label{fig:sharpness}
\end{figure}

The following lemma
turns Lemma~\ref{lemma:preface_second_coarse_graining} into a lower bound on $\calL_0$-circuits.

\begin{lemma}[Key step for sharpness, cf.~Figure~\ref{fig:sharpness}]
    \label{lemma:sharpness_key_step}
    There exists a universal constant $N\in\Z_{\geq 1}$ such that for
    any $n\in\Z_{\geq 1000}$,
    \[
        \mu_{\Lambda_{4n}}[\calC(\calL_0,n)]
        \geq
        1-N\cdot\sqrt[N]{p_{20n}}.
    \]
\end{lemma}

\begin{proof}
    By Theorem~\ref{theorem:interpolation_properties} (the intermediate value theorem) and the definition of $p_{20n}$,
    \[
        \mu_{\Lambda_{40n}}[ \calC(\calL_{\geq 1},20n)]
        =
        \mu_{\Lambda_{40n}}[ \calC(\calL_1,20n)]\leq p_{20n}.
    \]
    If neither $\calC(\calL_{\geq 1},20n)$ nor $\calC(\calL_{\leq -1},20n)$ occurs,
    then the event $\calA(\calL_0,20n)$ must occur, and
    \[
        \mu_{\Lambda_{40n}}[\calA(\calL_0,20n)]\geq 1-2p_{20n}.
    \]
    By Theorem~\ref{thm:generic_RSW}, Statement~1, 
    \[
        \mu_{\Lambda_{4n}}[\calC(\calL_0,n)]
        \geq
        \psiRSW{M}\big(\mu_{\Lambda_{40n}}[\calA(\calL_0,20n)]\big)
        \geq
        \psiRSW{M}(1-2p_{20n})
        .
    \]
    By analysing $\psiRSW{M}$ it is now easy to find the desired constant $N=N(\psiRSW{})$.
\end{proof}

\begin{lemma}
    Recall the constant $N$ from the previous lemma.
    For any $n\in\Z_{\geq 1000}$,
    \[
        \inf_{k\in\Z_{\geq 0}}
        \mu_{\Lambda_{4\cdot 2^kn}}[\calC(\calL_0,n)]
        \geq
        1-2N\cdot\sqrt[N]{p_n/\cdicho}.
    \]
\end{lemma}

\begin{proof}
    If $\sqrt[N]{p_n/\cdicho}\geq 1/2$ then the bound is trivial;
    we focus on the remaining case.
    Fix $k$.
    We first claim that, for any $\ell = 0,\ldots,k$,
    \[
        \mu_{\Lambda_{4\cdot 2^kn}}\big[
            \calC(\calL_0,2^\ell n)
        \big|
            \calC(\calL_0,2^{\ell+1} n)
            \cap \cdots \cap \calC(\calL_0,2^k n)
        \big]
        \geq 1-N\cdot \sqrt[N]{p_{20\cdot 2^\ell n}}.
    \]
    To see that the claim is true,
    first explore the outermost circuit contributing to
    the conditioning event $\calC(\calL_0,2^{\ell+1} n)$.
    Let $\Lambda_{2^{\ell+1} n}\subset\Lambda'\subset\Lambda_{2^{\ell+2} n}$ be the domain inside this circuit.
    Conditional on this exploration, the law inside the circuit
    is $\mu_{\Lambda'}$, independently of the outside.
    Thus, it suffices to prove that
    \[
        \mu_{\Lambda'}[\calC(\calL_0,2^\ell n)]
        \geq
        1-N\cdot \sqrt[N]{p_{20\cdot 2^\ell n}}.
    \]
    This is easy, since monotonicity in domains yields
    \[
        \mu_{\Lambda'}[\calC(\calL_0,2^\ell n)]
        \geq 
        \mu_{\Lambda_{2^{\ell+2} n}}[\calC(\calL_0,2^\ell n)],
    \]
    so that the previous lemma implies the claim.

    The claim implies that
    \[
        \mu_{\Lambda_{4\cdot 2^kn}}[\calC(\calL_0,n)]
        \geq
        1-N\sum_{\ell=0}^k
        \sqrt[N]{p_{20\cdot 2^\ell n}}
        \geq
        1-N\sum_{\ell=0}^\infty
        \sqrt[N]{p_{20\cdot 2^\ell n}}
        ,
    \]
    where the lower bound on the right is independent of $k$.
    Lemma~\ref{lemma:preface_second_coarse_graining}
    yields
    \[
        \mu_{\Lambda_{4\cdot 2^kn}}[\calC(\calL_0,n)]
        \geq
        1-N\sum_{\ell=0}^\infty
        (p_n/\cdicho)^{2^\ell/N}.
    \]
    This may of course be lower bounded by
    \[
        1-N\sum_{\ell=1}^\infty
        (p_n/\cdicho)^{\ell/N}
        =
        1-N\cdot\frac{\sqrt[N]{p_n/\cdicho}}{1-\sqrt[N]{p_n/\cdicho}}.
    \]
    The lemma now follows because $\sqrt[N]{p_n/\cdicho}\leq 1/2$.
\end{proof}

\begin{lemma}
    \label{lemma:annulus_bound_localised}
    There exist universal constants $C<\infty$ and $c>0$
    with the following properties.
    Recall the definition of the alternative correlation length $\xi'$
    from Equation~\eqref{eq:alternative_correlation_length}.
    If $V\in\PsiDecay$, then for any $n\in\Z_{\geq 1}$, we have
    \[
        \lim_{\Lambda\uparrow\C}
        \mu_{\Lambda}[\calC(\calL_0,n\xi')]
        \geq
        1-C e^{-cn}.
    \]
\end{lemma}

\begin{proof}
    This follows from the previous lemma and Lemma~\ref{lemma:preface_second_coarse_graining}.
\end{proof}

\begin{lemma}
    \label{lemma:explicit_localised_annulus_bound}
    If $V\in\PsiDecay$, then the family $(\mu_\Lambda)_\Lambda$ converges
    to some limit measure $\mu$ on height functions as $\Lambda\uparrow\C$
    in the local convergence topology.
    This limit is ergodic and extremal,
    and $\mu[\{|h_x|\geq \lambda\}]$ decays exponentially fast in $\lambda$.
    In particular, $\PsiDecay\subset\PhiLoc$.

    Finally, in $\mu$, we have, for any $n\in\Z_{\geq 1000}$,
    \begin{equation}
        \label{eq:localised_annulus_bound}
        \mu[\calC(\calL_0,n\xi')]
        \geq
        1-C e^{-cn}.
    \end{equation}
\end{lemma}

\begin{proof}
    The law of $h_0$ in $\mu_{\Lambda}$ is log-concave and symmetric around zero.
    Thus, either the probability of $\{h_0=0\}$ goes to zero,
    or it remains uniformly positive as $\Lambda\uparrow\infty$.
    In the former case the law of $h_0$ is not tight and the variance blows up (delocalised phase);
    in the latter case the law of $h_0$ is tight and the variance is uniformly bounded (localised phase).

    Now fix $V\in\PsiDecay$.
    Then the probability of $\calC(\calL_0,n)$ remains uniformly positive for large enough $n$,
    by the previous lemma.
    This means that the probability of $\{h_0=0\}$ also remains uniformly positive:
    we are in the localised phase.
    For the existence of the limit measure $\mu$,
    we argue as in~\cite{LammersOtt_2024_DelocalisationAbsolutevalueFKGSolidonsolid}.
    The law of $|h|$ converges because of monotonicity in domains
    and tightness.
    The law of $h$ given $|h|$ is simply sampled by flipping independent fair coins
    for the signs of the connected components of $\calE_0^\cts=\{h\neq 0\}$.
    Exponential decay of $\mu[\{|h_x|\geq \lambda\}]$
    follows from log-concavity of the law of $h_x$, which is preserved in the local convergence limit.

    As in~\cite{LammersOtt_2024_DelocalisationAbsolutevalueFKGSolidonsolid}, the law of $|h|$ is ergodic and extremal,
    and the law of $h$ is ergodic and extremal if we can prove that
    $\calE_0$ has no infinite connected component $\mu$-almost surely.
    Equation~\eqref{eq:localised_annulus_bound} (which follows from the previous lemma; the bound is preserved in the local convergence limit),
    implies that the event
    \[
        \left\{
            \sum_{n=1}^\infty
            1_{\calC(\calL_0,n)}
            =\infty
        \right\}
    \]
    occurs $\mu$-almost surely, which leads to non-percolation of $\calE_0$.
\end{proof}

We can now finish the proof of Theorem~\ref{thm:sharpness_NEW} for $\PsiDecay$. 

\begin{proof}[Proof of Theorem~\ref{thm:sharpness_NEW} with $\PsiDecay$]
    We want to prove Equations~\eqref{eq:norm1} and~\eqref{eq:norm2},
    the rest was done in the previous lemma.
    We first focus on~\eqref{eq:norm2}.
    Since $\calE_0$ connects vertices whose heights have the same sign,
    we get
    \[
        \Cov_{\mu^\dscrt}[\Sign{h_x};\Sign{h_y}]=\mu\big[\{\connect{x}{\calE_0}{y}\}\big]
        \qquad
        \forall x\neq y.
    \]
    The FKG inequality implies the triangular inequality for
    \[
        m:\Z^2\times\Z^2,\,(x,y)\mapsto-\log\mu\big[\{\connect{x}{\calE_0}{y}\}\big].
    \]
    The existence of the norm $\|\cdot\|_V$ for~\eqref{eq:norm2}
    therefore follows from a standard subadditivity argument
    as soon as the covariance decays exponentially fast in $\|y-x\|_2$.
    This exponential decay follows from Lemma~\ref{lemma:annulus_bound_localised}.

    Now focus on Equation~\eqref{eq:norm1}.
    Clearly $\Cov_{\mu^\dscrt}[h_x;h_y]\geq\Cov_{\mu^\dscrt}[\Sign{h_x};\Sign{h_y}]$.
    To prove~\eqref{eq:norm1} with the same norm $\|\cdot\|_V$,
    it suffices to prove that
    \begin{equation}
        \label{eq:covratio}
        \frac{\Cov_{\mu^\dscrt}[h_x;h_y]}{\Cov_{\mu^\dscrt}[\Sign{h_x};\Sign{h_y}]} 
            =
            \mu\big[
                |h_xh_y|
            \big|
                \{\connect{x}{\calE_0}{y}\}
            \big]
            \leq  e^{o(\|y-x\|_2)}
    \end{equation}
    as $\|y-x\|_2\to\infty$.
    Observe that:
    \begin{enumerate}
        \item We have already established~\eqref{eq:norm2},
        \item The law of $h_x$ in $\mu$ does not depend on $x\in\Z^2$,
        \item The probability $\mu[\{|h_x|>\lambda\}]$ decays exponentially fast in $\lambda$ (by log-concavity). 
    \end{enumerate}
    Equation~\eqref{eq:covratio} follows from these observations
    and a simple calculation for the worst case scenario where
    the random variables $|h_xh_y|$
    and $1\{\connect{x}{\calE_0}{y} \}$
    are maximally correlated.
\end{proof}

\begin{remark}
    \label{remark:correlation_lengths_comparison}
    In fact, the above proof implies that the two notions of correlation length are related via $\xi \leq 2\xi'/c$
    where $c$ is the constant from Lemma~\ref{lemma:annulus_bound_localised}.
\end{remark}

\section{Temperature gap in the delocalised regime (Theorem~\ref{thm:gap})}
\label{sec:deloc}

This section derives Lemma~\ref{lem:deloclem}.
Recall from Subsection~\ref{subsec:proof_organisation} that this implies
Equation~\eqref{eq:loc_deloc_equals_decay_positive}
(for $\Phi$ replaced with $\Psi$) and establishes
Theorems~\ref{thm:sharpness_NEW} and~\ref{thm:gap} over the class $\Psi$.

\begin{lemma}
    \label{lem:deloclem}
    Suppose that $V\in\PsiPositive$.
    Let $n\geq 2000$ and $1\leq m\leq n/2$.
    Let $\alpha$ denote a probability measure on $\Lambda_m\cap\Z^2$
    so that $\alpha(h)$ is a random variable.
    Then
    \[
        \Var_{\mu_{\Lambda_n}}[\alpha(h)]\geq \ceff  \times\log\frac{n}{m}  
    \]
    for some universal constant $\ceff>0$.
\end{lemma}

\begin{proof}
    Fix $m$ and $\alpha$.
    The variance of $\alpha(h)$ is increasing in $n$
    by Lemma~\ref{lemma:cov}.
    Assume therefore, without loss of generality,
    that $m\geq 1000$ and that $n=2^km$
    for $k\in\Z_{\geq 0}$ (indeed, we also consider $k=0$ here).
    Define
    \[
        v_k:=\Var_{\mu_{\Lambda_{2^km}}}[\alpha(h)]
        =\mu_{\Lambda_{2^km}}[\alpha(h)^2]
        \qquad \forall k\in\Z_{\geq 0}.
    \]
    This sequence is nonnegative and nondecreasing.
    We shall use induction to prove that $v_k\geq ck$ for some universal constant $c>0$ not depending on anything
    (the desired result).
    The case $k=0$ is trivial.
    For the induction step, we would like to argue that
    $v_{k+1}\geq v_k+c$,
    by using that the $p_n$ are uniformly bounded below.
    Unfortunately, we cannot quite prove this directly.
    To illustrate the difficulty, we first work out the $k=1$
    and $k=2$ cases, before giving a general argument for $v_k\geq ck$.
    For any $k\in\Z_{\geq 0}$,
    let $\Lambda_{2^km}\subset\Lambda^k\subset\Lambda_{2^{k+1}m}$ denote any domain such that
    \begin{equation}
        \label{eq:deloc_domain_def}
        \mu_{\Lambda^k}[\CIRCUIT(\calL_1,2^km)]\geq \cdicho/2,
    \end{equation}
    which exists by the hypothesis on $(p_n)_n$.

    \emph{The case $k=1$.}
    Since $\mu_{\Lambda}[\alpha(h)^2]$ is increasing in $\Lambda$,
    we get
    \[
        v_1
        \geq 
        \mu_{\Lambda^0}[\alpha(h)^2]
        \geq
        \frac{\cdicho}{2}\cdot \mu_{\Lambda^0}[\alpha(h)^2|\CIRCUIT(\calL_1,m)]
        \geq \frac{\cdicho}{2}.
    \]
    For the last inequality,
    we observe that by exploring the outermost $\calL_1$-circuit from the boundary,
    we find that the conditional law inside this circuit is flip-symmetric around the height $1$.
    This implies that the conditional expectation of $\alpha(h)$
    is exactly equal to $1$,
    which means that the conditional expectation of $\alpha(h)^2$ is at least $1$.

    \emph{The case $k=2$ and the difficulties that arise.}
    We would like to argue as before that
    \begin{multline}
        \label{eq:naive_induction}
        v_2
        \geq 
        \mu_{\Lambda^1}[\alpha(h)^2]
        \geq
        \mu_{\Lambda^1}[\CIRCUIT(\calL_1,2m)]
        \mu_{\Lambda^1}[\alpha(h)^2|\CIRCUIT(\calL_1,2m)]
        +
        R
        \\\geq
        \mu_{\Lambda^1}[\CIRCUIT(\calL_1,2m)]
        (1+v_1)+R,
    \end{multline}
    where
    \[
        R:=
        \mu_{\Lambda^1}[\CIRCUIT(\calL_1,2m)^\complement]
        \mu_{\Lambda^1}[\alpha(h)^2|\CIRCUIT(\calL_1,2m)^\complement]
        .
    \]
    Equation~\eqref{eq:naive_induction} is easy to justify, but will not lead to the desired conclusion.
    Indeed, to extract the desired inequality from the above,
    we would like to argue that
    \begin{equation}
        \label{eq:desired_step}
    \mu_{\Lambda^1}[\alpha(h)^2|\CIRCUIT(\calL_1,2m)^\complement]
    \stackrel{?}{\geq}v_1
    \qquad 
    \implies
    \qquad
    v_2 \stackrel{?}{\geq} v_1 + \frac{\cdicho}{2}.
    \end{equation}
    The problem is that the conditioning event in Equation~\eqref{eq:desired_step}
    is \emph{not} a measurable event in terms of $|h|$,
    and consequently we do not really have any control over the conditional expectation.
    In the absence of Equation~\eqref{eq:desired_step},
    we could use the bound $R\geq 0$,
    but this does not lead to the linear growth in $k$ we are looking for.
    We therefore propose another argument with a more intricate induction step.

    \begin{assertion*}[Induction step]
        Let $\phi>1$ denote the golden ratio.
        Then there exists a universal constant $\tilde c>0$
        such that for any $k\in\Z_{\geq 2}$,
        we may find some $s\in\{1,\ldots,k-1,\infty\}$ such that:
    \begin{enumerate}
        \item\label{case:onef} If $s<\infty$, then $v_{k+2}\geq v_{k-s-1}+ \tilde c\phi^{s}$,
        \item\label{case:twof} If $s=\infty$, then $v_k\geq(\cdicho/2) \phi^{k-1}$,
    \end{enumerate}
    where $\phi$ is the golden ratio.
    \end{assertion*}

    This assertion may look mysterious,
    but it is easy to see that it implies the desired result,
    together with the base cases $k=0$ and $k=1$ established above
    and monotonicity of $(v_k)_k$.
    To establish the assertion, we proceed as follows.
    Fix $k\in\Z_{\geq 2}$ and consider $\mu_{\Lambda^{k-1}}$
    (which is defined such that~\eqref{eq:deloc_domain_def} holds true).
    We start at the largest scale (namely $k$), and ask the following question:
    \begin{quote}
        ``If the event $\CIRCUIT(\calL_1,2^{k-1}m)$ occurs, does a similar event that \emph{is} increasing in $|h|$, also occur
        with a sufficiently good probability?''
    \end{quote}
    If ``yes'', then we can use a more complicated version of the argument for the case $k=2$ above.
    If ``no'', then we ask the same question at the \emph{smaller} scale $k-1$,
    but with $\calL_1$ replaced by $\calL_{F_1}$ where $F_1>1$ is a larger integer.
    By repeating this procedure over the scales,
    either we answer ``yes'' at some scale (Case~\ref{case:onef} in the assertion),
    or we reach the smallest scale, and conclude in a different way
    (Case~\ref{case:twof}).

    We now start the formal derivation.
    Let $(F_\ell)_{\ell\geq -2}$ denote the Fibonacci sequence
    started from $F_{-2}=0$ and $F_{-1}=1$.
    Let $\phi=(1+\sqrt{5})/2= 1.618\mathord{\ldots}$ denote the golden ratio,
    and observe that $F_\ell\geq \phi^\ell$ for all $\ell\geq 0$.
    Now let $\epsilon>0$ denote an extremely small universally fixed constant
    whose precise value is determined later.
    Recall the definition of $\Lambda^{k-1}\subset\Lambda_{2^km}$,
    and define
    \[
        s:=
        \inf\left\{
            \ell\in\{0,\ldots,k-1\}:
            \mu_{\Lambda^{k-1}}[\calC(\calL_{\geq F_\ell},2^{k-\ell-1}m)]
            <
            (\cdicho/2)(1-\epsilon)^\ell
        \right\}
        .
    \]
    This value belongs to $\{0,\ldots,k-1,\infty\}$,
    and in fact we have $s\geq 1$ due to Equation~\eqref{eq:deloc_domain_def}.
    We think of $s$ as the number of scales we must consider before answering
    ``yes'' to the question above.
    We now want to prove the two cases in the assertion
    (and we must also fit the correct values for $\epsilon$ and $\tilde c$).

    \emph{The case $s=\infty$.}
    This case is easy.
    By definition of $s$, we get
    \[
        \mu_{\Lambda^{k-1}}[\calC(\calL_{\geq F_{k-1}},m)]
            \geq 
            (\cdicho/2)(1-\epsilon)^{k-1}.
    \]
    Exactly like in the case $k=1$ above, we get
    \begin{multline}
        v_k
        \geq
        \mu_{\Lambda^{k-1}}[\alpha(h)^2]
        \geq
        \mu_{\Lambda^{k-1}}[\calC(\calL_{\geq F_{k-1}},m)]
        \mu_{\Lambda^{k-1}}[\alpha(h)^2|\calC(\calL_{\geq F_{k-1}},m)]
        \\
        \geq 
        (\cdicho/2)(1-\epsilon)^{k-1}F_{k-1}^2\geq (\cdicho/2)\phi^{k-1},
    \end{multline}
    where the final inequality holds true if we choose $\epsilon>0$ so small that $(1-\epsilon)\phi\geq 1$.

    \begin{figure}
      \centering
      \includegraphics{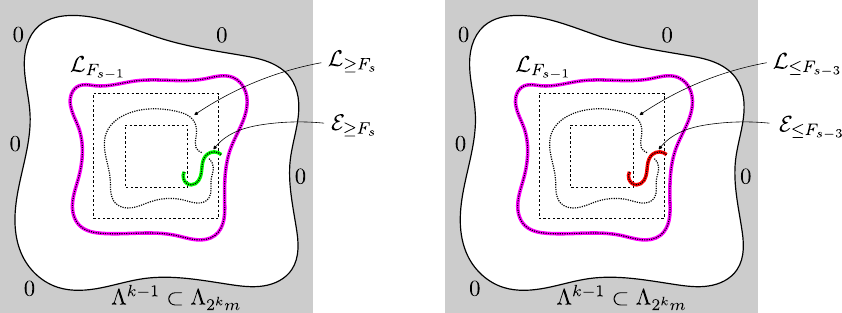}
      \caption{Flipping argument in the delocalisation proof.
      \textsc{Left}:
        When the inner circuit fails to close, there must be a dual arm.
        We use a lower bound on the probability of this event as an input.
        \textsc{Right}: By flipping around $F_{s-1}$, we obtain a lower 
        bound on the probability of an $\calE_{\leq F_{s-3}}$-arm.
      }
      \label{fig:gap_flip}
    \end{figure}

    \emph{The case $s\in\{1,\ldots,k-1\}$.}
    This case is more involved.
    The definition of $s$
    yields a bound on $\mu_{\Lambda^{k-1}}[\calC(\calL_{\geq F_\ell},2^{k-\ell-1}m)]$
    for \emph{both} $\ell = s-1$ and $\ell = s$,
    and those bounds may be combined to give
    \begin{equation}
        \label{eq:deloc_diff_ineq_start}
                \mu_{\Lambda^{k-1}}[
            \calC(\calL_{\geq F_{s-1}},2^{k-s}m)
            \setminus
            \calC(\calL_{\geq F_{s}},2^{k-s-1}m)
        ]
        \geq \epsilon (\cdicho/2)(1-\epsilon)^{s-1}.
    \end{equation}
    We now claim that also
    \[
        \mu_{\Lambda^{k-1}}[
            \calA(\calE_{\leq F_{s-3}},2^{k-s-1}m)
        ]
        \geq  \epsilon (\cdicho/2)(1-\epsilon)^{s-1}.
    \]
    This follows from a flipping argument, see Figure~\ref{fig:gap_flip}.
    To see that the claim is true, notice first that
    $\calC(\calL_{\geq F_{s}},2^{k-s-1}m)^\complement=\calA(\calE_{\geq F_{s}},2^{k-s-1}m)$.
    Now go back to Equation~\eqref{eq:deloc_diff_ineq_start}.
    Conditional on the event $\calC(\calL_{\geq F_{s-1}},2^{k-s}m)$,
    we may explore the outermost $\calL_{F_{s-1}}$-circuit from the boundary.
    Conditional on the circuit (which surrounds $\Lambda_{2^{k-s}m,2^{k-s}m}$),
    the law of $h$ is flip-symmetric around the height $F_{s-1}$.
    In particular, conditional on this circuit,
    the events
    \[
        \calA(\calE_{\geq F_{s}},2^{k-s-1}m)
        \qquad\text{and}\qquad
        \calA(\calE_{\leq F_{s-3}},2^{k-s-1}m)
    \]
    have the same probability (notice that $F_{s-1}$ is the average of $F_{s}$ and $F_{s-3}$).
    This yields the desired claim.

    The claim also implies that
    \begin{equation}
        \mu_{\Lambda^{k-1}}[
            \calA(\calE,2^{k-s-1}m)
        ]
        \geq  \epsilon (\cdicho/2)(1-\epsilon)^{s-1};
        \qquad
        \calE:=\calE_{\leq F_{s-3}} \cup \calE_{\geq - F_{s-3}},
    \end{equation}
    since adding edges to the percolation makes it easier to make the arm event
    occur.
    Moreover, the percolation $\calE$ is an increasing function
    of $|h|$, since
    it contains the edges where the interpolated height functions $|h|$ remains strictly above $F_{s-3}$
    on the entire unit line segment.

    \begin{figure}
      \centering
      \includegraphics{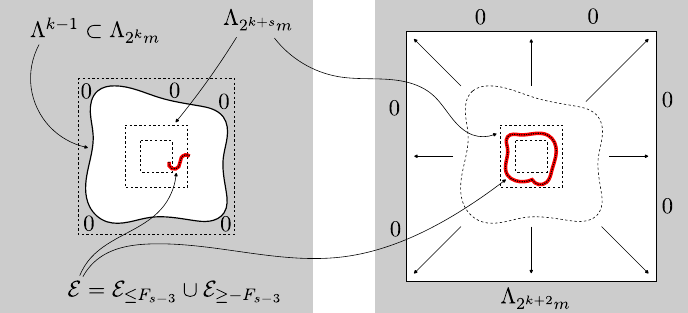}
      \caption{The percolation $\calE$ is increasing in $|h|$,
      and therefore we may push the boundary outwards to apply the RSW theory (Theorem~\ref{thm:generic_RSW}) to
      turn the lower bound on the arm probability into a lower bound on the circuit probability.}
      \label{fig:deloc_RSW}
    \end{figure}

    By Theorem~\ref{thm:generic_RSW}, Statement~2,
    there exists a constant $N\in\Z_{\geq 1}$ such that
    \begin{multline}
        \label{eq:deloc_RSW_application}
        \mu_{\Lambda_{2^{k+2}m}}[
            \calC(\calE,2^{k-s-1}m)
        ]
        \geq  \psiRSW{N}( \epsilon (\cdicho/2)(1-\epsilon)^{s-1})
        \\
        \geq (\epsilon (\cdicho/2N)(1-\epsilon)^{s-1})^N,
    \end{multline}
    where the second inequality follows from the properties of $\psiRSW{N}$.
    See Figure~\ref{fig:deloc_RSW} for an illustration.
    Finally, we claim that
    \begin{align}
        \label{eq:deloc_final_conditional_bounds_1}
               & \mu_{\Lambda_{2^{k+2}m}}\big[
            \alpha(h)^2\big|
            \calC(\calE,2^{k-s-1}m)
        \big]
        \geq v_{k-s-1} + (1+F_{s-3})^2
        \\
        \label{eq:deloc_final_conditional_bounds_2}
               &         \mu_{\Lambda_{2^{k+2}m}}\big[
            \alpha(h)^2\big|
            \calC(\calE,2^{k-s-1}m)^\complement
        \big]
        \geq v_{k-s-1}.
    \end{align}
    Suppose that the claim is true.
    Jointly with Equation~\eqref{eq:deloc_RSW_application},
    this implies that
    \[
        v_{k+2}
        \geq
        v_{k-s-1}
        +
        (\epsilon (\cdicho/2N)(1-\epsilon)^{s-1})^N
        (1+F_{s-3})^2.
    \]
    Now fix $\epsilon\in(0,1)$ such that $(1-\epsilon)^N\phi= 1$
    (since $N$ is a universal constant, so is $\epsilon$).
    The second term in the above formula may then be lower bounded
    by $\tilde c \phi^s$, where $\tilde c>0$ is another universal constant
    (chosen as a function of $\cdicho$, $N$, and $\epsilon$).
    This concludes the proof for the case that $s<\infty$,
    and thus also the proof of the assertion and the lemma.

    It suffices to prove the claim (Equations~\eqref{eq:deloc_final_conditional_bounds_1} and~\eqref{eq:deloc_final_conditional_bounds_2}).
    Equation~\eqref{eq:deloc_final_conditional_bounds_2} follows immediately from Lemma~\ref{lemma:monotonicity_abs_conditioned}.
    We can apply this lemma because the conditioning event is measurable with respect to $|h|$
    (unlike for the case $k=2$ discussed above).
    For Equation~\eqref{eq:deloc_final_conditional_bounds_1},
    we first explore the outermost $\calE$-circuit from the boundary of $\Lambda_{2^{k-s}m}$ (in the conditional event).
    Let $\Lambda$ denote the domain surrounded by this circuit,
    and let $\xi:=h|_{\partial\Lambda}$ denote the induced boundary condition.
    Then
    \[
        \mu_{\Lambda_{2^{k+2}m}}\big[
            \alpha(h)^2\big|
            \calC(\calE,2^{k-s-1}m)
        \big]
        =
        \mu_{\Lambda_{2^{k+2}m}}\big[
            \mu_{\Lambda}^\xi[\alpha(h)^2]
            \big|
            \calC(\calE,2^{k-s-1}m)
        \big].
    \]
    It suffices to prove that, for almost every $(\Lambda,\xi)$,
    \[
        \mu_{\Lambda}^\xi[\alpha(h)^2]
        \geq
        v_{k-s-1} + (1+F_{s-3})^2.
    \]
    Notice that almost surely $\xi\geq 1+F_{s-3}$ or $\xi\leq -1-F_{s-3}$.
    The two cases are identical up to a global sign flip;
    we focus on the first case.
    We claim that
    \[
        \mu_{\Lambda}^\xi[\alpha(h)^2]
        \geq
        \mu_{\Lambda}^{1+F_{s-3}}[\alpha(h)^2]
        =
        (1+F_{s-3})^2
        +
        \mu_{\Lambda}[\alpha(h)^2]
        \geq
        (1+F_{s-3})^2
        +
        v_{k-s-1}.
    \]
    The first inequality is monotonicity of absolute heights (Lemma~\ref{lemma:monotonicity_abs}),
    the equality follows from the flip-symmetry of the model with boundary condition $1+F_{s-3}$
    (Theorem~\ref{theorem:interpolation_properties}),
    and the last inequality follows from monotonicity of the variance in the domain (Lemma~\ref{lemma:cov}).
    In the very last step we use that almost surely $\Lambda_{2^{k-s-1}m}\subset\Lambda$.
    This concludes the proof.
\end{proof}

\begin{remark}
    \label{rem:deloc_extension}
    The $\Lambda\uparrow\C$ limits in Lemma~\ref{lemma:cov}
    are also well-defined when $V\in\PsiPositive=\PsiDeloc$.
    In that case, covariances between heights tend to infinity,
    and covariances between signs tend to one.
    This means that we may extend Equations~\eqref{eq:norm1} and~\eqref{eq:norm2}
    to $V\in\PsiDeloc$ by setting the ``norm'' $\|\blank\|_V$ to be identically zero.
\end{remark}

\section{Finite-size criterion and continuity (Theorems~\ref{thm:deloc_at_crit} and~\ref{thm:local_exp})}
\label{sec:continuity444}

We first establish the following lemma.

\begin{lemma}
    \label{lemma:continuity_observable_PSI}
    For any $n\in\Z_{\geq 1}$,
    the map $p_n|_\Psi$ is continuous.
    In particular, the alternative correlation length $\xi'|_\Psi$ defined in Equation~\eqref{eq:alternative_correlation_length} is upper semicontinuous.
\end{lemma}

\begin{proof}
    Define a distance between continuum domains via
    \[
        \operatorname{Distance}(\Lambda_1,\Lambda_2)
        :=
        d_{\operatorname{Hausdorff}}(\C\setminus \Lambda_1, \C\setminus \Lambda_2).
    \]
    Let $D_n:=\{\Lambda:\text{$\Lambda$ is a continuum domain with $\Lambda_n\subset\Lambda\subset\Lambda_{2n}$}\}$.
    Notice that $D_n$ is a compact metric space under the above distance metric.
    It is easy to see that the map
    \[
        f_n:
        D_n \times \Psi \to [0,1]
        ,\,
        (\Lambda,V)
        \mapsto
        \text{the probability $
        \mu_{\Lambda}
        [\CIRCUIT(\calL_1,n)]$ for the potential $V$}
    \]
    is continuous.
    As $D_n$ is compact,
    the map $p_n|_\Psi=\sup_{\Lambda\in D_n} f_n(\Lambda,\blank)$ is also continuous.
\end{proof}

\begin{proof}[Proof of Theorems~\ref{thm:deloc_at_crit} and~\ref{thm:local_exp} over $\Psi$]
    Fix $V\in\PsiLoc=\PsiDecay$.
    Then the set
    \begin{equation}
      \{V'\in\Psi:\xi'(V')\leq\xi'(V)\}
    \end{equation}
    is an open subset of $\Psi$.
    First, this implies that the height function is delocalised in a $\Psi$-neighbourhood of $V$
    (Theorem~\ref{thm:deloc_at_crit}).
    Second, Remark~\ref{remark:correlation_lengths_comparison} implies that $\xi$ is also uniformly bounded on a $\Psi$-neighbourhood of $V$ (Theorem~\ref{thm:local_exp}).
\end{proof}

\part{Proof of the interface coarse-graining inequality}
\label{part:lemma_proof}

This part is structured as follows:
Section~\ref{sec:statement_of_ingredients} states the three proof ingredients
used in the proof of Lemma~\ref{lemma:preface_second_coarse_graining}.
Those three ingredients are established in Section~\ref{sec:ingred_I_symmetry}--\ref{sec:push}.
The proof of Lemma~\ref{lemma:preface_second_coarse_graining} is contained
in Section~\ref{sec:cg2proof} (but can be read independently of Sections~\ref{sec:ingred_I_symmetry}--\ref{sec:push}).

\section{Statements of Ingredients~I--III}
\label{sec:statement_of_ingredients}

All steps in the proof of the interface coarse-graining inequality must be consistent
with the two phases.
The key feature shared by the two phases,
is the feature that with a good probability the ``total length''
of the interface $\calI_{01}$ is ``close to be as small as possible''
(we call this feature \emph{interface minimisation}).

\begin{lemma}[Symmetric domain lemma, cf.~Figure~\ref{fig:symm}]
    \label{lemma:symdom_new}
    Consider the following setup:
    $S\subset\R^2$ is a closed square with integer coordinates,
    let $B$ be the strip $S+\R e_1$,
    let $\Lambda\subset\C\cap B$ denote a continuum domain,
    and let $\xi$ be a boundary condition on $\partial\Lambda$ such that $0\leq\xi\leq 1$ and
    such that $\{\xi=1\}$ is disjoint from $S$.
    Then
    \[
        \mu_{\Lambda}^\xi\Big(\{\text{the regional excursion $\calR_0(\partial\Lambda)$ contains a horizontal crossing of $S$}\}\Big)
        \leq \frac12.
    \]
\end{lemma}

The previous lemma is the first lemma exhibiting interface minimisation:
indeed, the shortest possible interface $\calI_{01}$ separating the zeros
and the ones in the boundary condition, is by having a vertical crossing of zeros.

\begin{figure}
    \includegraphics{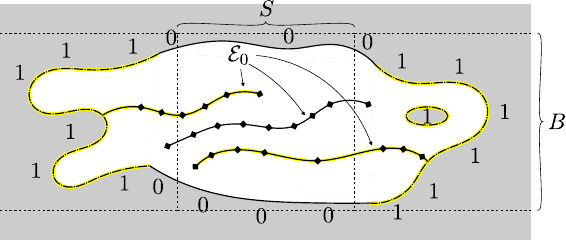}
    \caption{Lemma~\ref{lemma:symdom_new}.
    The excursion $\calE_0$ is drawn,
    and the subset $\calR_0(\partial\Lambda)$ is highlighted in yellow.
    This configuration \emph{does not} realise the event in the lemma:
    the square $S$ is crossed horizontally by a single $\calE_0$-path,
    but this path is not part of the boundary excursion $\calR_0(\partial\Lambda)$.
    }
    \label{fig:symm}
\end{figure}

\newcommand{\cfirstcoarsegrainingnew}{{c_{\operatorname{favour}}}}

\begin{lemma}[Percolation under favourable boundary conditions, cf.~Figure~\ref{fig:lemma:cg1_new}, Left]
    \label{lemma:cg1_new}
    There exists a universal constant $\cfirstcoarsegrainingnew>0$ with the following properties.
    For any $w,h,r\geq 1$, we have
    \[
        \mu_{\Lambda_{w,h}}\left[
            \left\{\parbox{20em}{all connected components of $\calE_0$ which intersect $[-w,w]\times[-r,r]$ have a diameter below $r/1000$}\right\}
        \right]
        \geq (\cfirstcoarsegrainingnew)^{w/r}.
    \]
    Notice that the probability of this event in $\mu_\Lambda$ is a decreasing function of $\Lambda$.
\end{lemma}

\begin{figure}
    \includegraphics{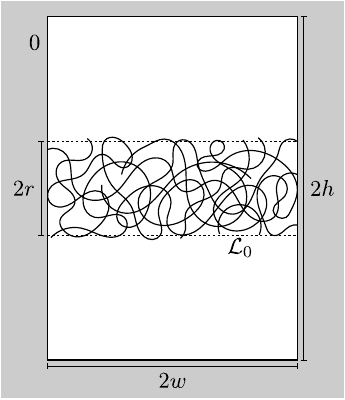}\qquad     \includegraphics{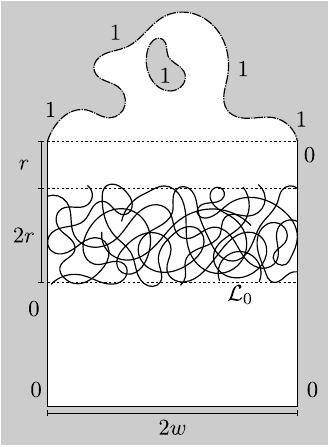}
    \caption{Left: Lemma~\ref{lemma:cg1_new}. The percolation $\calL_0$ is almost omnipresent within the small rectangle with probability at least $(\cfirstcoarsegrainingnew)^{w/r}$
    when boundary conditions are imposed at height $0$.
    Right: Lemma~\ref{lemma:pushing_new}. The boundary conditions are now different: they are at height $0$ below the line $\R\times\{2r\}$
    and at height $1$ above the line.
    The proof of this lemma is much harder because of the unfavourable boundary conditions at the top.}
    \label{fig:lemma:cg1_new}
\end{figure}

The lower bound in this lemma is consistent with the two phases:
in the delocalised phase we expect that it is more or less optimal (for the right constant);
in the localised phase, we expect that the lower bound can be improved
to something of the form $(1-e^{-cr})^{w/r}$.
We shall see in Section~\ref{sec:favourable} that Lemma~\ref{lemma:cg1_new}
relies on basic percolation theory; interface minimisation does not play a role.

The following lemma is a generalisation of the previous one:
it states the same inequality under more general boundary conditions.
More precisely, where only zeros were allowed on the boundary in the previous lemma,
we now also allow ones on the boundary (above the line $\R\times\{2r\}$).

\begin{lemma}[Pushing lemma, cf.~Figure~\ref{fig:lemma:cg1_new}, Right]
    \label{lemma:pushing_new}
    There exists a universal constant $\cpush>0$ with the following properties.
    Fix $w,r\geq 1$
    and let $(\Lambda,\xi)\in\Bound$ denote a boundary condition such that all of the following hold true:
    \begin{itemize}
        \item The domain satisfies $\Lambda\subset [-w,w]\times \R$,
        \item The boundary heights are bounded by $0\leq\xi\leq 1$,
        \item The location of the ones is constrained via $\{\xi=1\}\subset [-w,w]\times [2r,\infty)$.
    \end{itemize}
    Then
    \begin{equation}
            \label{eq:push_new_target}
            \mu_{\Lambda}^\xi\left[
                \left\{\parbox{20em}{all connected components of $\calE_0$ which intersect $[-w,w]\times[-r,r]$ have a diameter below $r/1000$}\right\}
            \right]
            \geq (\cpush)^{w/r}.
    \end{equation}
\end{lemma}

The pushing lemma is relatively hard to prove and requires some Russo--Seymour--Welsh theory
(although in a different way than stated in Theorem~\ref{thm:generic_RSW}).
The above lemma captures a second manifestation of interface minimisation:
indeed, the shortest possible interface $\calI_{01}$ separating the zeros
and the ones in the boundary condition, is an interface that stays close to the top boundary.

The ingredients (and some of the lemmas below) extend to a slightly wider setting via the following remark.

\begin{remark}
    \label{remark:generalisation}
    Let $\Lambda$ denote a domain and $\xi,\zeta$
    two boundary conditions with $\xi+a \leq \zeta$ for some $a\in\Z$.
    Let $E(\calX)$ denote an increasing event defined in terms of the percolation $\calX$.
    Then
    \[
        \mu_\Lambda^\xi[E(\calL_0)] \leq \mu_\Lambda^\zeta[E(\calL_{\leq a})].
    \]
\end{remark}

\begin{proof}
    We get
    \[
        \mu_\Lambda^\xi[E(\calL_0)]
        \leq \mu_\Lambda^\xi[E(\calL_{\leq 0})]
        = \mu_\Lambda^{\xi+a}[E(\calL_{\leq a})]
        \leq \mu_\Lambda^\zeta[E(\calL_{\leq 0})].
    \]
    The first inequality is a union bound,
    the second monotonicity in heights (Lemma~\ref{lemma:monotonicity}).
\end{proof}

\section{Ingredient~I: Analysis of symmetric domains}
\label{sec:ingred_I_symmetry}

We first state a general, abstract version of the symmetric domain lemma (Lemma~\ref{lemma:symdom_new}).
We recommend the reader to have Lemma~\ref{lemma:symdom_new} in mind when reading the statement and proof
of this general version.
We then give a proof of Lemma~\ref{lemma:symdom_new} as a corollary.

The general statement is necessary for the proof of the pushing lemma
in Section~\ref{sec:push}.
The general lemma is stated for the square lattice cable graph $\C$,
but also works for any other planar graphs as is explained later.

\begin{lemma}[cf. Figures~\ref{fig:symm} and~\ref{fig:symm2}]
\label{lemma:symdom_general}
    Let $\Lambda$ denote a continuum domain and $0\leq\xi\leq 1$ a boundary height function,
    and let $\Sigma$ denote an involutive automorphism of the underlying cable graph with the property that
    $\emptyset=\Sigma\Lambda\cap\{\xi=1\} =\Sigma\{\xi=1\}\cap\{\xi=1\}$.

    Define
    \begin{gather}
        \calQ:=\calQ_\Lambda:=\{xy\in\E:\text{$xy$ is not incident to $\calR_0=\calR_0(\partial\Lambda)$}\}.
    \end{gather}
    Then for any bounded increasing function $X$ on $\{0,1\}^\E$,
    we have
    \[
        \mu_{\Lambda}^\xi[X(\calR_0)]
        \leq 
        \mu_{\Lambda}^\xi[X(\Sigma\calQ)].
    \]
    In particular, if $A\subset\{0,1\}^\E$ is an increasing event such that 
     $\{\calR_0\in A\}$ and $\{\Sigma\calQ\in A\}$ are disjoint,
    then $\mu_{\Lambda}^\xi(\{\calR_0\in A\})\leq 1/2$.
\end{lemma}

\begin{figure}
    \includegraphics{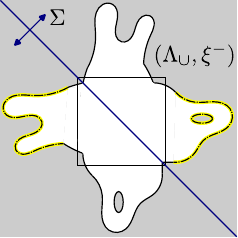}
    \caption{Lemma~\ref{lemma:symdom_new} (Figure~\ref{fig:symm})
    is a corollary of Lemma~\ref{lemma:symdom_general}.
    The figure above illustrates a symmetrised version of Figure~\ref{fig:symm},
    as it appears in the proof of Lemma~\ref{lemma:symdom_general}.
    }
    \label{fig:symm2}
\end{figure}

\begin{proof}
    Let $\Lambda_\cup = \Lambda\cup\Sigma\Lambda$ and define
    \[
        \xi^-:\partial\Lambda_\cup\to\Z,\,x\mapsto\begin{cases}
            \xi(x)&\text{if $x\in\partial\Lambda$,}\\
            0&\text{otherwise;}
        \end{cases}
        \hspace{4em}
        \xi^+:=1-\xi^-.
    \]
    The conditions on the triple $(\Lambda,\xi,\Sigma)$ imply that
    \(
        (\Lambda,\xi)\preceq(\Lambda_\cup,\xi^-)
        \preceq(\Lambda_\cup,\Sigma\xi^+)
    \).

    We claim that
    \begin{multline}
        \mu_{\Lambda}^\xi[X(\calR_0)]
        \leq \mu_{\Lambda_\cup}^{\xi^-}[X(\calR_0)]
        = \mu_{\Lambda_\cup}^{\xi^+}[X(\calR_1)]
        = \mu_{\Lambda_\cup}^{\Sigma\xi^+}[X(\Sigma\calR_1)]
        \\
        \leq
        \mu_{\Lambda_\cup}^{\Sigma\xi^+}[X(\Sigma\calQ)]
        \leq 
        \mu_{\Lambda}^\xi[X(\Sigma\calQ)].
    \end{multline}
    The first inequality is monotonicity in domains,
    using that $(\Lambda,\xi)\preceq(\Lambda_\cup,\xi^-)$.
    The second inequality follows from $\calR_1\subset\calQ$,
    which holds true almost surely since $\calR_0$ connects
    vertices with height $\geq 1$ on the endpoints, while $\calR_1$
    connects vertices with height $\leq 0$ on the endpoints.
    For the final inequality, we use again monotonicity in domains with $(\Lambda,\xi)\preceq(\Lambda_\cup,\Sigma\xi^+)$,
    noting that $\calQ$ is a decreasing function of $\calR_0$.
\end{proof}

\begin{proof}[Proof of Lemma~\ref{lemma:symdom_new}]
    Recall the context of the lemma.
    Choose one of the two diagonals of the square (extended to a full line),
    and let $\Sigma$ denote the reflection map over that diagonal (cf. Figures~\ref{fig:symm} and~\ref{fig:symm2}).
    Let $A\subset\{0,1\}^{\E}$ denote the set of subsets of $\E$
    containing a horizontal crossing of $S$.
    Then, in the context of Lemma~\ref{lemma:symdom_general}, we have
    \begin{gather}
        \{\text{$\calR_0$ contains a horizontal crossing of $S$}\}
        =
        \{\calR_0\in A\};
        \\
        \{\Sigma\calQ\in A\}=\{\text{$\calQ$ contains a vertical crossing of $S$}\}.
    \end{gather}
    Since no vertex is incident to both $\calR_0$ and $\calQ$,
    those percolations must respect the planarity of the plane,
    and therefore the two events are disjoint.
    Lemma~\ref{lemma:symdom_general} therefore yields the desired result.
\end{proof}

\section{Ingredient~II: Percolation with favourable boundary conditions}
\label{sec:favourable}

This section proves Lemma~\ref{lemma:cg1_new}.
The key step is the following proposition.

\begin{proposition}
\label{proposition:fav}
For any $n\in\Z_{\geq 1}$, we have
\(
    \mu_{\Lambda_{20n}}(\calA(\calL_0,n)) \geq 2^{-8000}
\).
\end{proposition}

In percolation theory, one may attempt to prove a \emph{coarse-graining inequality} (cf.~\cite[Lemma~10]{Duminil-CopinTassion_2019_RenormalizationCrossingProbabilities}).
This is a dichotomy where the first scenario corresponds to Proposition~\ref{proposition:fav},
and where the second scenario entails exponential decay of the percolation cluster.
We prove Proposition~\ref{proposition:fav} in two steps:
first we derive this dichotomy, and then we rule out the second scenario
using an argument specific to height functions.
The two steps are formalised in the following two lemmas.

\begin{lemma}[Coarse-graining inequality]
    \label{lemma:cg2_new}
    Suppose that $\mu_{\Lambda_{20m}}(\calA(\calL_0,m)) < 2^{-8000}$
    for some fixed $m\in\Z_{\geq 1}$.
    Then there exists a constant $c>0$ such that for all $n\in\Z_{\geq 1}$,
    \[
        \mu_{\Lambda_{20n}}(\calA(\calL_0,n)),\,\mu_{\Lambda_{20n}}(\calC(\calL_0,n)) \leq \frac1ce^{-cn}.
    \]
\end{lemma}

\begin{lemma}[Duality argument for height functions]
    \label{lemma:duality_new}
    It is impossible to have exponential decay of  $(\mu_{\Lambda_{20n}}(\calA(\calL_0,n)))_n$
    and $(\mu_{\Lambda_{20n}}(\calC(\calL_0,n)))_n$.
\end{lemma}

\begin{definition}
  \label{def:annulus_circuit_general}
    Recall the definitions of $\calC(\calX,n)$ and $\calA(\calX,n)$ from Definition~\ref{def:circuit_observable} and~\ref{def:arms}.
    Write $\calC_u(\calX,n)$ and $\calA_u(\calX,n)$ for the same events, except that the circuit or arm must occur in the topological annulus $([-2n,2n]^2\setminus[-n,n]^2)+u$.
\end{definition}

\begin{proof}[Proof of Lemma~\ref{lemma:cg2_new}]
    Fix $m$ as in the statement of the lemma.
    Consider some value of $n\in\Z_{\geq m}$.
    We first make the following claim:
    let $X\subset 2m\Z^2$ denote any set of vertices
    such that $(\Lambda_{20m}+u)_{u\in X}$ is a disjoint family of
    boxes contained in $\Lambda_{20n}$.
    Then
    \[
        \mu_{\Lambda_{20n}}\big(\cap_{u\in X} \calA_u(\calL_0,m)\big) \leq 2^{-8000|X|}.
    \]
    The claim follows from the FKG inequality and the Markov property, which yield
    \begin{align*}
        \mu_{\Lambda_{20n}}\big(\cap_{u\in X} \calA_u(\calL_0,m)\big)
           & \leq
        \mu_{\Lambda_{20n}}\big(\cap_{u\in X} \calA_u(\calL_0,m)
        \big|
        \cap_{u\in X} \{h|_{\partial\Lambda_{20m}+u}\equiv 0\}
        \big)
        \\
       & =
        \mu_{\Lambda_{20m}}(\calA(\calL_0,m))^{|X|}
        \leq 2^{-8000|X|}.
    \end{align*}

    The rest of the proof is now the same as the proof of~\cite[Lemma~10]{Duminil-CopinTassion_2019_RenormalizationCrossingProbabilities},
    and runs roughly as follows.
    Let $u\in[-10n,10n]^2$ and let $\calU$ denote the $\calL_0\cap[-10n,10n]^2$-connected
    component of $u$.
    Define
    \[
        X_u:=\{v\in 2m\Z^2:\calU\cap([-m,m]^2+v)\neq\emptyset\}.
    \]
    The set $X_u$ is called a \emph{coarse-grained} version of $\calU$.
    By comparing the combinatorial growth of the lattice animals corresponding to $X_u$
    to the bound in the claim, it is easy to see that $|X_u|$ decays exponentially fast
    (uniformly in $n$).
    We then conclude that the diameter of $\calU$ must decay exponentially fast as well (uniformly in $n$).
    By a union bound over the vertices in the box, we derive
    exponential decay for the events in the lemma.
\end{proof}

\begin{proof}[Proof of Lemma~\ref{lemma:duality_new}]
    Assume exponential decay.
    We show that this leads to a contradiction,
    by estimating in two incompatible ways the maximum of the height function.

    \textbf{Upper bound.}
    Recall that $H_{\Lambda_{20n}}$ denotes the Hamiltonian of the height function.
    By standard free energy calculations,
    we know that there exists some constant $C>0$ (depending on the potential function, but not on $n$)
    such that
    \[
        \mu_{\Lambda_{20n}}[H_{\Lambda_{20n}}] \leq Cn^2.
    \]
    Recall that $V$ denotes the potential function.
    Suppose that $V(0)$ without loss of generality (by adding a constant to $V$ is necessary),
    and set $c':=V(1)-V(0)>0$.
    It is then easy to see that
    \[
        H_{\Lambda_{20n}}(h)
        \geq c'\sum_{xy}|h_y-h_x| \geq c' \max_x |h_x|.
    \]
    A Markov inequality then yields that for any $\lambda>0$,
    \begin{equation}
        \label{eq:Hamiltonian_upperbound}
        \mu_{\Lambda_{20n}}(\{\max\nolimits_u |h_u| \geq \lambda n^2\})
        \leq \frac{C}{c'}\frac1\lambda.
    \end{equation}

    \emph{The estimate in Equation~\eqref{eq:Hamiltonian_upperbound} is far from optimal,
    but suffices for our purposes.
    We did not yet use the exponential decay assumption.
    We shall do so when deriving the lower bound, which shall contradict Equation~\eqref{eq:Hamiltonian_upperbound}.}

    \textbf{Lower bound.}
    Fix $n$,
    and let $a\in\Z$ denote the random variable defined as the smallest integer such that
    the event
    \begin{equation}
        \label{eq:firstconnection}
        \calA(\{h\leq a\},n)
    \end{equation}
    occurs, where $\{h\leq a\}$ means the subset of edges of $\E$ such that
    $h$ is at most $a$ at both endpoints.
    Claim that the following event must hold true almost surely:
    \[
        \calA(\calL_a,n)\cup\calC(\calL_a,n).
    \]
    To prove the claim, suppose that $\calC(\calL_a,n)$ does not occur.
    In other words, the event $\calA(\calE_a,n)$ must occur.
    We must prove that $\calA(\calL_a,n)$ occurs as well.
    Consider the path realising the event $\calA(\calE_a,n)$:
    on this path, either $h> a$ or $h< a$ at all vertices.
    But the second case is impossible by definition of $a$,
    which means that the event
    $\calA(\{h>a\},n)$ must occur.
    But by definition of $a$, the event $\calA(\{h\leq a\},n)$ occurs as well,
    which implies the occurrence of the desired event since
    \[
        \calA(\{h\leq a\},n) \cap \calA(\{h>a\},n) \subset \calA(\calL_a,n)
    \]
    via the intermediate value theorem. This proves the claim.

    Notice that $\max_x |h_x|\geq |a| $.
    To prove a contradiction with Equation~\eqref{eq:Hamiltonian_upperbound},
    it suffices to prove that $a$ is typically very large.
    Since $a$ is an integer, it suffices to prove that for any $k$, we have
    \[
        \mu_{\Lambda_{20n}}(\{a=k\})\leq
        \mu_{\Lambda_{20n}}(\calA(\calL_k,n)\cup\calC(\calL_k,n)) \leq \frac2c e^{-cn}.
    \]
    The inequality on the left is the claim proved above.
    We must prove the inequality on the right.
    For $k=0$ this is the exponential decay assumption.
    By monotonicity for absolute heights (Lemma~\ref{lemma:monotonicity_abs}),
    we get
    \begin{multline}
        \mu_{\Lambda_{20n}}(\calA(\calL_k,n)\cup\calC(\calL_k,n))
        =
        \mu_{\Lambda_{20n}}^k(\calA(\calL_0,n)\cup\calC(\calL_0,n))
        \\
        \leq 
        \mu_{\Lambda_{20n}}(\calA(\calL_0,n)\cup\calC(\calL_0,n)).
    \end{multline}
    This yields the desired bound for all $k$.
\end{proof}

We have now completed the proof of Proposition~\ref{proposition:fav}.
To arrive at Lemma~\ref{lemma:cg1_new}, we perform two more steps:
first, we use the RSW theory of~\cite{Kohler-SchindlerTassion_2023_CrossingProbabilitiesPlanar} as a black box to turn our arm estimate
in a circuit estimate,
then, we use the FKG inequality to combine circuit estimates to arrive at the desired bound.

\begin{lemma}
    \label{lemma:RSW_fav}
    There exists a universal constant $c>0$, such that for any $n\in\Z_{\geq 1}$,
    we have
    \[
        \mu_{\Lambda_{8n}}(\calC(\calL_0,n)) \geq c.
    \]
\end{lemma}

\begin{proof}
    Apply Theorem~\ref{thm:generic_RSW}, Statement~3
    with $c:=\psiRSW{N}(2^{-8000})$.
\end{proof}

We are now in a position to prove Lemma~\ref{lemma:cg1_new}.
It is not a ``height functions'' proof; the same proof works for the random-cluster model with cluster weight $q\geq 1$.
It combines the FKG inequality, the Markov property, and the previous lemma.
We recommend skipping the proof on a first read.

\begin{proof}[Proof of Lemma~\ref{lemma:cg1_new}]
    For convenience, let us impose that:
    \begin{itemize}
        \item $r\geq 2^{1000}$ (the case for smaller $r$ is handled a the end),
        \item $h\geq 2^{1000} \cdot r$ (since the probability is decreasing in $h$, we do not lose generality),
        \item $w\geq 2^{1000}\cdot h$ (this case for smaller $w$ is also handled at the end).
    \end{itemize}
    Roughly speaking, the proof runs as follows.
    We apply $O(w/r)$ times Lemma~\ref{lemma:RSW_fav}.
    Notice that $\calC(\calL_0,n)=\calA(\calE_0,n)^c$.
    In other words, every time a circuit event occurs, the excursion percolation $\calE_0$ is constrained further.
    We start on the largest scale, then work our way down to the smallest scale, until $\calE_0$ is so constrained that the desired event in Lemma~\ref{lemma:cg1_new} must clearly occur as well.

    To formalise this, let $\calL_0'$ denote the dual of $\calE_0$ in the full-plane lattice $\Z^2$
    (this means that $\calL_0'$ is the union of the random set $\calL_0$ with the set of dual edges not concerned in the domain under consideration).
    This definition is quite natural, since the law of $\calL_0'$ in $\mu_\Lambda$ is now truly stochastically increasing in the domain $\Lambda$.

    For $n\in\Z_{\geq 0}$, define the ``dyadic'' set
    \begin{equation}
      \label{eq:diadic_set}
      \D_n:= \big([-2w,2w]\times[-2r,2r]\big) \cap \big(2^n \Z^2\big),
    \end{equation}
    and the event
    \begin{equation}
      \label{eq:diadic_event}
      \calW_n := \cap_{u\in \D_n} \calC_u(\calL_0',2^n).
    \end{equation}
    
    The proof now runs as follows.
    \begin{itemize}
        \item First, we observe that at the largest scale $n=\lfloor\log_2(w)\rfloor+8$, we have $\D_n=\{0\}$,
        and the event $\calW_n$ occurs $\mu_{\Lambda_{w,h}}$-almost surely, since the circuit is realised by dual edges outside of the domain which are always open.
        \item Then, we shall show below that for any $n\in\Z_{\geq 0}$, we have
        \begin{equation}
          \label{eq:iterative_bound}
            \mu_{\Lambda_{w,h}}[\calW_n|\calW_{n+1}]\geq c^{|\D_n|},
        \end{equation}
        where $c$ is the constant from Lemma~\ref{lemma:RSW_fav}.
        This enables us to work our way down along the scales.
        \item
        Notice that for $n_0:= \lfloor\log_2(r)\rfloor-100$, we have
        \begin{equation}
          \label{eq:favourable_inclusion}
          \calW_{n_0}\subset
          \left\{\parbox{20em}{all connected components of $\calE_0$ which intersect $[-w,w]\times[-r,r]$ have a diameter below $r/1000$}\right\}=:\calY.
        \end{equation}
        We want to lower bound the probability of $\calW_{n_0}$.
        Iterating Equation~\eqref{eq:iterative_bound} yields
        \begin{equation}
          \label{eq:}
          \mu_{\Lambda_{w,h}}[\calW_{n_0}]
          \geq \exp\Big(\log (c) \cdot \sum_{m=n_0}^{\lfloor\log_2(w)\rfloor+8} |\D_{m}|\Big).
        \end{equation}
        Since $\sum_{m=n_0}^{\lfloor\log_2(w)\rfloor+8} |\D_{m}| \leq 2^{1000}\cdot w/r$, we have
        \begin{equation}
          \label{eq:final_bound}
          \mu_{\Lambda_{w,h}}[\calW_{n_0}] \geq (c^{2^{1000}})^{w/r},
        \end{equation}
        which is the desired bound.
    \end{itemize}

    \begin{figure}
  \centering
  \includegraphics{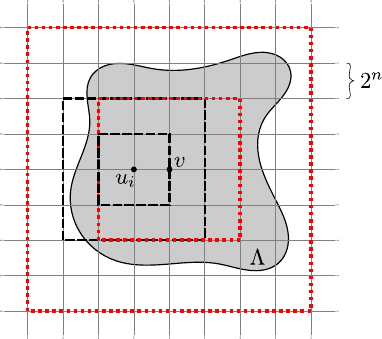}
  \caption{The larger (red) annulus corresponds to the event $\calC_{v}(\calL_0',2^{n+1})$,
    and the smaller (black) annulus corresponds to the event $\calC_{u_i}(\calL_0',2^n)$.
    The figure is to scale; the grid spacing is $2^n$.
    Suppose that the event $\calC_{v}(\calL_0',2^{n+1})$
    is realised by a $\calL_0'$-circuit $\gamma$ in the red annulus (the circuit is represented by the curved solid line).
    To see if $\calC_{u_i}(\calL_0',2^n)$ occurs, it suffices
    to exclude the occurrence of an $\calE_0$-arm within the domain $\Lambda$.}
  \label{fig:favour_annuli_setup}
\end{figure}

    It remains to derive Equation~\eqref{eq:iterative_bound}.
    Fix $n$.
    Let $(u_i)_{1\leq i\leq |\D_n|}$ denote an enumeration of $\D_n$.
    It suffices to prove that for any $i$, we get
    \begin{equation}
      \mu_{\Lambda_{w,h}}\Big[\calC_{u_i}(\calL_0',2^n)\Big|\calW_{n+1}\cap \Big(\bigcap_{j<i} \calC_{u_j}(\calL_0',2^n)\Big)\Big]\geq c.
    \end{equation}
    Fix $i$, and write $E$ for the conditioning event on the right.
    By definition of the dyadic sets, there exists a vertex $v\in\D_{n+1}$ at $\ell^\infty$-distance at most $2^n$ from $u_i$
    (Figure~\ref{fig:favour_annuli_setup}).
    Conditional on the event $\calC_{v}(\calL_0',2^{n+1})\supset E$,
    we can explore the outermost circuit of $\calL_0'$ in this annulus (the red annulus in Figure~\ref{fig:favour_annuli_setup}).
    Write $\nu:=\mu_{\Lambda_{w,h}}[\cdot|E]$;
    let $\Lambda$ denote the random domain enclosed by this circuit,
    and let $h'$ denote the restriction of $h$ to $\Lambda^c$.
    For a given $\Lambda$ and $h'$, let $E_{\Lambda,h'}$ denote the $\mu_{\Lambda}$-event that $h\sim\mu_{\Lambda}$ and $h'$ jointly realise the event $E$.
    Notice that for fixed $(\Lambda,h')$, the event $E_{\Lambda,h'}$ is also an increasing event in terms of $\calL_0$.
    We now claim that (Figure~\ref{fig:favour_annuli_setup})
    \begin{align}
      \mu_{\Lambda_{w,h}}[\calC_{u_i}(\calL_0',2^n)|E]
      &=
        \int \mu_{\Lambda}[\calC_{u_i}(\calL_0',2^n)|E_{\Lambda,h'}] \diffi\nu(\Lambda,h')
       \\& \geq 
        \int \mu_{\Lambda}[\calC_{u_i}(\calL_0',2^n)] \diffi\nu(\Lambda,h')
        \\&\geq \mu_{(\Lambda_{2^{n+3}})+u_i}[\calC_{u_i}(\calL_0',2^n)]
        \geq c.
    \end{align}
    The equality is by the Markov property and the definition of the conditional expectation.
    For the first inequality, observe that $E_{\Lambda,h'}$ is an increasing event in terms of $\calL_0$, so that the FKG inequality applies.
    The second inequality follows by boundary pushing, since $\Lambda$ is almost surely contained in $(\Lambda_{2^{n+3}})+u_i$.
    The final inequality is Lemma~\ref{lemma:RSW_fav}.

    \textbf{Proof for arbitrary $w\geq 1$.}
        \label{page:wadionawiodnawd}
    Fix $w$, and set $\bar w:=\lceil w\rceil \leq 2w$.
    Now for sufficiently large $k\in \Z_{\geq 0}$, the first part of the proof yields
    \[
        \mu_{\Lambda_{(2k+1)\bar w,h}}\left[\calY\right] \geq (c^{2^{1000}})^{(2k+1)\bar w/r}.
    \]
    Now the set $\Lambda_{(2k+1)\bar w,h}$ contains $2k+1$ disjoint copies of $\Lambda_{w,h}$, shifted by $i\cdot 2\bar we_1$ where $i=-k,\ldots,k$.
    Write $\Lambda'\subset\Lambda_{(2k+1)\bar w,h}$ for the union of those copies.
    The boundaries of those copies are not contained in $\Lambda'$.
    By monotonicity in domains and the Markov property,
    \begin{equation*}
        \mu_{\Lambda_{(2k+1)\bar w,h}}\left[\calY\right]
        \leq
        \mu_{\Lambda'}\left[\calY\right]
        \leq \mu_{\Lambda_{w,h}}\left[\calY\right]^{2k+1}.
    \end{equation*}
    In combination with the previous display, this yields the desired bound for arbitrary $w$. 

    \textbf{Proof for arbitrary $r\geq 1$.}
    We already proved the desired bound for $r\geq 2^{1000}$.
    To cover the case of smaller $r$, it suffices to lower bound the probability of the event
    \[
        \left\{\text{there is no $\calE_0$-open edge with an endpoint in $[-w,w]\times[-2^{1000},2^{1000}]$}\right\}.
    \]
    This can be done in two steps.
    First, we apply the argument above with $r=2^{2000}$ to the event $\calY$.
    Then, using similar ideas as in the proof of Equation~\eqref{eq:iterative_bound},
    it is easy to show that conditional on $\calY$, each edge with an endpoint in the strip has a uniformly positive probability of being $\calE_0$-closed, independently of the other edges.
    This yields the desired bound for arbitrary $r$.
\end{proof}

\section{Ingredient~III: Pushing lemma}
\label{sec:push}

Earlier versions of this manuscript (Versions~1 and~2 on arXiv)
contained a proof of the pushing lemma via an adaptation of the RSW theory in~\cite{Karrila_2023_LogarithmicDelocalizationRandom}
to the setting without invariance under $\pi/2$ rotation.
This method seems robust but technically involved.
After those earlier versions appeared online, Karrila~\cite{Karrila_2023_LogarithmicDelocalizationRandom}
found an easier way to prove the pushing lemma that is adaptable to our setting.
In the current (updated) version, an adaptation of Karrila's method is presented in this section in order to streamline the presentation.

\subsection{Statement of the main estimate and proof overview}
Lemma~\ref{lemma:pushing_simpler} below contains the main estimate.
We prove it first.
Subsection~\ref{subsec:proof_lemma_pushing_new} proves that Lemma~\ref{lemma:pushing_simpler} implies the pushing lemma (Lemma~\ref{lemma:pushing_new}),
by recycling some techniques of Section~\ref{sec:favourable}.

\begin{lemma}[cf. Figure~\ref{fig:pushing_simple_result}]
    \label{lemma:pushing_simpler}
    There exists a universal constant $p\in(0,1)$ with the following property.
    Fix $r\in 64\Z_{\geq 1}$.
    Let $\Lambda\subset\R\times[-11r,13r]$ denote a continuum domain,
    and $\xi$ a boundary condition with $0\leq\xi\leq 1$ and
    $\{\xi=1\}\subset\R\times[11r,13r]$.
    Write $\calR_0:=\calR_0(\partial\Lambda)$.
    Then for any $a\in\Z$, we have
    \begin{equation}
        \label{eq:pushing_simpler}
        \mu_{\Lambda}^\xi
        \big[\big\{
            \connectin{[a,a+\tfrac{r}{16}]\times\{r\}}{\calR_0}{\R\times[-r,r]}{\R\times\{-r\}}
        \big\}\big]
        \leq p.
    \end{equation}

    Combining this with the FKG inequality yields
    \begin{equation}
        \label{eq:pushing_simpler_conclusion}
        \mu_{\Lambda}^\xi
        [\{
            \calR_0\subset \R\times[-r,\infty)
        \}]
        \geq (1-p)^{\lceil 16(x_+-x_-) /r\rceil},
    \end{equation}
    where $x_-$ and $x_+$ are the $x$-coordinates of the left- and rightmost points in
    $\Lambda\cap(\Z\times\{r\})$.
\end{lemma}

\begin{figure}
  \centering
  \includegraphics{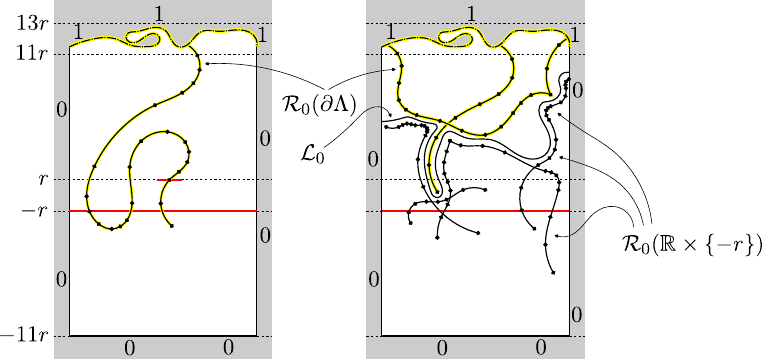}
  \caption{Left: Equation~\eqref{eq:pushing_simpler}
  says that with probability at most $p$,
  the percolation $\calR_0(\partial\Lambda)$ connects
  the short red edge on $\R\times\{r\}$
  to the long red edge on $\R\times\{-r\}$.
  Right: If this event \emph{does not} occur for any of the short red edges
  (Equation~\eqref{eq:pushing_simpler_conclusion}),
  then $\calR_0(\partial\Lambda)$ cannot reach to $\R\times\{-r\}$.
  In that case, $\calR_0(\R\times\{-r\})$ \emph{does not} reach to $\partial\Lambda$,
  and exploring this set induces boundary conditions at height $0$ on the boundary of the explored domain.
  }
  \label{fig:pushing_simple_result}
\end{figure}

To prove the lemma (in Subsections~\ref{subsec:pushing_lower}--\ref{subsec:pushing_upper_zigzag}),
we work in the following context.
Fix $r$, $\Lambda$, $\xi$, and $a$ as in the statement of the lemma.
Write $r_k:=r/2^k$ and $a':=a+r_4$ throughout.
Fix $p\in(0,1)$, and suppose that
\begin{equation}
    \label{eq:contradiction_inequality}
    \mu_{\Lambda}^\xi
    [\calT_{x_0,x_{128}}]
    > p,
\end{equation}
where
\begin{equation}
  \calT_{x,x'}:=\big\{
            \connectin{[x,x']\times\{r\}}{\calR_0}{\R\times[-r,r]}{\R\times\{-r\}}
        \big\},
\end{equation}
and where $x_0:=a$ and $x_{128}:=a'$
(the purpose of the indexing will be clear soon).
In other words, we assume that Equation~\eqref{eq:pushing_simpler} is false (for this value of $p$).

In this context, we shall define (in Subsection~\ref{subsec:pushing_lower}) three events, such that:
\begin{itemize}
    \item The \emph{most likely} one has a probability exceeding $1-127\sqrt[1152]{1-p}$ (Subsection~\ref{subsec:pushing_lower}),
    \item Each event has a probability of at most $1-\cfirstcoarsegrainingnew/2$ (Subsections~\ref{subsec:pushing_upper_straight}--\ref{subsec:pushing_upper_zigzag}).
\end{itemize}
The contradiction arises when $1-127\sqrt[1152]{1-p}=1-\cfirstcoarsegrainingnew/2$,
which implies Lemma~\ref{lemma:pushing_simpler} for this value of $p$.

\subsection{Lower bound on alternating crossing patterns}
\label{subsec:pushing_lower}

\begin{lemma}
    \label{lemma:pushing_step_a}
    Consider the context of Equation~\eqref{eq:contradiction_inequality}.
    If $1-\sqrt[384]{1-p}>1/2$, then we may find integers $ a=x_0\leq x_1\leq \cdots\leq x_{127}\leq x_{128}= a'$ such that
    \begin{equation}
        \label{eq:push_simpler}
        \min_{i=0,\ldots,127}
        \mu_{\Lambda}^\xi[\calV_i(\calR_0)]> 1-\sqrt[384]{1-p},
    \end{equation}
    where
    \begin{equation}
        \calV_i(\calR_0):=\big\{
            \connectin{[x_i,x_{i+1}]\times\{r\}}{\calR_0}{\R\times[-r,r]}{[x_i,x_{i+1}]\times\{-r\}}
               \big\}.
    \end{equation}
\end{lemma}

\begin{figure}
  \centering
  \includegraphics{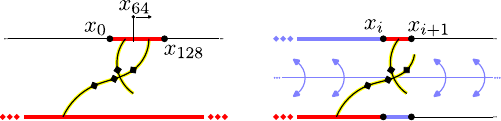}
  \caption{Proof of Lemma~\ref{lemma:pushing_step_a}. \textsc{Left}: By sliding $x_{64}$ from left to right,
  we may find a value at which the square root trick yields a lower bound for \emph{both}
  connection events.}
  \label{fig:pushing_step_a}
\end{figure}

\begin{proof}[Proof (cf.~Figure~\ref{fig:pushing_step_a})]
By the square root trick, we have, for any integer $x_0\leq x_{64}< x_{128}$,
\[
    \max\{\mu_{\Lambda}^\xi[\calT_{x_0,x_{64}}], \mu_{\Lambda}^\xi[\calT_{x_{64}+1,x_{128}}]\}> 1-\sqrt{1-p}.
\]
Let $x_{64}$ be the smallest integer such that the \emph{first} probability is larger
(Figure~\ref{fig:pushing_step_a}, \textsc{Left}).
Then
\[
    \min\{\mu_{\Lambda}^\xi[\calT_{x_0,x_{64}}], \mu_{\Lambda}^\xi[\calT_{x_{64},x_{128}}]\}> 1-\sqrt{1-p}.
\]
By repeating this trick in the subintervals,
we may find integers $x_0\leq x_1\leq x_2\leq \ldots\leq x_{127}\leq x_{128}$ such that
\begin{equation}
    \label{eq:push_simpler_1}
    \min_{i=0,\ldots,127}
    \mu_{\Lambda}^\xi[\calT_{x_i,x_{i+1}}]> 1-\sqrt[128]{1-p}.
\end{equation}

Now fix $i$.
If $\calT_{x_i,x_{i+1}}$ occurs, then one of the following three increasing events must also occur:
the event $\calV_i(\calR_0)$,
or one of the following two events:
\begin{gather}
  C_-:=
  \big\{
    \connectin{[x_i,x_{i+1}]\times\{r\}}{\calR_0}{\R\times[-r,r]}{(-\infty,x_i)\times\{-r\}} 
  \big\};
  \\
    C_+:=
  \big\{
    \connectin{[x_i,x_{i+1}]\times\{r\}}{\calR_0}{\R\times[-r,r]}{(x_{i+1},\infty)\times\{-r\}} 
  \big\}.
\end{gather}
Indeed, the percolation must hit the line $\R\times\{-r\}$
on the interval $[x_i,x_{i+1}]$ (corresponding to $\calV_i(\calR_0)$), or the left of this interval (corresponding to $C_-$), or on the right (corresponding to $C_+$).
By the square root trick, we have
\begin{equation}
  \max\{\mu_\Lambda^\xi[\calV_i(\calR_0)], \mu_\Lambda^\xi[C_-], \mu_\Lambda^\xi[C_+]\}> 1-\sqrt[384]{1-p}.
\end{equation}
To prove the lemma, it suffices to prove that $\mu_\Lambda^\xi[C_\pm]\leq 1/2$,
so that the maximum must necessarily come from $\calV_i(\calR_0)$.
This follows from Lemma~\ref{lemma:symdom_general}
applied with flip symmetry over the $x$-axis ($\Sigma:(x,y)\mapsto(x,-y)$).
Consider, for example, the event $C_-$;
see Figure~\ref{fig:pushing_step_a}, \textsc{Right}.
The event is disjoint from the event where $\calQ$ realises the same event but with the flipped
target zones (light blue), by planarity and the definition of $\calQ$.
Since the event involving $\calQ$ has the higher probability of the two events,
we deduce that $\mu_\Lambda^\xi[C_-]\leq 1/2$.
\end{proof}

\begin{lemma}[Lower bound on alternating crossing patterns]
  \label{lem:alternating_crossing}
  Consider the context of Equation~\eqref{eq:contradiction_inequality},
  and suppose that $1-\sqrt[384]{1-p}>1/2$.
  Recall the definition of the percolation $\calQ$ from Lemma~\ref{lemma:symdom_general}.
  Then we may find integers $a=x_0\leq x_1\leq \cdots\leq x_{127}\leq x_{128}= a'$ such that (at least) one of the following three alternatives holds true.
  \begin{itemize}
    \item \textbf{Straight crossings.} We have
    \begin{equation}
        \label{lem:alternating_crossing:straight}
        \mu_{\Lambda}^\xi[\calS_0(\calQ)\cap \calV_1(\calR_0)\cap\calS_2(\calQ)\cap\calV_3(\calR_0)\cap\cdots\cap\calS_{126}(\calQ)]
        > 1-127\sqrt[1152]{1-p},
    \end{equation}
    where \[\calS_i(\calX):=
    \big\{
        \connectin{[x_i,x_{i+1}]\times\{r\}}{\calX}{[a-r_1,a'+r_1]\times[-r,r]}{[x_i,x_{i+1}]\times\{-r\}}
        \big\}.
    \]
    \item \textbf{Zigzag crossings (left).}
    The same bound holds true, with $\calS_i(\calX)$ replaced by
    \[
        \calZ_i(\calX):=
        \big\{
            \connectin{[x_i,x_{i+1}]\times\{-r\}}{\calX}{\R\times[-r,r]}{[x_i,x_{i+1}]\times\{r\},\, \{a-r_3\}\times[-r,r]}    
        \big\}
    \]
    \item \textbf{Zigzag crossings (right).}
    The same bound holds true, with $\calS_i(\calX)$ replaced by
    \[
        \calZ'_i(\calX):=\big\{
            \connectin{[x_i,x_{i+1}]\times\{-r\}}{\calX}{\R\times[-r,r]}{[x_i,x_{i+1}]\times\{r\},\, \{a'+r_3\}\times[-r,r]} 
        \big\}.
    \]
  \end{itemize}
\end{lemma}

\begin{proof}
    Let $i^*$ denote the index of the shortest interval $[x_i,x_{i+1}]$
    (choosing the smallest index in case of ties),
    and let $I=[x_{i^*},x_{i^*+1}]$ denote this interval.
    The proof consists of three steps.
    \begin{enumerate}
        \item The event $\calV_{i^*}(\calR_0)$ is the union of three events,
        leading to another square root trick.
        \item Each of these events can be translated to an event in terms of $\calQ$,
        via Lemma~\ref{lemma:symdom_general}.
        \item Finally, we bound an intersection of events via a union bound.
    \end{enumerate}

    \begin{figure}
      \centering
      \includegraphics{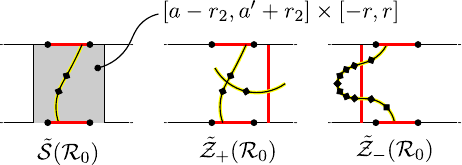}
      \caption{Proof of Lemma~\ref{lem:alternating_crossing}.  
      The vertical crossing must either remain in the shaded rectangle (\textsc{Left}),
      or hit its right side (\textsc{Middle}), or hit its left side (\textsc{Right}).}
      \label{fig:pushing_step_b}
    \end{figure}

    \emph{Step 1 (Figure~\ref{fig:pushing_step_b}).}
    The event $\calV_{i^*}(\calR_0)$ is the union of the following
    three increasing events:
    \begin{gather}
        \tilde\calS(\calR_0):=\big\{
        \connectin{I\times\{r\}}{\calR_0}{[a-r_2,a'+r_2]\times[-r,r]}{I\times\{-r\}}
        \big\};
        \\        \tilde\calZ_+(\calR_0):=\big\{
        \connectin{I\times\{r\}}{\calR_0}{\R\times[-r,r]}{I\times\{-r\},\, \{a'+r_2\}\times[-r,r]}
        \big\};
\\
        \tilde\calZ_-(\calR_0):=\big\{
        \connectin{I\times\{r\}}{\calR_0}{\R\times[-r,r]}{I\times\{-r\},\, \{a-r_2\}\times[-r,r]}
        \big\}.
    \end{gather}
    By another square root trick, we get
    \[
        \max\big\{\mu_\Lambda^\xi[\tilde\calS(\calR_0)], \mu_\Lambda^\xi[\tilde\calZ_+(\calR_0)], \mu_\Lambda^\xi[\tilde\calZ_-(\calR_0)]\big\}> 1-\sqrt[1152]{1-p}.
    \]
    We now prove that at least one of the three alternatives in Lemma~\ref{lem:alternating_crossing} must hold true,
    depending on the event(s) achieving the maximum.

    \begin{itemize}
        \item \textbf{If $\tilde\calS$ is the maximiser,
        then we obtain ``straight crossings''.}
        \emph{Step 2.}
        We claim that for any $i$, we have
        \begin{equation}
            \label{lem:alternating_crossing:straight:sufficient}
            \mu_\Lambda^\xi[\calS_i(\calQ)]\geq \mu_\Lambda^\xi[\tilde\calS(\calR_0)]> 1-\sqrt[1152]{1-p}.
        \end{equation}
        \emph{Step 3.}
        If the claim holds true, then a union bound implies Equation~\eqref{lem:alternating_crossing:straight}.

        It suffices to prove Equation~\eqref{lem:alternating_crossing:straight:sufficient} for fixed $i$.
        Before, we used Lemma~\ref{lemma:symdom_general} with a flip symmetry,
        but we shall now apply it with a rotation symmetry by an angle $\pi$
        around a properly chosen point.
        More precisely,
        let $\Sigma$ denote the rotation symmetry 
        around some half-integer point 
        $u\in(\Z/2)\times\{0\}$ 
        such that
        \begin{equation}
            \label{eq:patterns_rot_def}
            \Sigma (I\times [-r,r])\subset [x_i,x_{i+1}]\times [-r,r].
        \end{equation}
        It is easy to find such a symmetry $u$ because $I$ was the shortest interval.
        Since the rectangle in the definition of $\calS_i$ is much wider than the one in the definition of $\tilde\calS$, we also get
        \begin{equation}
            \Sigma ([a-r_2,a'+r_2]\times [-r,r])\subset 
            [a-r_1,a'+r_1]\times[-r,r].
        \end{equation}
        By definition of $\calS_i$, we then get
        \(
            \tilde\calS(\Sigma\calQ) \subset \calS_i(\calQ)
        \),
        and therefore Lemma~\ref{lemma:symdom_general} yields
        \[
            \mu_\Lambda^\xi[\calS_i(\calQ)]\geq \mu_\Lambda^\xi[\tilde\calS(\Sigma\calQ)]\geq \mu_\Lambda^\xi[\tilde\calS(\calR_0)]> 1-\sqrt[1152]{1-p}.    
        \]
        This proves Equation~\eqref{lem:alternating_crossing:straight:sufficient},
        which implies Equation~\eqref{lem:alternating_crossing:straight}. 

        \item \textbf{If $\tilde\calZ_+$ is the maximiser, then we obtain ``zigzag crossings (left)''.}
        \emph{Step~2.}
        Fix $i$.
        Like before, let $\Sigma$ denote rotation around a half-integer point such
        that~\eqref{eq:patterns_rot_def} holds true.
        This also implies that
        \begin{align}
            \Sigma ([a'+r_2,\infty)\times [-r,r])&\subset 
            (-\infty,a-r_3]\times[-r,r].
        \end{align}
        Then by definition of $\calZ_i$ we have
        \(\tilde\calZ_+(\Sigma\calQ)\subset\calZ_i(\calQ)\).
        Just like before, we use Lemma~\ref{lemma:symdom_general} to obtain
        \(\mu_\Lambda^\xi[\calZ_i(\calQ)]> 1-\sqrt[1152]{1-p}\).
        \emph{Step 3.}
        The desired inequality for ``zigzag crossings (left)'' now follows
        via a union bound.
        \item \textbf{If  $\tilde\calZ_-$ is the maximiser, then we obtain ``zigzag crossings (right)''.} Idem.
        \qedhere
    \end{itemize}
\end{proof}

\subsection{Upper bound on alternating crossings (straight)}
\label{subsec:pushing_upper_straight}

This subsection establishes an \emph{upper bound} on
the probability in Equation~\eqref{lem:alternating_crossing:straight}.

\begin{remark}
    \label{remark:extraconditioningevent}
    Consider the context of Lemma~\ref{lem:alternating_crossing}.
    Consider the regional exploration $\calR_{\leq 1}(Y)$
    started from the set $Y=[a-r,a'+r]\times\{-r,r\}$.
    Then 
    \[
        \mu_\Lambda^\xi[E]\geq \cfirstcoarsegrainingnew;\qquad
        E=\{\text{all conn.\ comp.\ of $\calR_{\leq 1}(Y)$ have a diam.\ below $r_4$}\}
    .\]
\end{remark}

\begin{proof}
    Since $\Lambda\subset \R\times[-11r,13r]$,
    we may apply Lemma~\ref{lemma:cg1_new}
    with $w=13r$ and $h\approx\infty$ to $\Lambda$ rotated by an angle $\pi/2$.
    This lemma may then be applied to the boundary condition $\xi\leq 1$
    and the percolation $\calR_{\leq 1}(\blank)$ via Remark~\ref{remark:generalisation}.
\end{proof}

\begin{lemma}
    \label{lemma:pushing_upper_straight}
In the context of Lemma~\ref{lem:alternating_crossing}, we have
\[
    \mu_{\Lambda}^\xi[\calS_0(\calQ)\cap \calV_1(\calR_0)\cap\calS_2(\calQ)\cap\cdots\cap\calS_{126}(\calQ)]\leq 1-\cfirstcoarsegrainingnew/2.
\]
\end{lemma}

\begin{figure}
  \centering
  \includegraphics{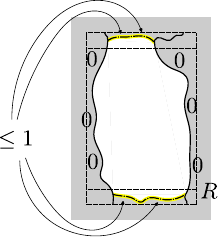}
  \caption{Proof of Lemma~\ref{lemma:pushing_upper_straight}.
  In the rectangle $R$, we explore the left- and rightmost vertical $\calL_0$-crossings,
  as well as connected components of $\calE_{\leq 1}$ intersecting the top and bottom
  (those connected components are small thanks to the conditioning event $E$).
  This induces a new domain (white region), where we may
    then apply the symmetric domain lemma (Lemma~\ref{lemma:symdom_new}).
    It says that the probability of a vertical $\calE_{\leq 0}$-crossing is at most $1/2$.
  }
  \label{fig:pushing_step_c}
\end{figure}

\begin{proof}
    Using the previous remark, it suffices to prove that
    \begin{equation}
      \label{eq:bafiownaowdin}
      \mu_{\Lambda}^\xi[\calS_0(\calQ)\cap \calV_1(\calR_0)\cap\calS_2(\calQ)|E]\leq 1/2.
    \end{equation}
    Recall that the percolations $\calR_0,\calQ\subset\E$ never touch each other (no vertex is incident to both).
    Recall also that the connected components of $\calR_0$ are surrounded by $\calL_0$-interfaces.
    We prove Equation~\eqref{eq:bafiownaowdin} using these properties 
    and Lemma~\ref{lemma:symdom_new}.
    Consider the alternating crossing event in Equation~\eqref{eq:bafiownaowdin}.
    The two vertical $\calQ$-crossings are contained in the rectangle 
    $R:=[a-r_1,a'+r_1]\times[-r,r]$, and
    sandwich the vertical $\calR_0$-crossing.
    In particular, this alternating crossing event is contained in the event $E'$ where
    \[
        E':=\left\{\parbox{22em}{from left to right, $R$ contains a vertical $\calL_0$-crossing,
        then a $\calE_{\leq 0}$-crossing, and then a $\calL_0$-crossing}\right\}.
    \]
    It suffices to prove that $\mu_{\Lambda}^\xi[E'|E]\leq 1/2$.

    This is straightforward;
    the following construction is illustrated by Figure~\ref{fig:pushing_step_c}. Conditional on $E$, we may first explore the left- and rightmost
    vertical $\calL_0$-crossings of $R$,
    and then use the conditioning on $E$ to induce $\calL_{\leq 1}$-boundary conditions
    on the top and bottom.
    For $E'$ to occur, the induced domain must be crossed vertically by $\calE_{\leq 0}$.
    But that event occurs with probability at most $1/2$,
    by Lemma~\ref{lemma:symdom_new} combined with Remark~\ref{remark:generalisation}.
\end{proof}

\subsection{Upper bound on alternating crossings (zigzag)}
\label{subsec:pushing_upper_zigzag}

This subsection establishes upper bounds on
the zigzag crossing events in Lemma~\ref{lem:alternating_crossing}.

\begin{lemma}
    \label{lemma:pushing_upper_zigzag}
In the context of Lemma~\ref{lem:alternating_crossing}, we have
\[
    \mu_{\Lambda}^\xi[\calZ_0(\calQ)\cap \calV_1(\calR_0)\cap\calZ_2(\calQ)\cap\cdots\cap\calZ_{126}(\calQ)]\leq 1-\cfirstcoarsegrainingnew/2.
\]
By symmetry this bound extends to the situation where each $\calZ_i$ is replaced with $\calZ_i'$.
\end{lemma}

\begin{proof}
    By Remark~\ref{remark:extraconditioningevent}, it suffices
    to prove that the conditional probability of the zigzag event in $\mu_\Lambda^\xi[\blank|E]$
    is at most $1/2$ (like in the previous proof).

    For $i\in\{1,3,5,\ldots,125\}$, write $\calE^i$
    for the union of all $(\calE_{\leq 0}\cap (\R\times[-r,r]))$-connected components that
    intersect $[x_i,x_{i+1}]\times\{r\}$.
    Now notice that:
    \begin{itemize}
        \item $\calV_i(\calR_0)\cap\calZ_{126}(\calQ)\subset\calZ_i(\calE^i)$ for any $i\in\{1,3,5,\ldots,125\}$,
        since the $\calQ$-event forces the $\calR_0$-crossing (which is also an $\calE^i$-crossing) to appear on its left,
        \item If $\calZ_{n-1}(\calQ)\cap\calV_{n}(\calR_0)\cap\calZ_{n+1}(\calQ)$ occurs for some $n\in\{1,3,5,\ldots,125\}$, then $\calE^i\cap\calE^j=\emptyset$ are disjoint
        for $i<n<j$.
    \end{itemize}
    Therefore the zigzag event in the statement of the lemma is contained in the event
    \[
        E':=\left(\bigcap_{i\in\{1,5,9,\ldots,125\}}\calZ_i(\calE^i)\right)\cap \left\{
        \text{$\calE^i\cap\calE^j=\emptyset$ for any distinct $i,j\in\{1,5,9,\ldots,125\}$}
        \right\}.
    \]
    It suffices to show that $\mu_{\Lambda}^\xi[E'|E]\leq 1/2$.

    The proof runs as follows.
    First, explore the sets $\calE^i$ for $i\in\{1,9,17,\ldots,121\}$,
    inducing boundary conditions at height $0$ on the left and right of each connected component.
    Next, explore the $\calE_{\leq 1}$-connected components starting from the top and bottom of the rectangle $R:=[a-r_1,a'+r_1]\times[-r,r]$,
    inducing boundary conditions at height $\leq 1$ on the top and bottom.
    These connected components are small thanks to the conditioning on the event $E$.
    This generates $15$ domains between the $\calE^i$-components.
    If the event $E'$ occurs, then each of these domains contains a vertical $\calE_{\leq 0}$-crossing.
    We prove that this probability is upper bounded by $1/2$,
    by distinguishing two cases based on the shape of the explored domains.

    To make this more precise, introduce the new random variable
    \[
        t_i:=\text{the $x$-coordinate of the leftmost point in $\calE^i$ (cf.\ Figure~\ref{fig:step_dii}).}
    \]
    If $E'$ occurs, then we must clearly have
    \[
        t_1<t_9<t_{17}<\cdots<t_{121}\leq a-r_3.
    \]
    
    \begin{figure}
      \centering
      \includegraphics{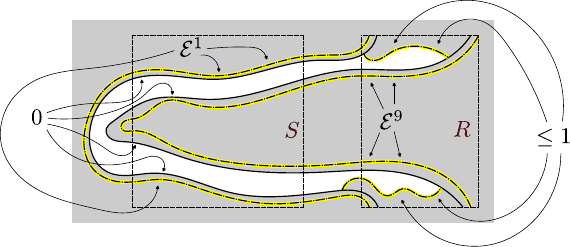}
      \caption{Proof of Lemma~\ref{lemma:pushing_upper_zigzag}, Case (i).
      The percolation $\calE^9$ reaches so far to the left,
        that the unexplored domain traverses the square $S$, enabling the application of Lemma~\ref{lemma:symdom_new} (cf. Remark~\ref{remark:generalisation}).}
      \label{fig:step_di}
    \end{figure}

    \textbf{Case (i): $t_9\leq a-4r$ (Figure~\ref{fig:step_di}).}
    The induced domain between $\calE^1$ and $\calE^9$ is so wide that we may apply Lemma~\ref{lemma:symdom_new} (cf. Remark~\ref{remark:generalisation}) to it directly, yielding that the probability of a vertical $\calE_{\leq 0}$-crossing in this domain (necessary for $\calZ_5(\calE^5)$) is at most $1/2$.

    \begin{figure}
      \centering
      \includegraphics{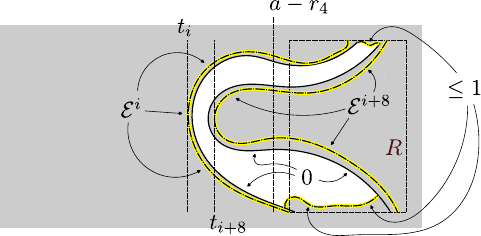}
      \caption{Proof of Lemma~\ref{lemma:pushing_upper_zigzag}, Case (ii).
      By choosing $i$ properly,
      we may apply Lemma~\ref{lemma:swirl} (cf. Remark~\ref{remark:generalisation})
      to get the desired bound $1/2$ on the probability of a vertical $\calE_{\leq 0}$-crossing.}
      \label{fig:step_dii}
    \end{figure}

    \textbf{Case (ii): $t_9> a-4r$ (Figure~\ref{fig:step_dii}).}
    We may find an index $i=9,17,25,\ldots,113$ such that
    \[
        |t_{i+8}-t_i|\leq |t_{i+8}-(a-3r_5)|.
    \]
    In that case, the induced domain between $\calE^{i}$ and $\calE^{i+8}$
    satisfies exactly the constraints of Lemma~\ref{lemma:swirl} below
    (this lemma is a variation of Lemma~\ref{lemma:symdom_new}), which implies that the
    probability of a vertical $\calE_{\leq 0}$-crossing (necessary for $\calZ_{i+4}(\calE^{i+4})$) is at most $1/2$.
\end{proof}

\begin{figure}
    \includegraphics{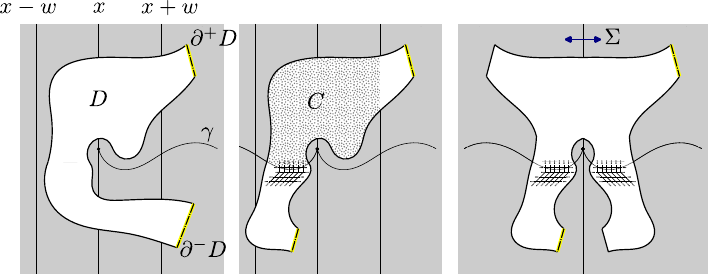}
    \caption{Lemma~\ref{lemma:swirl}. Left: Statement. Others: Parts of the proof.
    }
    \label{fig:swirl}
\end{figure}

The following result is a corollary of Lemma~\ref{lemma:symdom_general}.
It says that in Figure~\ref{fig:swirl}, Left, the probability that $\calE^\cts_0$
connects the two yellow areas, is at most $1/2$.
The statement is technical; this is necessary to rule out pathological cases.

\begin{lemma}[cf. Figure~\ref{fig:swirl}]
    \label{lemma:swirl}
    Let $D\subset\R^2$ be the image of the open unit disk under some homeomorphism from $\R^2$ to itself.
    Let $\partial^\pm D$ denote two disjoint closed arcs of $\partial D$.
    Fix $x\in\Z$ and $w\in\Z_{\geq 0}$, and suppose that all of the following hold true:
    \begin{itemize}
        \item $D\subset [x-w,\infty)\times\R$,
        \item $\partial^+ D\cup \partial^- D\subset [x+w,\infty)\times\R$,
        \item There exists some continuous injective curve $\gamma:[0,\infty)\to [x,\infty)\times\R$,
        such that:
        \begin{itemize}
            \item $\gamma(0)\in\{x\}\times\R$,
            \item $\lim_{t\to\infty}|\gamma(t)|=\infty$,
            \item $\partial^-D$ and $\partial^+ D$ lie below and above $\gamma$ respectively.
        \end{itemize}
    \end{itemize}
    Let $\Lambda:=D\cap \C$ and $\xi(u):=1[u\in \partial^+D\cup\partial^-D]$.
    Then
    \[
        \mu_{\Lambda}^\xi(\{\text{$\calE^\cts_0$ connects $\partial^-D$ to $\partial^+D$}\})
        \leq \frac{1}{2}.
    \]
\end{lemma}

\begin{proof}
    We first give an almost-proof which is correct unless the shape of $D$ is pathological.
    We then correct the mistake in our proof so that it works in full generality.
    Our intuition is guided by Figure~\ref{fig:swirl}, which illustrates a non-pathological case.

    Let $\Sigma$ denote the reflection symmetry over the vertical line $\{x\}\times\R$,
    and write $\bar\gamma:=\gamma\cup\Sigma\gamma$.
    Notice that the union of $\Lambda$ with its reflection $\Sigma\Lambda$ exhibits some geometrical pathologies:
    for example, it is not simply-connected.
    Now embed $\Lambda$ in $\R^2$ in a different way:
    we find some homeomorphism $\phi:\R^2\to\R^2$ which is the identity on
    the area above $\bar\gamma$, and such that $\phi(\Lambda)$ does not contain
    any point below the curve $\gamma$.
    We think of the homeomorphism $\phi$ as ``pushing the space below $\bar\gamma$ to the left'';
    see Figure~\ref{fig:swirl}, Middle.
    One may now apply Lemma~\ref{lemma:symdom_general} to the triple $(\phi(\Lambda),\xi\circ\phi^{-1},\Sigma)$.
    In Figure~\ref{fig:swirl} we see that the union of $\phi(\Lambda)$ with the reflected domain
    remains simply-connected, so that $\{\text{$\calE^\cts_0$ connects $\partial^-D$ to $\partial^+D$}\}$
    is disjoint from the event involving $\calQ$ in Lemma~\ref{lemma:symdom_general}.
    This gives the upper bound $1/2$ in the non-pathological case.
    We stress that Section~\ref{section:lupu} and Lemma~\ref{lemma:symdom_general}
    indeed apply to this setup, since the arguments are not sensitive to the choice of the underlying planar graph.

    The above argument may fail when $\partial^-D$ and $\partial^+ D$ are not contained in the outer boundary of $\phi(D)\cup\Sigma\phi(D)$,
    which may happen when $D$ has pathological ``tentacles'' winding around $D$ itself.
    We now present an adaptation of the above argument which works in full generality.
    Let $A$ denote the intersection of $[x-w,x+w]\times\R$ with the region above $\bar\gamma$,
    and let $C$ denote the connected component of $D\cap A$ adjacent to the connected component of $D\setminus A$ containing $\partial^-D$ (the dotted region in Figure~\ref{fig:swirl}, Middle).
    We may then find a homeomorphism $\phi$ which ``reduces'' all ``tentacles'' (connected components of $D\setminus C$) incident to $C$ to tiny bumps.
    More precisely, we require that the homeomorphism $\phi:\R^2\to\R^2$ satisfies the following:
    \begin{itemize}
        \item $\phi$ acts as the identity on $C$,
        \item $\phi(\partial^-\Lambda)$ lies below $\Sigma\gamma$,
        \item $\phi(\partial^+\Lambda)$ lies above $\gamma$,
        \item $\phi(\partial^\pm\Lambda)$ are contained in the outer boundary of $\phi(D)\cup\Sigma\phi(D)$.
    \end{itemize}
    It is straightforward to see that such a homeomorphism $\phi$ exists.
    The end of the proof is then the same as in the non-pathological case above.
\end{proof}

\subsection{Completing the proof of Lemma~\ref{lemma:pushing_new}}
\label{subsec:proof_lemma_pushing_new}

This subsection is mostly a technical repetition of ideas
in the proof of Lemma~\ref{lemma:cg1_new};
it may be skipped on a first read.

\begin{proof}[Proofs of Lemma~\ref{lemma:pushing_simpler} and Lemmas~\ref{lemma:pushing_new}]
    As explained above,
    Lemmas~\ref{lem:alternating_crossing}--\ref{lemma:pushing_upper_zigzag} imply
    Lemma~\ref{lemma:pushing_simpler}
    where $p<1$ is the solution to the equation
    $1-127\sqrt[1152]{1-p}=1-\cfirstcoarsegrainingnew/2$.

    Now assume the setting of Lemma~\ref{lemma:pushing_new}.
    To prove Equation~\eqref{eq:push_new_target},
    it suffices to find a constant $c>0$ such that
    \begin{equation}
      \label{eq:push_new_target_simplified}
                  \mu_{\Lambda}^\xi\left[
                \left\{\text{$\calR_0(\partial\Lambda)$ does not intersect $[-w,w]\times[-r,r]$}\right\}
            \right]
            \geq c^{w/r}.
    \end{equation}
    Indeed, this inequality implies the lemma with
    $\cpush=c\cdot \cfirstcoarsegrainingnew$,
    by simply exploring $\calR_0(\partial\Lambda)$ (inducing zero boundary conditions),
    and then applying Lemma~\ref{lemma:cg1_new}.
    Our objective is to establish Equation~\eqref{eq:push_new_target_simplified}.
    
    We shall first establish Equation~\eqref{eq:push_new_target_simplified}
    under the additional assumption that
    the domain satisfies $\Lambda\subset [-w,w]\times[-w/1000,w/1000]$
     (that is,
    it is much wider than it is tall).
    We do so in three steps.
    \begin{enumerate}
        \item First, we use Lemma~\ref{lemma:cg1_new} to argue that
        the ones are located in a thin strip, without loss of generality.
        This is an important requirement for the application of Lemma~\ref{lemma:pushing_simpler}.
        \item Then, we iterate Lemma~\ref{lemma:pushing_simpler}
        to constrain $\calR_0(\partial\Lambda)$ more and more.
        This argument handles all scales except for the smallest ones.
        \item Finally, the smallest scales are handled via another argument.
    \end{enumerate}

    \textbf{Step 1.}
    By Lemma~\ref{lemma:cg1_new} (cf. Remark~\ref{remark:generalisation}),
    with probability at least $(\cfirstcoarsegrainingnew)^{10^6 w/r}$,
    the strip $T=\R\times [\lfloor 2r\rfloor,\lceil(\frac{2001}{2000}r)\rceil]$
    is \emph{not} crossed vertically by $\calE_{\leq 1}$.
    Conditional on this event, we may explore $\calE_{\leq 1}$
    starting from the top of $T$, inducing boundary conditions at heights $\leq 1$
    on the exploration boundary.

    Thus, it suffices to find a constant $c'>0$
    such that 
        \begin{equation}
      \label{eq:push_new_target_simplified2}
                  \mu_{\Lambda}^\xi\left[
                \left\{\text{$\calR_0(\partial\Lambda)$ does not intersect $[-w,w]\times[-r,r]$}\right\}
            \right]
            \geq (c')^{w/r}.
    \end{equation}
    for any $(\Lambda,\xi)$
    with $\Lambda\subset [-w,w]\times[-w/1000,w/1000\wedge \lceil(\frac{2001}{2000}r)\rceil]$,
    and such that $\{\xi=1\}\subset\R\times[\lfloor 2r\rfloor,\infty)$.
    Equation~\eqref{eq:push_new_target_simplified} then follows with
    $c=c'\cdot (\cfirstcoarsegrainingnew)^{10^6}$.

    \textbf{Step 2.}
    We claim that for $r'\geq 1000\vee (r/10)$, we have
    \begin{equation}
      \label{eq:eqqinfwqionwqfwq}
       \mu_{\Lambda}^\xi
        \Big[
            \{
            \calR_0\subset \R\times[2r-2r',\infty)
        \}\Big|\{
            \calR_0\subset \R\times[2r-3r',\infty)
        \}
        \Big]
        \leq (1-p)^{1000 w/r'}.
    \end{equation}
    Indeed, if the conditioning event holds true,
    then we may explore $\calE_0$ starting from the line
    $\R\times\{2r-3r'\}$.
    This process finishes before hitting $\partial\Lambda$ and induces
    $0$ boundary conditions.
    We may then apply Lemma~\ref{lemma:pushing_simpler} to the domain above the exploration boundary, to obtain the desired inequality.

    Equation~\eqref{eq:eqqinfwqionwqfwq} can be iterated to obtain
    \begin{equation}
      \label{eq:eqwadawdawdqinfwqionwqfwq}
       \mu_{\Lambda}^\xi
        \Big[
            \{
            \calR_0\subset \R\times[2r-(1000\vee (r/10)),\infty)
        \}
        \Big]
        \leq (c'')^{w/r}.
    \end{equation}
    for some appropriately chosen constant $c''>0$ by iterating Equation~\eqref{eq:eqqinfwqionwqfwq}.

    \textbf{Step 3.}
    Notice that Equation~\eqref{eq:eqwadawdawdqinfwqionwqfwq}
    implies Equation~\eqref{eq:push_new_target_simplified2}
    with $c'=c''$ when $r> 1000$.
    When $r\leq 1000$, we must find a replacement for Lemma~\ref{lemma:pushing_simpler}
    in the above argument to handle the smallest scales.
    This is straightforward, as single edges are closed with probability at least $1/2$ via Remark~\ref{remark:dwaionawionodiwanawd} below.
    Details are left to the reader.

    We have now established Equation~\eqref{eq:push_new_target_simplified}
    under the  assumption that
    $\Lambda\subset [-w,w]\times[-w/1000,w/1000]$,
    for an appropriately chosen $c>0$.
    Finally, one may extend to the case that $w$ is smaller than $r$ by repeating ideas in the proof of Lemma~\ref{lemma:cg1_new} (``Proof for arbitrary $w\geq 1$''
    on Page~\pageref{page:wadionawiodnawd}).
\end{proof}

\begin{remark}
    \label{remark:dwaionawionodiwanawd}
    Consider Lemma~\ref{lemma:symdom_general}.
    Suppose that $\Sigma$ denotes rotation over the midpoint of some edge $xy$.
    Then $\mu_{\Lambda}^\xi[\{xy\in\calR_0\}]\leq 1/2$.
\end{remark}

\begin{proof}
    The events $\{xy\in\calR_0\}$ and $ \{xy\in\calQ\}$
    are disjoint, from which the bound follows.
\end{proof}

\section{Proof of Lemma~\ref{lemma:preface_second_coarse_graining}}
\label{sec:cg2proof}

\begin{proof}[Proof of the interface coarse-graining inequality (Lemma~\ref{lemma:preface_second_coarse_graining})]
    Fix $n$ and $k$.
    Fix $\Lambda_{20kn}\subset\Lambda\subset\Lambda_{40kn}$
    and consider $\mu_{\Lambda}$.
    The proof idea is to show that, ``conditional'' on the event
    $\CIRCUIT(\calL_1,20kn)$ in the definition of $p_{20kn}$,
    with a probability of at least $(\cdicho)^k$,
    exactly $k$ copies of ``the event $p_n$'' occur in an ``independent''
    (geometrically separated) fashion.
    This implies Equation~\eqref{eq:original_second_coarse_graining}.
    The geometrical argument is almost identical to the argument in~\cite{Duminil-CopinSidoraviciusTassion_2017_ContinuityPhaseTransition},
    except that we prove the inequality for general $k$ (and not just $k=2$)
    thanks to a slight optimisation, appearing in the proof of Equation~\eqref{renorm_claim_2},
    Step~2.

    Let us first introduce some notation to formalise the notion of ``$k$ independent copies''.
    More precisely, in this proof only, we shall write 
    \begin{align}
        \calC_0 := \cap_{\ell=0}^{k-1}
        \Big\{
            \text{%
                $([-2n,2n]^2\setminus[-\tfrac{3}{2}n,\tfrac{3}{2}n]^2)+(20\ell n)$
                contains a $\calL_0$-circuit%
            }
        \Big\};
        \\
        \calC_1 := \cap_{\ell=0}^{k-1}
        \Big\{
            \text{%
                $([-\tfrac32n,\tfrac32n]^2\setminus[-n,n]^2)+(20\ell n)$
                contains a $\calL_1$-circuit%
            }
        \Big\}.
    \end{align}
    Claim that there exist universal constants $c_A,c_B>0$ such that
    \begin{gather}
        \label{renorm_claim_1}
        \mu_{\Lambda}(\calC_1|\CIRCUIT(\calL_1,20kn))\geq c_A^k,
        \\
        \label{renorm_claim_2}
        \mu_{\Lambda}(\calC_0|\calC_1)\geq c_B^k,
        \\
        \label{renorm_claim_3}\mu_{\Lambda}( \calC_1|\calC_0)
        \leq p_n^k.
    \end{gather}
    The claim implies the lemma with $\cdicho=c_Ac_B$ because~\eqref{renorm_claim_1} and~\eqref{renorm_claim_2}
    imply that
    \[
        \sup_{\Lambda_{20kn}\subset\Lambda\subset\Lambda_{40kn}}
        \mu_{\Lambda}(\calC_1\cap \calC_0)\geq p_{20kn}\cdot c_A^kc_B^k,
    \]
    while~\eqref{renorm_claim_3} says that
    \[
        \sup_{\Lambda_{20kn}\subset\Lambda\subset\Lambda_{40kn}}
        \mu_{\Lambda}(\calC_1\cap \calC_0)
        \leq
        p_n^k.
    \]

    The proof of Equations~\eqref{renorm_claim_1} and~\eqref{renorm_claim_3} is straightforward
    (and relies on Lemma~\ref{lemma:cg1_new} and the basic properties in Section~\ref{section:lupu});
    Equation~\eqref{renorm_claim_2} is the crux of the proof:
    it is more complicated and relies on interface minimisation
    (Lemmas~\ref{lemma:symdom_new} and~\ref{lemma:pushing_new} are also used in this proof).

    \begin{figure}
        \includegraphics{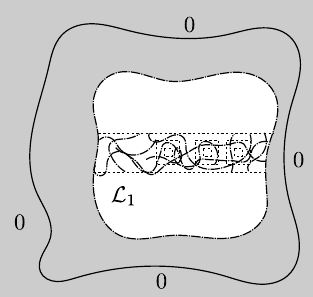}
        \hspace{4em}
        \includegraphics{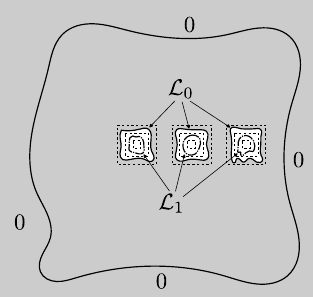}
        \caption{Illustrations not to scale.
            In both illustrations, the white region marks the domain $\Lambda'$
            resulting from the exploration.
            \textsc{Left}: Illustration of Equation~\eqref{renorm_claim_1}.
        The small annuli are crossed when $\calL_1$ percolates well within the rectangle.
        \textsc{Right}: Illustration of Equation~\eqref{renorm_claim_3}. Exploring the outermost $\calL_0$-circuits in each large annulus induces the Markov property.}
        \label{fig:renorm_claim_1}
    \end{figure}

    \begin{proof}[Proof of Equation~\eqref{renorm_claim_1} (cf.~Figure~\ref{fig:renorm_claim_1}, \textsc{Left})]
        The idea is to explore the $\calL_1$-circuit (in the conditioning event) from the outside using
        Lemma~\ref{lemma:exploration},
        inducing boundary heights $1$, so that we may apply Lemma~\ref{lemma:cg1_new}.

        Recall Lemma~\ref{lemma:exploration};
        consider the measure $\mu_{\Lambda}$,
        the starting set $S=\partial\Lambda$,
        and the target set $\calL^\cts_1$.
        Let $(\Lambda',\xi')$ denote the outcome of the exploration.
        This means that $\Lambda':=\Lambda\setminus\bar\calR_1(\partial\Lambda)$
        and $\xi'\equiv 1$.
        From the definition of the event $\CIRCUIT(\calL_1,20kn)$,
        it is easy to see that 
        \[
            \CIRCUIT(\calL_1,20kn) = \{\Lambda_{20kn}\subset\Lambda'\subset\Lambda_{40kn}\}.
        \]
        Thus, Lemma~\ref{lemma:exploration} implies
        \[
            \mu_{\Lambda}(
            \calC_1
            |\CIRCUIT(\calL_1,20kn))
            =
            \mu_{\Lambda}\Big(
            \mu_{\Lambda',1}[\calC_1]
            \Big|\{\Lambda_{20kn}\subset\Lambda'\subset\Lambda_{40kn}\}\Big).
        \]
        It suffices to prove that for fixed such $\Lambda'$,
        we may suitably lower bound
        \(
            \mu_{\Lambda',1}(\calC_1).
        \)
        By choice of $\Lambda'$,
        this quantity is lower bounded by $(\cfirstcoarsegrainingnew)^{2000k}$
        via Lemma~\ref{lemma:cg1_new} (applied at height $1$ rather than $0$).
        Setting $c_A=(\cfirstcoarsegrainingnew)^{2000}$, this yields
         Equation~\eqref{renorm_claim_1}.
    \renewcommand\qedsymbol{}
    \end{proof}

    \begin{proof}[Proof of Equation~\eqref{renorm_claim_3} (cf.~Figure~\ref{fig:renorm_claim_1}, \textsc{Right})]
        The idea is to explore the $k$ distinct $\calL_0$-circuits (in the conditioning event) from the outside,
        then use the Markov property and the definition of $p_n$ for the desired bound.
        Let $(\Lambda',\xi')$ be the outcome of the exploration process from the starting set $\cup_{\ell=0,\ldots,k-1}(\partial\Lambda_{2n}+20\ell n)$ to the stopping set $\calL_0^\cts$ in $\mu_{\Lambda}$.
        Then
        \[
            \xi'\equiv 0;\qquad
            \calC_0 = \{\cup_{\ell=0,\ldots,k-1}(\Lambda_{3n/2}+20\ell n)\subset\Lambda'\}.
        \]
        It suffices to prove that for any such $\Lambda'$,
        we have
        \[
            \mu_{\Lambda'}(\calC_1)\leq p_n^k.
        \]
        For each $\ell=0,\ldots,k-1$,
        the set $\Lambda'$ contains a distinct connected component $K_\ell$
        with
        \[
        \Lambda_{ n}+20\ell n\subset K_\ell \subset \Lambda_{2n}+20\ell n.
        \]
        By the Markov property (Theorem~\ref{theorem:interpolation_properties}) and the definition of $p_n$, we get
        \[
            \mu_{\Lambda'}(\calC_1)
            =
            \prod_{\ell=0}^{k-1} \mu_{K_\ell}(\CIRCUIT_{(20\ell n,0)}(\calL_1,n))
            \leq
            p_n^k.
        \]
        This finishes the proof of Equation~\eqref{renorm_claim_3}.
        \renewcommand\qedsymbol{}
    \end{proof}

    Now that Equations~\eqref{renorm_claim_1} and~\eqref{renorm_claim_3}
    have been established, it suffices to prove Equation~\eqref{renorm_claim_2}.

    \begin{figure}
        \includegraphics{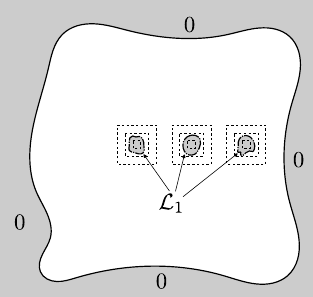}
        \hspace{4em}
        \includegraphics{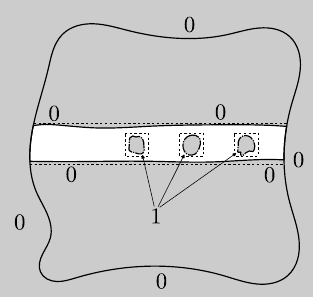}
        \caption{Proof of Equation~\eqref{renorm_claim_2}. \textsc{Left}: Step~1. \textsc{Right}: Step~2.}
        \label{fig:renorm_claim_2a}
    \end{figure}

    \begin{proof}[Proof of~\eqref{renorm_claim_2}]
        \emph{Step~1: Preparation (cf.~Figure~\ref{fig:renorm_claim_2a}, \textsc{Left})}. Explore the $k$ distinct $\calL_1$-circuits from the inside (the starting set $\cup_{\ell=0}^{k-1}(\partial\Lambda_{n}+(20\ell n))$)
        using Lemma~\ref{lemma:exploration},
        inducing a new pair $(\Lambda',\xi')$.
        Notice that $\xi'$ equals $1$ on $\partial\Lambda'\setminus\partial\Lambda$,
        the boundary of the exploration,
        and $0$ on $\partial\Lambda$.
       By definition of $\calC_1$, we have
        \[
            \calC_1 = \big\{\Lambda\setminus\Lambda'\subset\cup_{\ell=0}^{k-1}(\Lambda_{3n/2}+20\ell n)\big\}.
        \]
        For \eqref{renorm_claim_2}
       it suffices to prove that for any such $(\Lambda',\xi')$,
        we have
        \begin{equation}
            \label{eq:renorm_claim_2a}
            \mu_{\Lambda'}^{\xi'}(\calC_0)\geq c_B^k.
        \end{equation}
        This is the end of Step~1.

        Equation~\eqref{renorm_claim_2} would now be easy to prove if it were true that $\xi'\equiv 0$ (using Lemma~\ref{lemma:cg1_new});
        the difficulty stems from the fact that $\xi'$ equals $1$ on the ``hole boundaries''.
        This negative effect is encoded in terms of the boundary excursion set $\calR^\cts_0=\calR^\cts_0(\partial\Lambda')$.
        In Steps 2--4, we gradually constrain the set $\calR^\cts_0$ to a smaller and smaller geometrical region.
        In Step~5, we conclude the proof by showing that,
        once $\calR^\cts_0$ has been sufficiently constrained,
        the event $\calC_0$ occurs with a good probability via an application of Lemma~\ref{lemma:cg1_new}.

        \emph{Step~2: Application of the pushing lemma (cf.~Figure~\ref{fig:renorm_claim_2a}, \textsc{Right})}.
        Consider $\mu_{\Lambda'}^{\xi'}$, and introduce the events
        \begin{gather}
            \calV_-:=\{\text{$[-40kn,40kn]\times[-2n,-7n/4]$ contains a vertical $\calE_0$-crossing}\};
            \\
            \calV_+:=\{\text{$[-40kn,40kn]\times[7n/4,2n]$ contains a vertical $\calE_0$-crossing}\}.
        \end{gather}
        The pushing lemma (Lemma~\ref{lemma:pushing_new})
        and FKG for absolute heights (Lemma~\ref{lemma:monotonicity_abs})
        imply
        \[
            (\cpush)^{1280k} \leq \mu_{\Lambda'}^{\xi'}[\calV_-^c]\cdot\mu_{\Lambda'}^{\xi'}[\calV_+^c]
            \leq \mu_{\Lambda'}^{\xi'}[\calV_-^c\cap\calV_+^c].
        \]
        Now run an exploration started from $\partial\Lambda_{40kn,2n}$ with target set $\calL_0^\cts$.
        Let $(\Lambda'',\xi'')$ be the outcome of this exploration process,
        after removing all connected components of the domain
        that do not intersect $\Lambda_{40kn,7n/4}$.
        Then
        \[
            \calV_-^c\cap\calV_+^c
            =
            \big\{
                \Lambda '' \cap \Lambda_{40kn,7n/4}
                =
                \Lambda ' \cap \Lambda_{40kn,7n/4}
            \big\}.
        \]
        To prove~\eqref{eq:renorm_claim_2a},
        it suffices to obtain $\mu_{\Lambda''}^{\xi''}[\calC_0]\geq (c_B')^k$
        for some universal constant $c_B'>0$, and set $c_B:=(\cpush)^{1280}\cdot c_B'$.
        We are now in the following situation:
        \begin{itemize}
            \item $(\Lambda_{20kn,7n/4}\setminus\cup_{\ell=0}^{k-1}(\Lambda_{3n/2}+(20\ell n,0)))\subset\Lambda''\subset \Lambda_{40kn,2n}$,
            \item $\xi''$ equals zero on the ``outer boundary'' of $\partial\Lambda''$ and one on the ``hole boundaries'',
            \item It suffices to prove $\mu_{\Lambda''}^{\xi''}(\calC_0)\geq (c_B')^k$ for any such $(\Lambda'',\xi'')$.
        \end{itemize}

    \begin{figure}
        \includegraphics{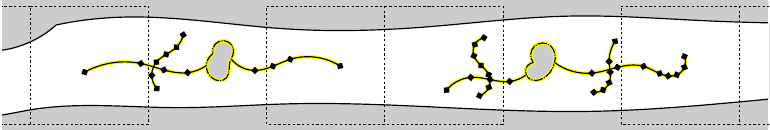}
        \includegraphics{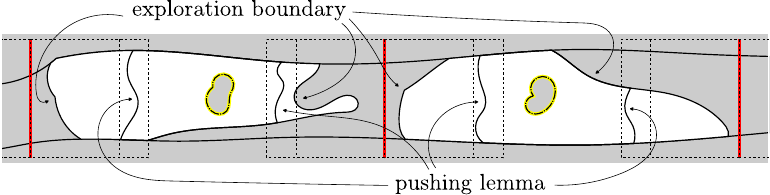}
        \includegraphics{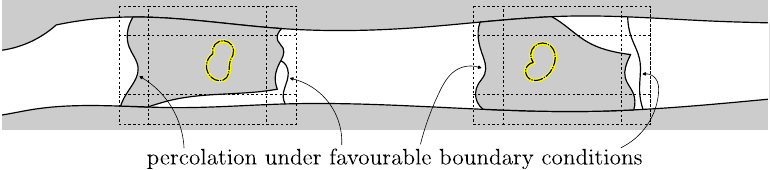}
        \caption{\textsc{Top}: Step~3. \textsc{Middle}: Step~4. \textsc{Bottom}: Step~5.}
        \label{fig:narrow1}
    \end{figure}

        \emph{Step~3: Application of the symmetric domain lemma (cf.~Figure~\ref{fig:narrow1}, \textsc{Top}).}
        Consider the measure $\mu_{\Lambda''}^{\xi''}$ described above.
        The probability that the square $S_{\ell,\pm}:=[-2n,2n]^2 + (20\ell n,0)\pm(5n,0)$
        is \emph{not} crossed horizontally by $\calR_0$ is at least $1/2$,
        by the symmetric domain lemma (Lemma~\ref{lemma:symdom_new}).
        If no such squares is crossed, then $\calR^\cts_0\subset\calQ_{7}$,
        where $\calQ_m:=\cup_{\ell=0}^{k-1}\Lambda_{mn,2n}+(20\ell n,0)$.
        By the FKG lemma for absolute heights (Lemma~\ref{lemma:monotonicity_abs}),
        we get
        \[
            \mu_{\Lambda''}^{\xi''}(\{
                \text{no $S_{\ell,\pm}$ is crossed by $\calR_0$}
                \}
            )
            =
            \mu_{\Lambda''}^{\xi''}(
                \{\calR^\cts_0\subset\calQ_{7}\}
            ) \geq (1/2)^{2k}.
        \]
        This ends Step~3.
        
        The remaining two steps consist in pushing the interfaces closer to the small annuli,
        and making sure that the percolation $\calL_0$ is in the correct position.
        These steps are fairly simple.

        \emph{Step~4: Pushing the interfaces closer to the small annuli (cf.~Figure~\ref{fig:narrow1}, \textsc{Middle}).}
        We claim that for some constant $c_B''>0$, we get
        \[
            \mu_{\Lambda''}^{\xi''}\Big(
                \{\calR^\cts_0\subset\calQ_3\}
                \Big|\{\calR^\cts_0\subset\calQ_{7}\}
            \Big)\geq (c_B'')^k.
        \]
        To obtain this inequality,
        we first start an exploration at $\partial\calQ_{7}$ with target set $\calL_0$.
        If the conditioning event occurs, then this exploration stops before hitting
        the set $\{\xi''=1\}$ (the ``hole boundaries'').
        See the figure for this exploration; the white set is the unexplored set as per usual.
        The induced boundary conditions are now good enough to invoke the pushing lemma (Lemma~\ref{lemma:pushing_new}),
        guaranteeing that the boundary excursions are sufficiently close to the ``holes''.
        This yields Step~4 with $c_B''=(\cpush)^{1280}$.
    
        \emph{Step~5: Final positioning of $\calL_0$ (cf.~Figure~\ref{fig:narrow1}, \textsc{Bottom}).}
        Finally, we claim that for some constant $c_B'''>0$, we get
        \begin{equation}
            \label{eq:final_step_claim111}
            \mu_{\Lambda''}^{\xi''}\Big(
                \calC_0
                \Big|\{\calR^\cts_0\subset\calQ_3\}
            \Big)\geq (c_B''')^k.
        \end{equation}
        Combining Steps~3--5, this yields the desired result (end of Step~2)
        with $c_B':= (1/2)^{2}\cdot c_B''\cdot c_B'''$.

        We now prove Equation~\eqref{eq:final_step_claim111}. We want that $\calL_0$ contains an annulus circuit in each of the $k$ small annuli.
        For each annulus, the bottom and top rectangle are already crossed horizontally because of the boundary conditions at height $0$.
        We must ensure that the left and right rectangle of each annulus are crossed vertically by a $\calL_0$-path connecting the horizontal crossings.
        By exploring the set $\calL_0$ from the \emph{inside}, starting from each hole, we know that we encounter a circuit $\calL_0$
        very quickly (before leaving the set $\calQ_3$).
        However, the problem is now that the this circuit may be \emph{too close} to the hole.
        But the resulting boundary conditions have a boundary height of zero everywhere,
        so that we may simply apply Lemma~\ref{lemma:cg1_new} to ensure that each vertical rectangle is
        crossed with a probability of at least $\cfirstcoarsegrainingnew$.
        This proves Equation~\eqref{eq:final_step_claim111} with $c_B'''=(\cfirstcoarsegrainingnew)^2$.
        \renewcommand{\qedsymbol}{}
    \end{proof}
    This completes the proof of the interface coarse-graining inequality.
\end{proof}

\part{Extension to general potentials}
\label{part:extensions}

\section{Extension to the dual height function of the XY model}
\label{section:extension_poisson_XY}

The square potentials (the class $\Psi$)
are dual to the Villain model.
Above, we used an infinite divisibility property of the square potential to interpolate with a Brownian motion,
which simplified the technical aspects of the proof.
The Poisson potentials $V(a)=-\log I_a(\beta)$
are dual to the XY model.
The Poisson potential has a similar infinite divisibility property, which allows us to interpolate with a continuous-time random walk on $\Z$.
Indeed, recall that $e^{-V(a)}$ is proportional to $\P[\{X-Y=a\}]$,
where $X,Y\sim\operatorname{Poisson}(\beta/2)$.
This means that we can interpolate with a continuous-time random walk on $\Z$,
with a jump rate of $\beta/2$ for going up and a jump rate $\beta/2$
for going down.
The model can thus be defined in any continuum domain $\Lambda\subset\C$
and for any boundary condition $\xi$ on $\partial\Lambda$,
by simply sampling the jumps from a Poisson point process,
and then conditioning that they define a consistent height function with boundary
condition $\xi$.
The entire analysis can be performed in the exact same way,
after the following simple observations.
\begin{itemize}
    \item The intermediate value theorem holds true at integer values,
    since the height function is integer-valued and the Poisson jumps always have
    size one.
    \item The flip symmetry and the Markov property over continuum domains is immediate from the description above.
    \item For any $\beta$ and any continuum domain $\Lambda$,
    the distributions of $h|_O$ and $|h||_O$ satisfy the FKG lattice condition for a
    sufficiently dense observation set $O$, thanks to~\cite{EngelenburgLis_2023_ElementaryProofPhase}
    (which observes that the potentials between the observation points then belong to the class
    described in~\cite{LammersOtt_2024_DelocalisationAbsolutevalueFKGSolidonsolid}).
\end{itemize}

\begin{remark}
The analysis was performed for the square potential first, for two reasons.
First, the Brownian interpolation is well-known in the mathematics literature since the work of Lupu.
Second, the Brownian motion is continuous, while the Poisson interpolation is not.
This leads to the technical nuisance that the height is not unambiguously defined at the jump points,
and we would have to consider two different heights at those points (thought of as left and right limits).
This issue, which is purely technical in nature, is absent in the case of the square potential.
\end{remark}

\section{Extension to all potentials in $\Phi$}
\label{section:extension_general}

We now want to extend to all potentials in $\Phi$,
where a continuous-time interpolation is not (generally) available.
By giving an appropriate alternative definition of the level lines $(\calL_a)_a$,
one can run the entire argument without modifications.

We proceed as follows:
\begin{itemize}
    \item First, we give a description of $(\calL_a)_a$ in the context of $\mu^\dscrt_\Lambda$ (Definition~\ref{def:height_functions}),
    \item Then, we formally define the level lines $(\calL_a)_a$ in the desired generality,
    \item Finally, we derive a few key properties, which enable us to extend the result.
\end{itemize}

\subsection{Informal description of the level lines}
Recall the definition of the height function measure $\mu^\dscrt_\Lambda$ (Definition~\ref{def:height_functions};
in this measure, the integer heights are defined at the vertices, but there is no additional structure).
The potential functions $V$ are convex and symmetric, and therefore we have the following property:
\emph{conditional} on $|h-a|$ (for some $a\in\Z$), the distribution of $h-a$ behaves like
a ferromagnetic Ising model, with explicit coupling constants in terms of $|h-a|$ and $V$.
We may therefore consider the Fortuin--Kasteleyn--Ising coupling of the Ising model
with the random cluster model, and write $\calE_a$ for the set of Fortuin--Kasteleyn--Ising edges.
For each $a\in\Z$, the yields a unique coupling between $h$ and $\calE_a$.
These edges have the following properties:
\begin{itemize}
    \item If $xy\in\calE_a$, then $h_x$ and $h_y$ are on the same side of $a$,
    \item Conditional on the edges $\calE_a$, the distribution of $h-a$ is obtained
    by flipping a fair coin for each cluster not connected to the boundary
    in order to determine its sign.
\end{itemize}
Let $\calL_a$ denote the set of dual edges of $\calE_a$.
The two properties above imply the following.
\begin{itemize}
    \item \textbf{Intermediate value property.}
    If $h-a$ changes sign on an edge $xy$, then $xy\not\in\calE_a$ and therefore $\calL_a$ contains the dual edge of $xy$.
    \item \textbf{Flip symmetry.}
    If $\gamma$ is a circuit in the dual graph surrounding some set of vertices $S$,
    then conditional on $\gamma\subset\calL_a$, the distribution of $(h-a)|_S$ is invariant under a sign flip.
\end{itemize}

We refer to~\cite{LammersOtt_2024_DelocalisationAbsolutevalueFKGSolidonsolid} for a detailed description of this coupling in the context of height functions.
Moreover, it is proved in~\cite{LammersOtt_2024_DelocalisationAbsolutevalueFKGSolidonsolid}
that $\calL_a$ satisfies the FKG inequality in $\tilde\mu^\dscrt_\Lambda$
(the coupling between $\mu^\dscrt_\Lambda$ with the Fortuin--Kasteleyn--Ising edges)
for any $a\in\Z$ and in any domain $\Lambda$.
This is the main reason why we can extend our results to the class $\Phi$.

\subsection{Formal definition of the level lines in the desired generality}
Our arguments rely on exploration processes which reveal, in particular,
whether or not edges belong to $\calL_a$.
We now give a formal definition of the level lines $\calL_a$ in a larger setting,
which enables us to make sense of the conditional law following such an exploration process.

We first introduce the notion of \emph{truncated domains}, which are special cases of continuum domains (Definition~\ref{def:continuum_domains}).
Then, we define the height function measure on such truncated domains, which is a natural extension of $\mu^\dscrt_\Lambda$.
Finally, we define the level lines $\calL_a$ in this setting, and derive the key properties mentioned above.

\begin{figure}[b]
  \centering
  \includegraphics{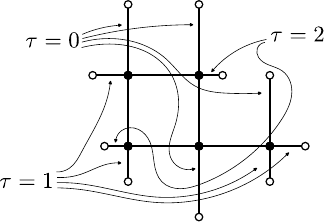}
  \caption{A truncated domain $\Lambda$ with its truncation function $\tau$}
  \label{fig:truncated}
\end{figure}

\begin{definition}[Truncated domains (Figure~\ref{fig:truncated})]
    For any $x,y\in\R^2$, let $L_{xy}$ denote the closed line segment connecting $x$ and $y$ (the convex hull of the two points).
    Consider an arbitrary continuum domain $\Lambda$.
    This continuum domain is called a \emph{truncated domain} if for any square lattice edge $xy\in\E$, the following properties hold true:
    \begin{itemize}
        \item $\Lambda\cap L_{xy}$ contains at most one connected component,
        \item If the intersection is not empty, then either $x\in\Lambda$ or $y\in\Lambda$ (or both),
        \item If the intersection is not empty, then its diameter belongs to $\{1,1/2,1/3,\ldots\}$.
    \end{itemize}

    Now let $\Lambda$ be a truncated domain.
    In this section, we write $\E(\Lambda)$ for the set of unordered pairs of endpoints of line segments of the form $\Lambda\cap L_{xy}$, with $xy\in\E$.
    Remark that this means that $\E(\Lambda)$ is not necessarily a subset of $\E$, due to the truncation.
    The \emph{edge boundary} $\partial_e\Lambda$ of $\Lambda$ is the set of directed edges $xy\in\E(\Lambda)$ such that $y\not\in\Lambda$
    (the edges point outwards).
    For each $xy\in\E(\Lambda)$, we define $\tau_{xy}\in\Z_{\geq 0}$ such that $\|x-y\|_2=1/(\tau_{xy}+1)$.
    This is called the \emph{truncation} of the edge $xy$.
    The definitions imply that $\tau_{xy}=0$ for any edge $xy\in\E(\Lambda)$
    that is not in the edge boundary of $\Lambda$.

    Write $\TBound\subset\Bound$ for the set of boundary conditions $(\Lambda,\xi)$ such that $\Lambda$ is a truncated domain.
    Recall that this means that $\xi:\partial\Lambda\to\Z$, where $\partial\Lambda$ is the topological boundary of $\Lambda$ as a subset of $\C$.
    Write $\TBoundNonneg:=\TBound\cap\BoundNonneg$.
\end{definition}

\begin{definition}[Truncated potentials]
  \label{def:truncated_potential}
  For any $V\in\Phi$ and $\tau\in\Z_{\geq 0}$, let $V_{\tau}\in\Phi$ denote the potential function $V_\tau(a):=V(|a|+\tau)$.
  This is called a \emph{truncated potential function}.
\end{definition}

By going back to the definition of $\Phi$ (Definition~\ref{def:pot}),
it is straightforward to work out that $\Phi$ is indeed preserved under truncation.

From now on, $V$ shall always denote a fixed potential function in $\Phi$.

\begin{definition}[Interpolated height function measure]
  \label{def:interpolated_height_function_measure_general}
    Consider $(\Lambda,\xi)\in\TBound$.
    Then define the \emph{interpolated height function measure} $\mu_{\Lambda}^{\xi}$
    as a probability measure on $(h,\rho)\in\Z^{\Z^2\cap\Lambda}\times[0,\infty)^{\E(\Lambda)}$ with a density proportional to
    \[
    \exp
    -\left[
        \sum_{xy\in\E(\Lambda)}
        \begin{cases}
            \rho_{xy}+
        V_{\tau_{xy}}(h_x-h_y) & \text{if $xy\not\in\partial_e\Lambda$}\\
        \rho_{xy}+V_{\tau_{xy}}(h_x-\xi_y) & \text{if $xy\in\partial_e\Lambda$}
            \end{cases}
    \right].
    \]
    We extend $h$ to $(\Z^2\cap\Lambda)\cup\partial\Lambda$ by setting $h_x=\xi_x$ for $x\in\partial\Lambda$.

    Each edge $xy\in\E(\Lambda)$ has a \emph{residual energy} $\rho_{xy}$, and a \emph{total energy} $T_{xy}$
    defined as the residual energy plus the potential energy, which is $V_{\tau_{xy}}(h_x-h_y)$ for internal edges and $V_{\tau_{xy}}(h_x-\xi_y)$ for boundary edges.
\end{definition}

\begin{remark}
\begin{itemize}
    \item The definition implies immediately that $h$ is independent of $\rho$, and that $\rho$ consists of i.i.d.\ exponential random variables with parameter $1$.
    \item If $\tau\equiv 0$, then all potential functions are equal to $V$, and we recover the original height function measure $\mu^\dscrt_{\Lambda\cap\Z^2}$.
    \item The family $\rho$ contains the independent randomness that we shall use to define the Fortuin--Kasteleyn--Ising edges $\calE_a$ and the level lines $\calL_a$.
    This idea of using residual and total energies goes back to~\cite{Sheffield_2005_RandomSurfaces}.
\end{itemize}
\end{remark}

\begin{definition}[Level lines]
  \label{def:level_lines_general}
    Consider the measure $\mu_\Lambda^\xi$ for some boundary condition $(\Lambda,\xi)\in\TBound$.
    Fix $a\in\Z$.
    Then we define $\calE_a$ as the set of edges $xy\in\E(\Lambda)$ such that
    \begin{equation}
      \label{eq:totalenergycriterion}
      T_{xy} < V_{\tau_{xy}}(|h_x-a|+|h_y-a|).
    \end{equation}
    Write $\calL_a$ for the dual edges of $\calE_a$.
\end{definition}

\begin{remark}[Intermediate value theorem] If $h_x$ and $h_y$ lie on a different side of $a$,
        then
        \[
            V_{\tau_{xy}}(|h_x-a|+|h_y-a|)=V_{\tau_{xy}}(|h_x-h_y|).
        \]
        Since the residual energy is positive,
        Equation~\eqref{eq:totalenergycriterion} cannot hold, and therefore $xy\not\in\calE_a$.
        This guarantees that $\calL_a$ contains the dual edge of $xy$,
        which we interpret as a form of the intermediate value theorem.
\end{remark}

Our exploration processes always take the following form (Figure~\ref{fig:truncated_expl}).
\begin{itemize}
    \item Start from some initial distribution $\mu_\Lambda^\xi$, where $(\Lambda,\xi)\in\TBound$.
    \item Select an edge $xy\in\E(\Lambda)$ with $y\not\in\partial\Lambda$, and reveal $h_x$.
    \item Select a target value $a\in\Z$.
    \item Reveal whether or not $xy\in\calE_a$, and proceed as follows.
    \begin{itemize}
        \item If $xy\in\calE_a$, then $h_y$ is revealed.
        The conditional law of $h$ is then $\mu_{\Lambda'}^{\xi'}$,
        where:
        \begin{itemize}
            \item $\Lambda'\subset\Lambda$ is the largest truncated subdomain
        disjoint from $\{x,y\}$,
        \item $\xi'$ is $\xi$ supplemented with $\xi'_x=h_x$ and $\xi'_y=h_y$,
        and restricted to $\partial\Lambda'$.
        \end{itemize}
        \item If $xy\not\in\calE_a$, then the conditional law of $h$ is given by the lemma below.
    \end{itemize}
\end{itemize}

\begin{figure}
  \centering
  \includegraphics{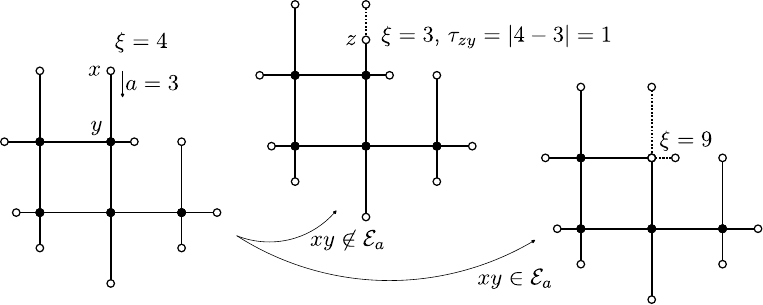}
  \caption{An exploration process revealing whether or not $xy\in\calE_a$.
  If the edge belongs to $\calE_a$, then the two values lie on the same side of $a$, and we choose to reveal both heights.
  If the edges does not belong to $\calE_a$, then we choose not to reveal the height on the other edge endpoint.
  This induces a new law with a domain that is further truncated (Lemma~\ref{lem:conditional_law_after_revealing_level_line_edge}).
  }
  \label{fig:truncated_expl}
\end{figure}

\begin{lemma}[Conditional law after revealing an $\calL_a$-edge]
  \label{lem:conditional_law_after_revealing_level_line_edge}
    Consider the measure $\mu_\Lambda^\xi$ for some boundary condition $(\Lambda,\xi)\in\TBound$.
    Fix $xy\in\E(\Lambda)$ with $y\not\in\partial\Lambda$.
    Fix $b,a\in\Z$.
    Then conditional on the event $\{h_x=b,\,xy\not\in\calE_a\}$, the law of $h$ is given by $\mu_{\Lambda'}^{\xi'}$,
    where $(\Lambda',\xi')\in\TBound$ is chosen such that:
    \begin{itemize}
        \item $\Lambda'\subset\Lambda$ is maximal subject to $x\not\in\Lambda'$ and $\tau'_{zy} = \tau_{xy} + |b-a|$,
        \item $\xi'$ is the same as $\xi$, supplemented with the values $\xi'_x=b$
        and $\xi'_z=a$ ,
        and then restricted to $\partial\Lambda'$.
    \end{itemize}
Here $z\in L_{xy}$ is the new boundary point appearing after the truncation of $xy$.
\end{lemma}

\begin{proof}
    The proof is a straightforward computation.
    Rather than giving the details, we sketch its relation to the
    Fortuin--Kasteleyn--Ising coupling.
    It is easy to determine the conditional law of $h$ given $\{h_x=b\}$;
    the difficulty is in the second step.
    Suppose that $xy\not\in\calE_a$.
    Notice that
    \[
        h_y\mapsto V_{\tau_{xy}}(|b-a|+|h_y-a|)
    \]
    is the only potential function which:
    \begin{itemize}
        \item Equals $V_{\tau_{xy}}(h_y-b)$ when $h_y$ and $b$ lie on different sides of $a$,
        \item Is symmetric around the value $h_y=a$.
    \end{itemize}
    Now the potential function associated with the edge $xy$ in the conditional measure must clearly satisfy the first property
    by the intermediate value theorem.
    Moreover, the flip symmetry property in the Fortuin--Kasteleyn--Ising coupling implies that the conditional potential function must also satisfy the second property.
    This uniquely determines the conditional potential function.
    Notice that $V_{\tau_{xy}}(|b-a|+|h_y-a|)=V_{\tau_{xy}+|b-a|}(h_y-a)$,
    which is exactly the potential function associated with the edge $zy$ in $\mu_{\Lambda'}^{\xi'}$.
\end{proof}

\begin{lemma}
    The consistency, Markov, and flip symmetry properties of Theorem~\ref{theorem:interpolation_properties} generalise to the current setting.
\end{lemma}

\begin{proof}
    Consistency and the Markov property are immediate from the definitions.
    The flip symmetry property follows from the fact that both $V$ and any truncated version $V_\tau$ of it are symmetric functions.
\end{proof}

\subsection{Key properties of the generalised interpolation}

The several manifestations of the FKG inequality play an essential role.
We argue that they generalise to the current setting.
\begin{itemize}
    \item \textbf{Monotonicity for heights (Lemma~\ref{lemma:monotonicity}) and log-concavity (Lemma~\ref{lemma:log_concavity}).}
    These results generalise to the current setting (with truncated domains) simply because every potential function $V_\tau$ is convex and symmetric.
    \item \textbf{Monotonicity for absolute heights (Lemma~\ref{lemma:monotonicity_abs}).}
    Suppose that $\Lambda$ is any truncated domain, $X$ and $Y$ are two increasing functions of $(|h|,\calE_0)$,
    and $\xi'\geq\xi\geq 0$ two boundary conditions,
    then
    \[
        \mu_\Lambda^\xi[XY]\geq \mu_\Lambda^\xi[X]\cdot\mu_\Lambda^\xi[Y]
        \qquad\text{and}\qquad
        \mu_\Lambda^{\xi'}[X]\geq \mu_\Lambda^\xi[X].
    \]
    These results were derived in~\cite{LammersOtt_2024_DelocalisationAbsolutevalueFKGSolidonsolid}.
    \item \textbf{Monotonicity in domains (Lemmas~\ref{lemma:monotonicity_abs_new} and~\ref{lemma:monotonicity_abs_conditioned}).}
    The exact same statement can be shown to hold true, but an extra step is required in the proof.
    Indeed, the original proof still implies that lowering the boundary heights or conditioning vertices to take the value zero can only decrease the distribution of $(|h|,\calE_0)$.
    The only question that remains open, is whether the distribution of $(|h|,\calE_0)$ decreases when we truncate one of the edges in the edge boundary of the domain (where the boundary height $0$ is imposed).
    It is easy to see that increasing the truncation of an edge is equivalent to applying a \emph{decreasing} Radon--Nikodym derivative to the distribution of the adjacent height.
    By the FKG inequality for $(|h|,\calE_0)$, this can only decrease the distribution of $(|h|,\calE_0)$.
\end{itemize}

\subsection{Continuity of the observable}

All arguments rely on percolation theory using the properties described above,
except for Lemma~\ref{lemma:continuity_observable_PSI}.
We generalise that lemma in this last subsection.

\begin{lemma}
    \label{lemma:continuity_observable_GENERAL}
    For any $n\in\Z_{\geq 1}$,
    the map $p_n|_\Phi$ is continuous.
\end{lemma}

\begin{proof}
    Recall the proof of Lemma~\ref{lemma:continuity_observable_PSI},
    and write
    $D_n'$ for the set of \emph{truncated} domains $\Lambda_n\subset\Lambda\subset\Lambda_{2n}$.
    This set can be written generically as a finite union $\cup_{\Delta} D_n'(\Delta)$,
    where
    \[
        D_n'(\Delta):=\{\Lambda\in D_n':\Lambda\cap\Z^2=\Delta\}.
    \]
    The general observable $p_n$ may be written as
     \begin{multline}
        p_n(V):=\sup_\Delta
        \sup_{\Lambda\in D'_n(\Delta)}
        \chi(\Lambda,V);\\
        \chi(\Lambda,V):=
        \text{the probability $\mu_{\Lambda}^0
        [\CIRCUIT(\calL_1,n)]$ for the potential $V$}.
     \end{multline}
     It suffices to prove that $\sup_{\Lambda\in D'_n(\Delta)}
        \chi(\Lambda,\blank)$ is continuous over $\Phi$ for any $\Delta$.
    Fix the values of $n$ and $\Delta$ in the remainder of the proof, and write $T:=D'_n(\Delta)$.

    With $n$ and $\Delta$ fixed, the domain $\Lambda\in T$ is determined in terms of the truncation on the boundary of $\Delta$ (hence the notation $T$).
    It is easy to see that the map $\chi(\Lambda,\blank)$ is continuous over $\Phi$ for any fixed $\Lambda$,
    an we would like to argue that continuity is preserved when we take a supremum over all possible truncations.
    Suppose that the following two properties hold true:
    \begin{itemize}
        \item For any fixed potential $V\in\Phi$, the quantity $\chi(\Lambda,V)$ converges as we send one or multiple truncations to infinity for $\Lambda\in T$,
        \item The limits so obtained (one for each subset of $\partial_e\Lambda$) are continuous in $V$.
    \end{itemize}
    Then $T$ may be viewed as a subset of some compact space $\tilde T$,
    and $\chi$ extends to a continuous function on $\tilde T\times\Phi$.
    The desired result then follows.

    Recall the definition of the topology on $\Phi$ from Definition~\ref{def:pot}; it will now start to play a role.
    We first derive the first property for a fixed potential $V\in\Phi$,
    and then the second property (when $V$ is varied).
    Let $c:=c(V):=\lim_{a\to\infty} V(a)/a$,
    and introduce the potential $W_c(a):=c\cdot |a|$.

    \emph{Convergence for a fixed potential $V$.}
    It is straightforward to see that sending the truncation of an edge to infinity is equivalent to sending the potential of that edge to $W_c$.
    This implies in particular that the limiting measure is well-defined;
    it is simply the original height functions measure but with the potential $V_{\tau}$ on that edge replaced with $W_c$.
    The same holds true if we send multiple truncations to infinity;
    in that case, the potential $W_c$ is imposed at all concerned edges.

    \emph{Continuity of the limits in $V$.}
    It is clear from the definitions that the truncated potential functions associated with each edge are continuous in the choice of $V$.
    For normal edges this follows from the fact that $V(a)-V(0)$ defines a continuous function on $\Phi$ for each $a\in\Z$.
    For edges where the truncation is sent to infinity, this follows from the fact that $c(V)$ defines a continuous function on $\Phi$.
    This yields the desired continuity.
\end{proof}


\addtocontents{toc}{\protect\vspace{1em}}

\subsection*{Acknowledgements}
\addcontentsline{toc}{section}{Acknowledgements}

The author would like to thank Hugo Duminil-Copin
for several directions which turned out to be extremely useful,
as well as for many encouragements during this research.
The author would also like to thank Laurin Köhler-Schindler
for a very enthusiastic discussion on Russo-Seymour-Welsh theory
at the conference celebrating 100 years of Ising model,
and Paul Dario, Aernout van Enter, Alexander Glazman,
Trishen Gunaratnam,
Alex Karrila, 
Romain Panis, and Franco Severo for comments and suggestions regarding this work.
This project has received funding from the European Research Council
(ERC) under the European Union's Horizon 2020 research and innovation programme
(grant agreement No. \texttt{757296}) and the French National Research Agency
(ANR), project number \texttt{ANR-23-CPJ1-0150-01}.

\printbibliography


\end{document}